\documentclass[10pt]{amsart}
\usepackage[english]{babel}
\usepackage{graphicx}
\usepackage{amsmath,amssymb,hyperref}
\usepackage{srcltx}
\usepackage{accents}

\newcommand{\diff}[2]{\mbox{{\rm Diff}{${\,}_{#1}({\mathbb C}^{#2},0)$}}}
\newcommand{\dif}[2]{\mbox{{\rm Diff}{${\,}_{#1}^{*}({\mathbb C}^{#2},0)$}}}

\newcommand{\cn}[1]{\mbox{(${\mathbb C}^{#1},0$)}}
\newcommand{\Xt}{{\mathcal X}_{tp1}\cn{2}}
\newcommand{\Xtg}{{\mathcal X}_{tp1}^{*} \cn{2}}
\newcommand{\Xnt}{{\mathcal X}_{p1} \cn{2}}
\newcommand{\Xntg}{{\mathcal X}_{p1}^{*} \cn{2}}

\newcommand{\pn}[1]{{\mathbb P}^{#1}({\mathbb C})}

\newcommand{\ox}{\'{o}}

\newtheorem{pro}{Proposition}[section]
\newtheorem{teo}{Theorem}[section]

\newtheorem{cor}{Corollary}[section]

\newtheorem{lem}{Lemma}[section]

\theoremstyle{remark}
\newtheorem{rem}{Remark}[section]

\theoremstyle{remark}
\newtheorem{exa}{Example}[section]

\theoremstyle{definition}
\newtheorem{defi}{Definition}[section]

\begin{document}

\title[Rigidity of unfoldings of resonant diffeomorphisms]
{Topological rigidity of unfoldings of resonant diffeomorphisms}
\author{Javier Rib\ox n}
\address{Instituto de Matem\'{a}tica, UFF, Rua M\'{a}rio Santos Braga S/N
Valonguinho, Niter\'{o}i, Rio de Janeiro, Brasil 24020-140}
\thanks{e-mail address: javier@mat.uff.br}
\thanks{MSC-class. Primary: 37F45, 37F75; Secondary: 37G10, 34E10}
\maketitle

\bibliographystyle{plain}
\section*{Abstract}
We prove that a topological homeomorphism conjugating two generic
$1$-parameter unfoldings of $1$-variable complex analytic
resonant diffeomorphisms
is holomorphic or anti-holomorphic by restriction to the unperturbed
parameter.
We provide examples that show that the genericity hypothesis is necessary.
Moreover we characterize the possible behavior of conjugacies for the
unperturbed parameter in the general case.
In particular they are always real analytic outside of the origin.

We describe the structure of the limits of orbits
when we approach the unperturbed parameter.
The proof of the rigidity results is based on the study of the action of a
topological conjugation on the limits of orbits.
\section{Introduction}
We are interested in the study of the topological properties of unfoldings
of tangent to the identity diffeomorphisms.
We define $\diff{p1}{n+1}$ as the set of $n$-parameter unfoldings of local complex
analytic $1$-variable tangent to the identity diffeomorphisms.
An element $\varphi$ of $\diff{p1}{n+1}$ is of the form
\[ \varphi(x_{1},\hdots,x_{n},y)=(x_{1},\hdots,x_{n},f(x_{1},\hdots,x_{n},y)) \]
where $f(0)=0$ and $(\partial f/\partial y)(0)=0$.
We define $\dif{p1}{2}$ as the set of diffeomorphisms whose fixed points set
does not contain $x=0$. We denote by $N(\varphi)$ or $N$ if there is no confusion
the number of points in $\{x=x_{0}\} \cap \mathrm{Fix} (\varphi)$ for any $x_{0} \neq 0$.
Clearly it is a topological invariant.
We prove the following rigidity theorem:
\begin{teo}[Main Theorem]
\label{teo:main}
Let $\varphi, \eta \in \dif{p1}{2}$ with $N>1$ such that
there exists a homeomorphism $\sigma$ satisfying
$\sigma \circ \varphi = \eta \circ \sigma$.
Suppose that either $\varphi_{|x=0}$ or $\eta_{|x=0}$
is non-analytically trivial.
Then $\sigma_{|x=0}$ is
holomorphic or anti-holomorphic.
\end{teo}
We denote by $\diff{1}{}$ the group of local complex analytic $1$-dimensional
diffeomorphisms whose linear part is the identity.
An element $\phi \in \diff{1}{}$ is analytically trivial if
it is embedded in an analytic flow, i.e.
$\phi$ is the exponential of an
analytic singular local vector field
$X=g(y) \partial / \partial y$.
A consequence of the Ecalle-Voronin
analytic classification of tangent to the identity
diffeomorphisms \cite{Ecalle} \cite{V} \cite{mal:ast}
is that elements of $\diff{1}{}$ are generically non-analytically trivial.

In particular if a generic element $\varphi$ of $\diff{p1}{2}$
is topologically conjugated to $\eta \in \diff{p1}{2}$
then $\varphi_{|x=0}$ and $\eta_{|x=0}$ are either holomorphically or
anti-holomorphically conjugated.
The result is far from trivial since a topological class of conjugacy
of a tangent to the identity diffeomorphism in one variable
contains a continuous infinitely dimensional moduli of analytic
classes of conjugacy.

Let us point out that all the topological conjugations in this paper
between elements of $\diff{p1}{2}$ preserve the fibration $dx =0$.
In other words they are of the form
$\sigma(x,y) =(\sigma_{0}(x), \sigma_{1}(x,y))$.
This is a natural hypothesis since we are interested in the topological
classification of unfoldings.

A natural problem is determining the classes of conjugacy of unfoldings
up to topological, formal or analytic equivalence.

The study of the analytic properties of unfoldings is an active field of research.
A natural idea to study an unfolding $\varphi$  in $\diff{p1}{2}$
is comparing the dynamics of $\varphi$ and ${\rm exp}(X)$
where $X = g(x,y) \partial / \partial y$ is a vector field
with $\mathrm{Fix} (\varphi) = \mathrm{Sing} (X)$ whose time $1$ flow ``approximates" $\varphi$.
This point of view has been developed by Glutsyuk  \cite{Gluglu}.
In this way extensions of the Ecalle-Voronin invariants \cite{Ecalle} \cite{V}
to some sectors in the parameter space are obtained. The extensions are
uniquely defined. The sectors of definition have to avoid a finite set of
directions of instability, typically associated (but not exclusively) to small
divisors phenomena.  The rich dynamics of $\varphi$ around the directions
of instability prevents the extension of the Ecalle-Voronin invariants to
be defined in the neighborhood of the instability directions.
Interestingly the study of the dynamics around instability
directions is one of the key elements of the proof
of the Main Theorem.

A different point of view was introduced by Shishikura for codimension $1$
unfoldings \cite{Shishi}. The idea is constructing appropriate fundamental domains bounded
by two curves with common ends at
singular points: one curve is the image of the other one.  Pasting the boundary curves by the
dynamics yields (by quasiconformal surgery) a Riemann surface that is conformally
equivalent to the Riemann sphere. The logarithm of an appropriate affine complex
coordinate on the sphere induces a Fatou coordinate for $\varphi$.
These ideas were generalized to higher codimension unfoldings by Oudkerk \cite{Oudkerk}.
In this approach the first curve is a phase
curve of an appropriate vector field transversal to the real flow of $X$. In both cases
the Fatou coordinates provide Lavaurs vector fields $X^{\varphi}$
such that $\varphi= {\rm exp}(X^{\varphi})$ \cite{Lavaurs}. The Shishikura's approach was
used by Mardesic, Roussarie and Rousseau to provide a complete system of invariants
for  unfoldings of codimension $1$ tangent to the identity diffeomorphisms \cite{MRR}.
Rousseau and Christopher classified the generic unfoldings of codimension $1$
resonant diffeomorphisms \cite{Rou-Chris:mod}.
The analytic classification for the unfoldings of finite codimension
resonant diffeomorphisms was
completed in \cite{JR:mod} by using the Oudkerk's point of view.

We described the formal invariants of elements $\varphi$ of $\diff{p1}{n+1}$
for any $n \in {\mathbb N}$ in \cite{UPD}.
The invariants are divided in two sets, namely those that are analogous to
the $1$-dimensional formal invariants and invariants that are associated
to the position of $\mathrm{Fix} (\varphi)$ with respect to the fibration
$dx_{1}= \hdots = dx_{n} = 0$.

\subsection{Topological classification}
In contrast with the analytic and formal cases there is no topological classification
of unfoldings of tangent to the
identity diffeomorphisms. One of the obstacles is the absence of a complete system
of analytic invariants for elements of $\diff{}{}$. More precisely the problem
is associated with small divisors; it is not known the topological classification of elements
$\phi(z) = \lambda z + O(z^{2}) \in \diff{}{}$ such that $\lambda \in {\mathbb S}^{1}$
is not a root of the unit and $\phi$ is not analytically linearizable.

Let $\varphi = (x,f(x,y)) \in \diff{p1}{2}$.
We denote by $m(\varphi)$ the vanishing order of $f-y$ at the line $x=0$.
We study unfoldings $\varphi$ such that $(N,m)(\varphi) \neq (1,0)$.
The remaining case is trivial and hence uninteresting since the
only topological invariant is the vanishing order of $f(0,y)-y$ at $0$.
From now on we assume $(N,m) \neq (1,0)$.

The situation in absence of small divisors (multi-parabolic case)
has been studied in \cite{rib-mams}.
An element $\varphi(x,y)=(x,f(x,y))$ of $\diff{p1}{2}$ is multi-parabolic
if $(\partial f/\partial y)_{|\mathrm{Fix} (\varphi)} \equiv 1$.
A complete system of topological invariants is presented in \cite{rib-mams} for
the classification of multi-parabolic diffeomorphisms under the assumption that
a conjugation $\sigma$ such that $\sigma \circ \varphi = \eta \circ \sigma$
is of the form $\sigma(x,y)=(x, f(x,y))$ and satisfies
$\sigma_{|\mathrm{Fix} (\varphi)} \equiv Id$.
One of the topological invariants is the analytic class of the unperturbed
diffeomorphism of the unfolding. Moreover $\sigma_{|x=0}$ is always a
local biholomorphism.

A key point of the classification is a shadowing property for
multi-parabolic diffeomorphisms. Roughly speaking, given a multi-parabolic
$\varphi \in \diff{p1}{2}$ there exists a vector field $X = g(x,y) \partial / \partial y$
with $\mathrm{Fix} (\varphi) = \mathrm{Sing} (X)$ such that every orbit of $\varphi$ can be
approximated by an orbit of ${\rm exp}(X)$ (Theorem 7.1 \cite{rib-mams}).
As a consequence the continuous dynamical system defined by the real flow $\Re (X)$ of
$X$ is  a good model of the topological behavior of $\varphi$.
In spite of this, generically there is no shadowing for unfoldings of tangent
to the identity diffeomorphisms.
Indeed the existence of a shadowing property for
a non-multi-parabolic element $\varphi$ of $\diff{p1}{2}$ implies
that $\varphi$ is embedded in an analytic flow \cite{rib-ast}.
Our strategy in this paper includes, as in the multi-parabolic case, approximating
$\varphi \in \diff{p1}{2}$ with ${\rm exp}(X)$  for some local vector field
$X = g(x,y) \partial / \partial y$
and then studying the real flow of $X$ to try to obtain interesting dynamical phenomena
associated to $\varphi$. Since there is no shadowing property for all orbits of $\varphi$
we have to show that the dynamics of $\Re (X)$ that we are trying to
replicate for $\varphi$ takes place in regions in which the orbits of
${\rm exp}(X)$ and $\varphi$ remain close.

The main tool in this paper is the study of Long Trajectories and Long Orbits.
These concepts were introduced in \cite{rib-mams}.
They are analogous to the concept of homoclinic trajectories for polynomial
vector fields introduced by Douady, Estrada and Sentenac \cite{DES}.
Let us focus on vector fields since the concepts are analogous and
the presentation is a little simpler.
Consider a local vector field $X=g(x,y) \partial / \partial y$ with
$g(0)=0$, $(\partial g/\partial y)(0)=0$ and $g(0,y) \not \equiv 0$,
i.e. an unfolding of a non-trivial vector field of vanishing
order higher than $1$.
Roughly speaking a Long Trajectory is given by the choice of a
point $y_{+} \neq 0$, a curve $\beta$ in the parameter space and
a continuous function $T:{\beta} \to {\mathbb R}^{+}$ such that
\[ (0,y_{-}) \stackrel{\mathrm{def}}{=} \lim_{x \in \beta, x \to 0} \mathrm{exp}(T(x)X)(x,y_{+}) \]
exists and
$\lim_{x \in \beta, x \to 0} T(x) = \infty$.
In general $(0,y_{-})$ does not belong to the trajectory through
$(0,y_{+})$.
We go from  $(0,y_{+})$ to $(0,y_{-})$ by following the real flow of $X$
an infinite time.
We say that $(X,y_{+},\beta,T)$ generates a {\it Long Trajectory} of $X$
containing $(0,y_{-})$. Denote $\varphi = {\rm exp}(X)$.
The point $(0,y_{-})$ is in the limit of the orbits of
$\varphi$ passing through points
$(x,y_{+})$ with $x \in T^{-1}({\mathbb N})$ when $x \to 0$.
We say that $(\varphi,y_{+},\beta,T)$ generates a {\it Long Orbit}
containing $(0,y_{-})$.
By replacing $T$ with $T+s$ for $s \in {\mathbb R}$
we obtain that $\mathrm{exp}(s X)(0,y_{-})$ is in the Long Orbit generated
by $(\varphi,y_{+},\beta,T+s)$.
The rest of the points in a neighborhood of $(0,y_{-})$ in $x=0$ are also
in Long Orbits of $\varphi$ through $(0,y_{+})$. They are
obtained by varying the curve $\beta$.
In particular the complex flow of the infinitesimal
generator of $\varphi_{|x=0}$ in the
repelling petal containing $(0,y_{-})$ can be
retrieved from Long Orbits through
$(0,y_{+})$.
In other words such complex flow is in the topological closure of the
pseudogroup generated by $\varphi$.

The Long Orbits phenomenon reminds Shcherbakov and Nakai's results
\cite{Shcherbakov-topan} \cite{Nakai-nonsolvable}
for non-solvable pseudogroups of holomorphic diffeomorphisms
of open neighborhoods of $0$ in ${\mathbb C}$.
A pseudogroup is non-solvable if its
associated group of local diffeomorphisms is non-solvable.
More precisely Nakai proves that there exists a real semianalytic subset $\Sigma$
such that any orbit of the pseudogroup is dense or empty
in every connected component of the complementary of $\Sigma$
(see \cite{Nakai-nonsolvable} for further details).
Moreover
the proof of Scherbakov's theorem
(a homeomorphism conjugating non-solvable pseudogroups is
holomorphic or anti-holomorphic) by Nakai is based
on finding real flows of holomorphic
vector fields that are in the topological closure
of a non-solvable pseudogroup.

Long Orbits are interesting in themselves.
Long Trajectories and Long Orbits are phenomena related to
instable behavior in the unfolding. Given $\varphi \in \dif{p1}{2}$
and a curve $\beta$ in the parameter space supporting a
Long Orbit then $\beta$ is tangent at $0$ to a unique
semi-line $\lambda {\mathbb R}^{+}$ for some $\lambda \in {\mathbb S}^{1}$.
Moreover $\lambda$ belongs to a finite set that only depends on
$\varphi$.
Generically in the parameter space there are no Long Orbits.
Notice that the absence of Long Orbits is a necessary condition
in the Glutsyuk point of view described above.
In spite of being scarce Long Orbits somehow vary continuously.
For instance the function $T$ in the definition can be
calculated by applying conveniently the residue theorem.
Indeed $T$ is (up to a bounded additive function)
a sum of meromorphic functions that are formal invariants of the
unfolding. The residue formula allows to describe the
evolution of the Long Orbits when we replace $\beta$
with nearby curves.
On the one hand Long Orbits appear in the regions of instability
of the unfolding and generically together with small divisors
phenomena. On the other hand they have a (rich) regular structure.
The main technical difficulty regarding Long Orbits is proving
their existence and properties. Once the setup is established
the Main Theorem is obtained by a relatively simple description of
the action of topological
conjugations on Long Orbits.

The analytic classification of
elements of $\diff{p1}{2}$ depends on studying transversal
structures to the dynamics of the unfolding. The point of view
behind the Main Theorem is closer to
Glutsyuk's point of view.
Anyway the focus on the parameter space is of complementary
type.
The extension of the Ecalle-Voronin invariants
\`{a} la Glutsyuk is obtained for regions of stability of
the unfolding. Nevertheless the topological dynamics in stability
regions is uninteresting. The
significant topological information is located in the neighborhood
of the instability directions.
\subsection{Rigidity of unfoldings}
The rigidity result of the Main Theorem extends to the
general case.
\begin{defi}
\label{def:simb}
Let $\phi, \rho \in \diff{1}{} \setminus \{Id\}$. We say that $\phi$ and $\rho$
have the same topological bifurcation type and we denote
$\phi \sim_{b} \rho$ if there exist
topologically conjugated unfoldings $\varphi$, $\varrho$ such that
$\varphi_{|x=0} \equiv \phi$, $\varrho_{|x=0} \equiv \rho$ and $N(\varphi) \neq 1$.
If the restriction $\sigma_{|x=0}$ of the topological conjugation $\sigma$
to the unperturbed line is
orientation-preserving (resp. orientation-reversing) we denote
$\phi {\sim}_{b}^{+} \rho$ (resp. $\phi {\sim}_{b}^{-} \rho$).
\end{defi}
We could naively think that this equivalence relation is the same induced by the
topological classification. The Main Theorem implies that a topological class
of conjugacy contains a continuous moduli of classes of
${\sim}_{b}$.
This is even true if we restrict ourselves to diffeomorphisms that are
embedded in analytic flows (Lemma \ref{lem:res})
since residues are not topological invariants.
Somehow surprisingly the analytic nature of
a generic $\phi \in \diff{1}{}$ is encoded in the topological
dynamics of any of its non-trivial unfoldings.

This kind of rigidity properties are typical in theory of
complex analytic foliations.
We already mentioned the results on non-solvable groups
by Scherbakov and Nakai. Other instances of the rigidity
of the moduli topological/analytic can be found in
Ilyashenko  \cite{Ilya-toppor}, Cerveau and Sad  \cite{Cerveau-Sad:modules},
Lins Neto, Sad and Scardua  \cite{NSS:rigidity},
Mar\'{i}n  \cite{Marin-rigidity},
Rebelo  \cite{Rebelo-rigidity}...
Moreover Cerveau and Moussu proved that in the context of non-solvable
non-exceptional groups, formal conjugacy implies analytic conjugacy
\cite{CM:bsmf}.

A natural question is what happens in the setup of the Main Theorem if
$\varphi_{|x=0}$ is analytically trivial.
It turns out that the situation is still rigid.
\begin{teo}
\label{teo:general} (General Theorem)
Let $\varphi, \eta \in \diff{p1}{2}$ with $(N,m)(\varphi) \neq (1,0)$ such that
there exists a homeomorphism $\sigma$ satisfying
$\sigma \circ \varphi = \eta \circ \sigma$.
Then $\sigma_{|x=0}$ is affine in Fatou coordinates.
Moreover $\sigma_{|x=0}$ is orientation-preserving if and only if
the action of $\sigma$ on the parameter space is
orientation-preserving.
\end{teo}
Topological conjugations are of the form
$\sigma(x,y) =(\sigma_{0}(x), \sigma_{1}(x,y))$. We say that
the action of $\sigma$ in the parameter space is orientation-preserving
if $\sigma_{0}$ is. Analogously we define the concept of holomorphic
action on the parameter space.

The definition of affine in Fatou coordinates is provided in
Definitions \ref{def:afffc} and \ref{def:afffc2}.
Affine in Fatou coordinates implies real analytic outside the origin.
In order to compare the Main Theorem and Theorem \ref{teo:general}
let us point out that holomorphic conjugations between elements of
$\diff{1}{} \setminus \{Id\}$ are
translations in Fatou coordinates.
The Main Theorem is a consequence of Theorem \ref{teo:general}.
Indeed we show that affine in Fatou coordinates implies
holomorphic or anti-holomorphic in the non-analytically trivial case.

How to strengthen the General Theorem?
A first approach is provided by the Main Theorem by considering
generic classes of analytic conjugacy.
Another possibility is trying to impose conditions on the action
of conjugations on the parameter space.
Finally we notice that for analytically trivial elements of $\diff{1}{}$
the formal and analytic conjugacy classes coincide. So it is
interesting to study the action of $\sigma_{|x=0}$ on formal invariants.
The next propositions establish a relation between the topological, formal and analytic
classifications.
\begin{pro}
\label{pro:holpar}
Let $\varphi, \eta \in \diff{p1}{2}$ with $(N,m)(\varphi) \neq (1,0)$ such that
there exists a homeomorphism $\sigma$ satisfying
$\sigma \circ \varphi = \eta \circ \sigma$.
Suppose that the action of $\sigma$ on the parameter space is
holomorphic (resp. anti-holomorphic). Then
$\sigma_{|x=0}$ is holomorphic (resp.  anti-holomorphic).
\end{pro}
Let $\phi(y) = y + c y^{\nu +1} + h.o.t. \in \diff{1}{}$ with
$\nu \in {\mathbb N}$ and $c \in {\mathbb C}^{*}$.
The number $\nu$ determines the class of topological conjugacy of
$\phi$. The diffeomorphism $\phi$ is formally conjugated to
a unique diffeomorphism
$y +  y^{\nu +1} + ((\nu+1)/2 - \lambda) y^{2 \nu +1}$ for
some $\lambda \in {\mathbb C}$.
The pair $(\nu,\lambda)$ provides a complete system of formal
invariants. We define $Res_{\phi}(0)=\lambda$
and $Res_{\varphi}(0,0) = Res_{\varphi_{|x=0}}(0)$
for $\varphi \in \dif{p1}{2}$.
\begin{pro}
\label{pro:rigi}
Let $\varphi, \eta \in \dif{p1}{2}$ with $N>1$ such that
there exists a homeomorphism $\sigma$ satisfying
$\sigma \circ \varphi = \eta \circ \sigma$.
Suppose that either $\varphi_{|x=0}$ or $\eta_{|x=0}$ is analytically trivial.
Suppose that either $Res_{\varphi}(0,0) \not \in i {\mathbb R}$ or
$Res_{\eta}(0,0) \not \in i {\mathbb R}$.
Then
\begin{itemize}
\item If $\sigma_{|x=0}$ is orientation-preserving
then $\sigma_{|x=0}$ is holomorphic if and only if $Res_{\varphi}(0,0) =Res_{\eta}(0,0)$.
\item If $\sigma_{|x=0}$ is orientation-reversing
then $\sigma_{|x=0}$ is anti-holomorphic if and only if $Res_{\varphi}(0,0) =\overline{Res_{\eta}(0,0)}$.
\end{itemize}
\end{pro}
On the one hand it is possible to construct examples of diffeomorphisms $\varphi, \eta$
satisfying the hypotheses of the previous proposition such that
$\varphi_{|x=0}$ and $\eta_{|x=0}$ are neither holomorphically nor
anti-holomorphically conjugated (Section \ref{sec:build}).
On the other hand if they are holomorphically conjugated (in the orientation-preserving case)
then $\sigma_{|x=0}$ is also holomorphic.
In other words given $\varphi \in \dif{p1}{2}$ as in Proposition \ref{pro:rigi}
the analytic class of $\eta_{|x=0}$ is not determined for $\eta$ in the class of
topological conjugacy of $\varphi$ but the conjugation $\sigma_{|x=0}$ is determined
up to composition with a holomorphic diffeomorphism
(see Proposition \ref{pro:atu}).
The condition $Res_{\varphi}(0,0) \not \in i {\mathbb R}$ on formal invariants implies
flexibility in the analytic classes of
$\eta_{|x=0}$ but once they are fixed there is rigidity of the conjugating
mappings.

Next we consider the case of purely imaginary formal invariants.
\begin{pro}
\label{pro:atui}
Let $\varphi, \eta \in \dif{p1}{2}$ with $N>1$ such that
there exists a homeomorphism $\sigma$ satisfying
$\sigma \circ \varphi = \eta \circ \sigma$.
Suppose that either $\varphi_{|x=0}$ or $\eta_{|x=0}$ is analytically trivial.
Suppose that either $Res_{\varphi}(0,0) \in i {\mathbb R}$ or
$Res_{\eta}(0,0) \in i {\mathbb R}$.
Then $\varphi_{|x=0}$ and $\eta_{|x=0}$ are analytically conjugated
(resp. anti-analytically conjugated)
if $\sigma$ is orientation-preserving (resp. orientation-reversing) on
the parameter space.
\end{pro}
The roles of analytic classes and conjugacies are reversed with respect to
Proposition \ref{pro:rigi}.
Indeed there are at most $2$ classes of analytic conjugacy of $\eta_{|x=0}$
in the set composed of the diffeomorphisms $\eta \in \dif{p1}{2}$ in the
topological class of $\varphi$.
In spite of the rigidity of analytic classes, conjugations are not rigid.
Even if $\varphi_{|x=0}$ and $\eta_{|x=0}$ are analytically conjugated the
mapping $\sigma_{|x=0}$ is not necessarily holomorphic.
Examples of this behavior are presented in Section \ref{sec:build}.

The following result is an
immediate consequence of Proposition \ref{pro:atui},
the Main and the General Theorems.
\begin{cor}
\label{cor:topiana}
Let $\phi, \rho \in \diff{1}{} \setminus \{Id\}$ with $Res_{\phi}(0) \in i {\mathbb R}$.
Then $\phi$ and $\rho$ have the same topological bifurcation type
if and only if $\phi$ and $\rho$ are holomorphically
or anti-holomorphically conjugated.
Moreover $\phi \sim_{b}^{+} \rho$ (resp. $\phi \sim_{b}^{-} \rho$) if and only if
$\phi$ and $\rho$ are holomorphically (resp. anti-holomorphically) conjugated.
\end{cor}
\subsection{Generalizations and consequences}
The results have a straightforward generalization to unfoldings
of resonant diffeomorphisms. A diffeomorphism
$\phi \in \diff{}{}$ is resonant if $\phi'(0)$ is a root of the unit
of order $q \in {\mathbb N}$.
An unfolding $\varphi(x,y) =(x,f(x,y))$ of $\phi$ satisfies
that the iterate $\varphi^{q}$ belongs to $\diff{p1}{2}$.

Consider unfoldings $\varphi, \varrho$ of resonant diffeomorphisms
$\phi, \rho \in \diff{}{}$ and a local homeomorphism $\sigma$
such that $\sigma \circ \varphi = \varrho \circ \sigma$.
We have $\phi'(0) = \rho'(0)$ if $\sigma_{|x=0}$ is orientation-preserving
and $\phi'(0) = \overline{\rho'(0)}$ if
$\sigma_{|x=0}$ is orientation-reversing by Naishul's theorem \cite{Naishul}.
Since $\sigma$ conjugates iterates of $\varphi$ and $\varrho$
then all theorems in the introduction have obvious generalizations.
Moreover all results
(except Proposition \ref{pro:atui} and
Corollary \ref{cor:topiana}) describe properties of $\sigma$
so the generalizations are trivial consequences.

The generalizations of Proposition \ref{pro:atui} and
Corollary \ref{cor:topiana} are also simple.
We apply our results to the iterates.
Then it suffices to prove that
given resonant $\phi, \rho \in \diff{}{}$ such that
$\phi'(0) = \rho'(0)$, $\phi^{q} \in \diff{1}{}$ and
$\phi^{q}$ is analytically conjugated to $\rho^{q}$ then
$\phi$ and $\rho$ are analytically conjugated.
This is a trivial consequence of the description
of the formal centralizer of $\phi^{q}$
(see Corollary 6.17, p. 88 \cite{Ilya-Yako}).

A very simple consequence of our results is that a homeomorphism
conjugating two generic unfolding of saddle-nodes
is either transversaly conformal or transversaly
anti-conformal by restriction to the unperturbed parameter.
\subsection{Outline of the paper}
The properties of Long Trajectories and Long Orbits are studied
by dividing a neighborhood of the origin in two kind of sets:
exterior sets in which the unfolding behaves as a trivial one
($N=1$) and compact-like sets in which the dynamics of the unfolding
is described in terms of the dynamics of a polynomial vector field.
This decomposition is called
{\it dynamical splitting} and it is explained in Section \ref{sec:dynspl}.

The existence of Long Trajectories and Long Orbits in the multi-parabolic
case was proved in \cite{rib-mams}. We introduce a simpler proof
that is valid in a more general setting. The idea is taking profit
of the polynomial vector fields that are canonically associated to
the unfolding. The dynamics of the real flow of polynomial vector fields
is treated in Section \ref{sec:dpvf}.
At this point it is good to point out that we need to
compare the dynamics of elements of $\diff{p1}{2}$ with exponentials
of vector fields. In Section \ref{sec:dynext} we develop the tools required for such a
task in the exterior sets of the dynamical splitting.
We complete the proof of the existence of Long Trajectories in
Section \ref{sec:lt}.

It is easy to see that the existence of Long trajectories implies
the existence of Long Orbits for elements of $\dif{p1}{2}$ with $N>1$
(Proposition \ref{pro:ltlo}). Indeed the Long Orbits are constructed
in the neighborhood of Long Trajectories of the real flow of
a holomorphic vector field $X=g(x,y) \partial / \partial y$ such that
${\rm exp}(X)$ approximates $\varphi$.
A topological homeomorphism $\sigma$ conjugating $\varphi, \eta \in \dif{p1}{2}$
with $N>1$ does not conjugate the real flows of $X$ and $Y$ if
${\rm exp}(Y)$ approximates $\eta$.
Then it is not clear a priori that the image by $\sigma$ of a Long Orbit
is in the neighborhood of a Long Trajectory of the real flow of $Y$.
This shadowing property is important since it is the base for the residue formula
that provides the quantitative estimates of Long Orbits.
The tracking (or shadowing) property is proved in Section \ref{sec:tracking}
by showing that trajectories of the real flow of $Y$ in the
neighborhood of Long Orbits of $\eta$ satisfy a Rolle property.

The rigidity results in the introduction are proved in Section
\ref{sec:infinite} for unfoldings of the identity map and in Section
\ref{sec:finite} for the remaining cases.
Examples showing that the hypotheses in the results are optimal
are presented in Section \ref{sec:build}.
\section{Notations}
We denote by $\diff{}{n}$ the group of local complex analytic diffeomorphisms
defined in a neighborhood of $0$ in ${\mathbb C}^{n}$.
We denote by $\diff{1}{}$ the group of local complex analytic one-dimensional
diffeomorphisms whose linear part is the identity.
\begin{defi}
We define $\diff{p1}{n+1}$ as the set of $n$-parameter unfoldings of local complex
analytic tangent to the identity diffeomorphisms.
In other words $\varphi \in \diff{p1}{n+1}$ is of the form
$\varphi(x_{1},\hdots,x_{n},y)=(x_{1},\hdots,x_{n},f(x_{1},\hdots,x_{n},y))$
where $f \in {\mathbb C}\{x_{1},\hdots,x_{n},y\}$ and
the unperturbed diffeomorphism $f(0,\hdots,0,y)$ is tangent to the identity, i.e.
$f(0)=0$ and $(\partial f/\partial y)(0,\hdots,0)=1$.
We denote by $\dif{p1}{n+1}$ the subset of elements
$\varphi \in \diff{p1}{n+1}$ such that $\varphi_{|x_{1}=\hdots=x_{n}=0} \neq Id$.
Indeed $\diff{p1}{n+1} \setminus \dif{p1}{n+1}$ is the set of unfoldings of
the identity map.
\end{defi}
We relate the topological properties of unfoldings of tangent to the identity
diffeomorphisms and unfoldings of vector fields with a multiple singular point.
\begin{defi}
We denote by $\Xnt$ the set of local complex analytic vector fields of the form
$X=g(x,y) \partial / \partial y$ where $g \in {\mathbb C}\{x,y\}$ satisfies
$g(0,0)=0$ and $(\partial g/\partial y)(0,0)=0$. In other words $X$ is
an unfolding of the vector field $g(0,y) \partial /\partial y$ that has
a multiple zero at the origin.
We denote $\Xntg=\{ X \in \Xnt: X_{|x=0} \not \equiv 0\}$.
\end{defi}
\begin{defi}
We denote by $\Xt$ the subset of $\Xnt$ of local complex analytic vector fields $X$
such that any irreducible component of $\mathrm{Sing}(X)$ different than $x=0$ is of the
form $y=\gamma(x)$ for some $\gamma \in {\mathbb C}\{x\}$.
In other words the irreducible components of $\mathrm{Sing}(X)$ are transversal to the
fibration $dx=0$.
Let us remark that given $g(x,y) \partial /\partial y \in \Xnt$ there exists
$k \in {\mathbb N}$ such that $g(x^{k},y) \partial /\partial y$ belongs to $\Xt$.
We denote $\Xtg = \Xt \cap \Xntg$.
\end{defi}
Given a vector field $X$ defined in a domain $U \subset {\mathbb C}^{n}$
we denote by $\Re (X)$ the real flow of $X$, namely the flow defined in
${\mathbb R}^{2n} = {\mathbb C}^{n}$ by considering real times.
For instance if $X$ is of the form $a(x,y) \partial/\partial x + b(x,y) \partial / \partial y$
we have
\[ \Re (X) = Re (a) \frac{\partial}{\partial x_{1}} + Im (a)  \frac{\partial}{\partial x_{2}}  +
Re(b) \frac{\partial}{\partial y_{1}}  + Im(b) \frac{\partial}{\partial y_{2}} \]
where $x=x_{1} + i x_{2}$ and $y = y_{1} + i y_{2}$.
\begin{defi}
\label{def:traj}
Let $\gamma_{P}(s)$ be the trajectory of $\Re (Z)$ such that $\gamma_{P}(0)=P$.
We define ${\mathcal I}(Z, P,F)$ the maximal interval where
$\gamma_{P}(s)$ is well-defined and belongs to $F$ for any $s \in {\mathcal I}(Z, P,F)$ whereas
$\gamma_{P}(s)$ belongs to the interior $\accentset{\circ}{F}$ of $F$ for
any $s \neq 0$ in the interior of ${\mathcal I}(Z, P,F)$.
We denote $\Gamma(Z,P,F)=\gamma_{P}({\mathcal I}(Z, P,F))$. We define
\[ \partial {\mathcal I}(Z, P,F) = \{ \inf({\mathcal I}(Z, P,F)), \sup({\mathcal I}(Z, P,F)) \}
\subset {\mathbb R} \cup \{-\infty, \infty\}. \]
We denote $\Gamma(Z,P,F)(s) = \gamma_{P}(s)$.
\end{defi}
\begin{defi}
\label{def:normal}
Let $\varphi   \in \diff{1}{}$ (resp. $\diff{p1}{2}$).
Consider a vector field $X=g \partial / \partial y$ for $g \in {\mathbb C}\{y\}$
(resp. ${\mathbb C}\{x,y\}$)
such that $y \circ \varphi - y \circ \mathrm{exp}(X) \in (y \circ \varphi -y)^{3}$.
We say that $X$ and ${\mathfrak F}_{\varphi}=\mathrm{exp}(X)$ are {\it convergent normal forms}
of $\varphi$. There exist convergent normal forms
(Proposition 1.1 of \cite{UPD}).
\end{defi}
The idea is that the dynamics of $\mathrm{exp}(X)$ is much simpler than the dynamics of $\varphi$.
In particular the orbits of $\mathrm{exp}(X)$ are contained in the trajectories of $\Re (X)$.
Generically the orbits of $\mathrm{exp}(X)$ and $\varphi$ are very different. In spite of this
$\Re (X)$ provides valuable information of the dynamics of $\varphi$ (Section \ref{sec:tracking}).
\begin{defi}
\label{def:delta}
Let $\varphi \in \diff{}{} \cup \diff{p1}{2}$. Fix a convergent normal form $X$ of $\varphi$
and ${\mathfrak F}_{\varphi}=\mathrm{exp}(X)$. We define
\[ \Delta_{\varphi} = \psi_{X} \circ \varphi - \psi_{X} \circ {\mathfrak F}_{\varphi} =
\psi_{X} \circ \varphi - (\psi_{X} +1)  . \]
Indeed we have
\[ \Delta_{\varphi} = \psi_{X} \circ \varphi - \psi_{X} \circ {\mathfrak F}_{\varphi}  \sim
\frac{\psi_{X}}{\partial y} (y \circ \varphi - y \circ \mathrm{exp}(X)) =
O \left( \frac{(y \circ \varphi -y)^{3}}{X(y)} \right). \]
The function $\Delta_{\varphi}$ belongs to the ideal $(y \circ \varphi -y)^{2} = (X(y))^{2}$
of ${\mathbb C}\{x,y\}$ (see Lemma 7.2.1 of \cite{rib-mams}).
\end{defi}
The function $\Delta_{\varphi}$ measures how good is the approximation of $\varphi$ provided by
$\mathrm{exp}(X)$.
\begin{defi}
Let $X \in \Xnt$. We define $N(X)$ as the number of points in
$\mathrm{Sing} (X) \cap \{x=x_{0}\}$ for $x_{0} \neq 0$.
Analogously we define $N(\varphi)$ for $\varphi \in \diff{p1}{2}$
by replacing $\mathrm{Sing} (X)$ with $\mathrm{Fix} (\varphi)$.
We have $N(\varphi)=N(X)$ if $X$ is a convergent normal form of $\varphi$.
\end{defi}
\begin{defi}
Let $X \in \Xnt$. We define $m(X) \in {\mathbb N} \cup \{0\}$
as the multiplicity of
$x=0$ in $\mathrm{Sing} (X)$. More precisely  $X$ is of the form $x^{m(X)} X_{0}$ for some
holomorphic vector field such that $\{x=0\} \not \subset \mathrm{Sing} (X)$.
We define $m(\varphi)$ as the multiplicity of $x=0$ in $\mathrm{Fix} (\varphi)$.
\end{defi}
\begin{defi}
Let $X=g(y) \partial / \partial y$ with $g \in {\mathbb C}\{y\}$.
We define $\nu (X)=a-1$ where $a$ is the vanishing order of of $g(y)$ at $0$.
\end{defi}
\begin{defi}
Let $\varphi(y) \in \diff{}{}$.
We define $\nu (\varphi)=a-1$ where $a$ is the vanishing order of of $\varphi(y)-y$ at $0$.
\end{defi}
\begin{defi}
Let $X =x^{m(X)} g(x,y) \partial / \partial y \in \Xnt$.
We define $\nu (X)$ as $\nu (g(0,y) \partial / \partial y)$.
Let $\varphi=(x,y + x^{m(\varphi)} f(x,y)) \in \diff{p1}{2}$.
We define $\nu (\varphi)$ as $a-1$ where $a$ is the vanishing order of $f(0,y)$ at $0$.
\end{defi}
\begin{defi}
\label{def:atpetvf}
Let $X = g(y) \partial / \partial y$ with $X \neq 0$ and $\nu(X)>0$.
We say that ${\mathcal P}$
is an {\it attracting petal} of $\Re (X)_{|B(0,\epsilon)}$
if it is a connected component of
\[ \{ y \in B(0,\epsilon): [0,\infty) \subset {\mathcal I}(y,X,B(0,\epsilon)) \
\mathrm{and} \ \lim_{s \to \infty} \Gamma(y,X,B(0,\epsilon))(s) = 0\} \]
where $B(0,\epsilon)$ is the open ball of center at the origin and radius $\epsilon$.
Analogously a repelling petal of $\Re (X)_{|B(0,\epsilon)}$
is an attracting petal of $\Re (-X)_{|B(0,\epsilon)}$.
We consider the petals ${\{ {\mathcal P}_{j}\}}_{j \in {\mathbb Z}/(2 \nu(X){\mathbb Z})}$
ordered in counter clock wise sense (see \cite{Loray5}).
\end{defi}
A vector field $X=(a_{\nu +1} y^{\nu +1} + h.o.t.) \partial /\partial y$, $a_{\nu+1} \neq 0$, has
very similar petals as $a_{\nu +1} y^{\nu +1} \partial /\partial y$.
Consider a half line $e^{i \theta_{0}} {\mathbb R}^{+}$ with $a_{\nu+1} e^{i \theta_{0} \nu} \in {\mathbb R}^{*}$.
The set of half lines in $\{ a_{\nu+1} y^{\nu} \in {\mathbb R}\}$ is
${\{ e^{i \theta_{j}} {\mathbb R}^{+} \}}_{j \in {\mathbb Z}/(2 \nu {\mathbb Z})}$
where $\theta_{j} = \theta_{0} + \pi j /\nu$.
Given $j \in {\mathbb Z}/(2 \nu {\mathbb Z})$ there exists a petal
${\mathcal P}_{j}$ that is bisected by $e^{i \theta_{j}} {\mathbb R}^{+}$.
More precisely given $\eta>0$
the sector
$(0,\delta) e^{i (\theta_{j} - \pi /\nu + \eta, \theta_{j} +  \pi / \nu - \eta)}$ is contained in
${\mathcal P}_{j}$ for $\delta >0$ small enough.
Moreover ${\mathcal P}_{j}$ is attracting if and only if
$a_{\nu+1} e^{i \theta_{j} \nu} \in {\mathbb R}^{-}$.
Two petals have non-empty intersection if and only if they are consecutive.
These properties can be easily proved by using the change of coordinates
$z = -1/(\nu a_{\nu} y^{\nu})$. The vector field $X$ is of the form
$(1 + o(1)) \partial / \partial z$ where $z$ is defined in a neighborhood of $\infty$.
\begin{defi}
Let $X$ be a holomorphic vector field defined in an open set
$U$ of ${\mathbb C}^{n}$. We say that  a holomorphic
$\psi: U \to {\mathbb C}$ is a {\it Fatou coordinate} of $X$ if
$X(\psi) \equiv 1$ in $U$.
\end{defi}
\begin{defi}
\label{def:atpetd}
Let $\phi \in \diff{1}{}$ with $\phi \neq Id$.
We say that ${\mathcal P}'$
is an {\it attracting petal} of $\phi_{|B(0,\epsilon)}$
if it is a connected component of
\[ \{ y \in B(0,\epsilon): \phi^{j}(y) \in B(0,\epsilon) \ \forall j \in {\mathbb N}
\ \mathrm{and} \ \lim_{j \to \infty}  \phi^{j}(y)=0 \} . \]
Analogously a {\it repelling petal} of $\phi_{|B(0,\epsilon)}$
is an attracting petal of $\phi_{|B(0,\epsilon)}^{-1}$.
We consider the petals ${\{ {\mathcal P}_{j}'\}}_{j \in {\mathbb Z}/(2 \nu(\phi){\mathbb Z})}$
ordered in counter clock wise sense (see \cite{Loray5}).
\end{defi}
The petals of $y + a_{\nu +1} y^{\nu +1} + h.o.t.$, $a_{\nu +1} \neq 0$,
satisfy the properties
described below Definition \ref{def:atpetvf} for the petals of
$a_{\nu +1} y^{\nu +1} \partial / \partial y$.
\begin{defi}
Let $\phi   \in \diff{1}{}$ with
$\phi \neq Id$. Consider a petal ${\mathcal P}'$ of $\phi$.
Consider a convergent normal form $X$ of $\phi$ and a Fatou coordinate
$\psi$ of $X$ in ${\mathcal P}'$. We say that $\psi_{{\mathcal P}'}^{\phi}$ is
a {\it Fatou coordinate} of $\phi$ in ${\mathcal P}_{+}'$ if
$\psi_{{\mathcal P}_{+}'}^{\phi} \circ \varphi \equiv \psi_{{\mathcal P}_{+}'}^{\phi} +1$
and there exists
$c \in {\mathbb C}$ such that
\[ (\psi_{{\mathcal P}'}^{\phi} - (\psi + c))(y) = o(\max_{j \geq 0} |\phi^{j}(y)|) \ \ \mathrm{or}
\ \ (\psi_{{\mathcal P}'}^{\phi} - (\psi + c))(y) = o(\max_{j \geq 0} |\phi^{-j}(y)|) \]
depending on wether ${\mathcal P}'$ is attracting or repelling.
The definition depends only on $\phi$ and ${\mathcal P}'$. The Fatou coordinate is unique
up to an additive constant. Indeed
\[ \psi_{{\mathcal P}'}^{\phi}(y) = \psi(y) + \sum_{j=0}^{\infty} \Delta_{\phi}(\phi^{j}(y)) \ \ \mathrm{or} \ \
\psi_{{\mathcal P}'}^{\phi}(y) = \psi(y) - \sum_{j=1}^{\infty} \Delta_{\phi}(\phi^{-j}(y)) \]
is a Fatou coordinate of $\phi$ in ${\mathcal P}'$ (see Definition \ref{def:delta})
depending on wether ${\mathcal P}'$ is attracting or repelling (see \cite{Loray5}).
\end{defi}
\begin{defi}
\label{def:infgen}
Let $\phi \in \diff{1}{}$. Consider a petal ${\mathcal P}'$ of $\phi$.
There exists a unique vector field $X_{{\mathcal P}'}^{\phi}$
defined in ${\mathcal P}'$ such that
$X_{{\mathcal P}'}^{\phi}$ is the $1/\nu(\phi)$
Gevrey sum of the infinitesimal generator of $\phi$ in
${\mathcal P}'$ and
$\phi = \mathrm{exp}(X_{{\mathcal P}'}^{\phi})$ \cite{Ecalle}.
Equivalently $X_{{\mathcal P}'}^{\phi}$ is the unique holomorphic vector field
defined in ${\mathcal P}'$ such that
$X_{{\mathcal P}'}^{\phi}(\psi_{{\mathcal P}'}^{\phi}) \equiv 1$ for
some (and then every) Fatou coordinate $\psi_{{\mathcal P}'}^{\phi}$ of
$\phi$ in ${\mathcal P}'$.
If ${\mathcal P}'$ is a petal of $\varphi_{|x=0}$ for
$\varphi \in \diff{p1}{2}$ we denote
$X_{{\mathcal P}'}^{\varphi} = X_{{\mathcal P}'}^{\varphi_{|x=0}}$.
\end{defi}
\begin{defi}
\label{def:afffc}
Let $\varphi, \eta \in \dif{p1}{2}$ with $N>1$ and
a homeomorphism $\sigma$ conjugating $\varphi$ and $\sigma$.
We say that $\sigma_{|x=0}$ is {\it affine in Fatou coordinates} if there exists a
${\mathbb R}$-linear isomorphism ${\mathfrak h}:{\mathbb C} \to {\mathbb C}$
such that
\[ {\mathfrak h} (z) = (\psi_{\sigma({\mathcal P}')}^{\eta} \circ \sigma \circ
\mathrm{exp}(z X_{{\mathcal P}'}^{\varphi}) -
\psi_{\sigma({\mathcal P}')}^{\eta} \circ \sigma)(0,y) \]
for any petal ${\mathcal P}'$ of $\varphi_{|x=0}$.
The previous property implies that
$\psi_{\sigma({\mathcal P}')}^{\eta} \circ \sigma \circ (\psi_{{\mathcal P}'}^{\varphi})^{-1}$
is affine for any choice $\psi_{{\mathcal P}'}^{\varphi}$,
$\psi_{\sigma({\mathcal P}')}^{\eta}$
of Fatou coordinates.
\end{defi}
\begin{defi}
\label{def:afffc2}
Let $\varphi, \eta \in \diff{p1}{2}$ with $m(\varphi)>0$ such that
there exists a homeomorphism $\sigma$ satisfying
$\sigma \circ \varphi = \eta \circ \sigma$.
Consider normal forms
$X=x^{m(\varphi)} X_{0}$ and $Y=x^{m(\eta)} Y_{0}$
for $\varphi$ and $\eta$ respectively.
Let $\psi_{0}^{\varphi}$, $\psi_{0}^{\eta}$ Fatou coordinates
of $(X_{0})_{|x=0}$ and $(Y_{0})_{|y=0}$ respectively.
We say that $\sigma_{|x=0}$ is {\it affine in Fatou coordinates} if
$\psi_{0}^{\eta} \circ \sigma \circ (\psi_{0}^{\varphi})^{-1}$
is an affine isomorphism.
\end{defi}
\begin{rem}
The vector fields $(X_{0})_{|x=0}$ and $(Y_{0})_{|y=0}$ are well-defined
up to multiplication by a non-zero complex number. They do not depend on
the choices of convergent normal forms. Hence the previous property is
well-defined.
\end{rem}
\begin{defi}
\label{def:contsets}
We consider coordinates $(x,y) \in {\mathbb C} \times {\mathbb C}$ or
$(r,\lambda,y) \in
{\mathbb R}_{\geq 0} \times {\mathbb S}^{1} \times {\mathbb C}$
in ${\mathbb C}^{2}$. Given a set $F \subset {\mathbb C}^{2}$ we denote
by $F(x_{0})$ the set $F \cap \{ x=x_{0} \}$ and by $F(r_{0},\lambda_{0})$
the set $F \cap \{ (r,\lambda)=(r_{0},\lambda_{0}) \}$.
\end{defi}
\begin{defi}
\label{def:ue}
We define $U_{\epsilon} = {\mathbb C} \times B(0,\epsilon)$. In practice we always
work with domains of the form $B(0,\delta) \times B(0,\epsilon)$ for some
small $\delta \in {\mathbb R}^{+}$.
\end{defi}
\section{Dynamical splitting}
\label{sec:dynspl}
We define a dynamical splitting $\digamma_{X}$ associated to an element $X$
of $\Xt$ such that $N(X) \geq 1$. Most of the concepts were introduced
in \cite{JR:mod}. The idea is dividing a neighborhood of the origin
${\mathcal T}_{0}^{\epsilon}=\{(x,y) \in B(0,\delta) \times \overline{B(0,\epsilon)} \}$,
where $B(0,\epsilon)$ is the open ball of center at the origin and radius $\epsilon$,
in sets in which the dynamics of $\Re (X)$ is simpler to describe.
The sets of the division are obtained through a process of desingularization of
$\mathrm{Sing} (X)$.

We say that ${\mathcal T}_{0}^{\epsilon}=\{(x,y) \in B(0,\delta) \times \overline{B(0,\epsilon)} \}$
is a {\it seed}. Let us explain the terminology.
The set ${\mathcal T}_{0}={\mathcal T}_{0}^{\epsilon}$ is the starting point of the division.
Throughout the process we obtain sets of the form
${\mathcal T}_{\beta} = \{ (x,t) \in  B(0,\delta) \times \overline{B(0,\eta)} \}$
for some new coordinate $t$. Since these sets share analogous properties as
${\mathcal T}_{0}$ we can define a recursive process of division.
The sets of the form ${\mathcal T}_{\beta}$ are called {\it seeds} and the
coordinate $t$ is canonically associated to ${\mathcal T}_{\beta}$ along the process.

 We provide a recursive method to divide ${\mathcal T}_{0}$. At each step we have
a vector
$\beta=(0, \beta_{1}, \hdots, \beta_{k}) \in \{0\} \times {\mathbb C}^{k}$ with $k \geq 0$
and a seed
${\mathcal T}_{\beta} = \{ (x,t) \in  B(0,\delta) \times \overline{B(0,\eta)} \}$ in coordinates
$(x,t)$ canonically associated to ${\mathcal T}_{\beta}$.
We say that the coordinates $(x,t)$ are {\it adapted} to
${\mathcal T}_{\beta}$ and ${\mathcal E}_{\beta}$ (it is defined below).
In the first step we have
$k=0$, $\beta=0$ and $t=y$. Suppose also that
\begin{equation}
\label{for:forx}
 X = x^{e({\mathcal E}_{\beta})}
v(x,t) (t-\gamma_{1}(x))^{s_{1}} \hdots (t-\gamma_{p}(x))^{s_{p}}
\partial / \partial{t}
\end{equation}
in ${\mathcal T}_{\beta}$ where $\gamma_{1}(0)=\hdots=\gamma_{p}(0)=0$
and $\{ v=0 \} \cap {\mathcal T}_{\beta}=\emptyset$.

For $p=1$ we define the {\it terminal exterior basic set} ${\mathcal E}_{\beta}={\mathcal T}_{\beta}$,
we do not split the terminal seed ${\mathcal T}_{\beta}$.
The singular set of $X$ in ${\mathcal T}_{\beta}$ is already desingularized.
Suppose $p>1$.
We define $t=xw$ and
$S_{\beta} =  \{ (\partial \gamma_{1}/\partial x)(0), \hdots,
(\partial \gamma_{p}/\partial x)(0) \}$.
Consider the blow-up of the point $(x,t)=(0,0)$;
the set $S_{\beta}$ can be interpreted as the
intersection of the strict transform of
$\prod_{j=1}^{p} (t-\gamma_{j}(x))^{s_{j}}=0$ and the divisor.
We define the {\it exterior basic set}
${\mathcal E}_{\beta} = {\mathcal T}_{\beta} \cap \{ |t| \geq |x|\rho \}$
and $M_{\beta}= \{ (x,w) \in B(0,\delta) \times \overline{B(0,\rho)} \}$
for some $\rho>>0$.
The set $M_{\beta}$ contains $\prod_{j=1}^{p} (t-\gamma_{j}(x))^{s_{j}}=0$.
One of the ideas of the construction is that since ${\mathcal E}_{\beta}$ is
far away of the singular points then the dynamics of the vector fields
$\Re (X)_{|{\mathcal E}_{\beta}}$ and $\Re (x^{e({\mathcal E}_{\beta})}
v(x,t) (t-\gamma_{1}(x))^{s_{1}+\hdots+s_{p}} \partial / \partial{t})_{|{\mathcal E}_{\beta}}$
are very similar. We denote
\[ \partial_{e} {\mathcal E}_{\beta} = \{ (x,t) \in B(0,\delta) \times \partial B(0,\eta) \}
\ {\rm and} \
\nu({\mathcal E}_{\beta}) = s_{1}+\hdots+s_{p}-1.  \]
We define
\[ \partial_{I} {\mathcal E}_{\beta} = \{ (x,t) \in B(0,\delta) \times \overline{B(0,\eta)} :
|t| = |x| \rho \} \]
if ${\mathcal E}_{\beta}$ is not terminal.
We say that the sets $\partial_{e} {\mathcal E}_{\beta}$ and
$\partial_{I} {\mathcal E}_{\beta}$ are the {\it exterior} and {\it interior boundaries} of
${\mathcal E}_{\beta}$ respectively and $e({\mathcal E}_{\beta})$ is the
{\it exterior exponent} of ${\mathcal E}_{\beta}$.
\begin{defi}
We say that an exterior set ${\mathcal E}_{\beta}$ is {\it parabolic} if
$\nu({\mathcal E}_{\beta})>0$. Every non-parabolic exterior set is terminal
but a terminal exterior set can be parabolic.
\end{defi}
\begin{defi}
\label{def:rvfe}
Given an exterior set ${\mathcal E}_{\beta}$ we define
$X_{{\mathcal E}_{\beta}}$ as the vector field defined in ${\mathcal T}_{\beta}$ such that
$X= x^{e({\mathcal E}_{\beta})} X_{{\mathcal E}_{\beta}}$.
\end{defi}
We have
\[ X  =
x^{e({\mathcal E}_{\beta})+s_{1}+\hdots+s_{p}-1}v(x,xw)
{\left({ w -  \gamma_{1}(x) / x }\right)}^{s_{1}} \hdots
{\left({ w - \gamma_{p}(x) / x }\right)}^{s_{p}}
\partial / \partial{w}. \]
\begin{defi}
\label{def:rvfc}
We define $\nu({\mathcal C}_{\beta}) =  \nu({\mathcal E}_{\beta})$ and
$e({\mathcal C}_{\beta}) = e({\mathcal E}_{\beta})+ \nu({\mathcal E}_{\beta})$.
We define
$X_{{\mathcal C}_{\beta}}$ as the vector field defined in
$M_{\beta}$ such that
$X= x^{e({\mathcal C}_{\beta})} X_{{\mathcal C}_{\beta}}$.
\end{defi}
\begin{defi}
\label{def:pol}
We define the polynomial vector field
\[ X_{\beta}(\lambda) = \lambda^{e({\mathcal C}_{\beta})} v(0,0)
{\left({ w - (\partial \gamma_{1} / \partial x)(0) }\right)}^{s_{1}} \hdots
{\left({ w - (\partial \gamma_{p} / \partial x)(0) }\right)}^{s_{p}}
\partial / \partial{w}  \]
for $\lambda \in {\mathbb S}^{1}$ (see Eq. (\ref{for:forx}), note that $t=xw$)
associated to $X$, ${\mathcal T}_{\beta}$ and ${\mathcal C}_{\beta}$.
\end{defi}
The dynamics of $\Re(X)_{|M_{\beta}}$ and $\Re(X_{\beta}(\lambda))_{|M_{\beta}}$
are similar (up to reparametrization of the trajectories) outside of a neighborhood of
the singular set.
If $X_{\beta}(\lambda)$ has a multiple zero at $\zeta \in S_{\beta}$ we just
choose $w_{\zeta}=w-\zeta$.
Suppose now that
$X_{\beta}(\lambda)$ has a simple zero at $\zeta \in S_{\beta}$.
Assume that $(\partial \gamma_{1} / \partial x)(0)=\zeta$.
As a consequence there exist coordinates
$(r,\lambda,w_{\zeta})$ defined in the neighborhood of
$\{ (r,\lambda,w) \in \{0\} \times {\mathbb S}^{1} \times \{\zeta\}\}$
such that $\lambda^{e({\mathcal C}_{\beta})} X_{{\mathcal C}_{\beta}}$
is of the form $\lambda^{e({\mathcal C}_{\beta})} h(r \lambda) w_{\zeta} \partial / \partial w_{\zeta}$
for some function $h$ (see \cite{Loray5}).
Indeed $w_{\zeta}$ is a linearizing coordinate.
Hence $|w_{\zeta}| =\eta'$ for $\eta'>0$ small
is transversal to $\Re (X)$ if $(x,w)=(x_{0},\gamma_{1} (x_{0})/x_{0})$
is an attractor or a repeller.
Moreover $|w_{\zeta}|=\eta'$
is invariant by $\Re (X)$ if $(x,w)=(x_{0},\gamma_{1} (x_{0})/x_{0})$
is an indifferent singular point.

We define the {\it compact-like basic set}
\[ {\mathcal C}_{\beta}=\{
(x,w) \in B(0,\delta) \times \overline{B(0,\rho)} \} \setminus
(\cup_{\zeta \in S_{\beta}}
\{ (x,w_{\zeta}) \in  B(0,\delta) \times B(0,\eta_{\beta,\zeta}) \}) \]
where $\eta_{\beta,\zeta}>0$ is small enough for any $\zeta \in S_{\beta}$.  We denote
\[ \partial_{e} {\mathcal C}_{\beta} = \{ (x,w) \in B(0,\delta) \times \partial B(0,\rho) \}, \
\partial_{I} {\mathcal C}_{\beta} = \cup_{\zeta \in S_{\beta}} \{ (x,w_{\zeta}) \in B(0,\delta) \times
\partial B(0, \eta_{\beta,\zeta}) \} . \]
We say that   $e({\mathcal C}_{\beta})$ (see Definition \ref{def:rvfc})
is the {\it exponent} of ${\mathcal C}_{\beta}$.
Notice that the vector field $X_{\beta}(\lambda)$ determines the dynamics of
$\Re (X)$ in ${\mathcal C}_{\beta}$
since
$X /|x|^{e({\mathcal C}_{\beta})} \to X_{\beta}(\lambda_{0})$ in ${\mathcal C}_{\beta}$
if $x \to 0$ and $x/|x| \to \lambda_{0}$.

Fix $\zeta \in S_{\beta}$. We define the seed
${\mathcal T}_{\beta, \zeta}= \{ (x,t') \in B(0,\delta) \times \overline{B(0,\eta_{\beta,\zeta})} \}$ where
$t'$ is the coordinate $w_{\zeta}$. By definition $(x,t')$ is the set of adapted coordinates
associated to ${\mathcal T}_{\beta, \zeta}$.
We denote $e({\mathcal E}_{\beta, \zeta}) = e({\mathcal C}_{\beta})$.
We have
\[ X = x^{e({\mathcal E}_{\beta, \zeta})} h(x,t')
\prod_{(\partial \gamma_{j}/\partial x)(0) = \zeta}
{\left({ t' - \left({
\frac{\gamma_{j}(x)}{x} - \zeta }\right)
}\right)}^{s_{j}}  \frac{\partial}{\partial{t'}}, \
 X = x^{e({\mathcal E}_{\beta, \zeta})} h(x)
t'  \frac{\partial}{\partial{t'}}
 \]
depending on wether or not
$X_{\beta}(\lambda)$ has a multiple zero at $\zeta$.

Resuming, in each step of the process we either decide not to split ${\mathcal T}_{\beta}$
or we divide it in sets ${\mathcal E}_{\beta}$,
${\mathcal C}_{\beta}$, ${\{{\mathcal T}_{\beta, \zeta}\}}_{\zeta \in S_{\beta}}$
where $S_{\beta}$ is a finite subset of ${\mathbb C}$.
The seeds ${\mathcal T}_{\beta, \zeta}$ for  $(\beta,\zeta) \in \{0\} \times {\mathbb C}^{k+1}$
with $\zeta \in S_{\beta}$ are divided in ulterior steps of the process.
The sets ${\mathcal T}_{\beta}$, ${\mathcal E}_{\beta}$ and ${\mathcal C}_{\beta}$
are defined by induction on $k$.
The sets ${\mathcal E}_{\beta}$ are called
{\it exterior basic sets} whereas the sets ${\mathcal C}_{\beta}$
are called {\it compact-like basic sets} (see Example \ref{exa:dynsplit}).
\begin{defi}
\label{def:dynfx}
A dynamical splitting $\digamma_{X}$ associated to $X \in \Xt$
is a division of a neighborhood of the origin
${\mathcal T}_{0}^{\epsilon}=B(0,\delta) \times \overline{B(0,\epsilon)}$ in exterior and
compact-like basic sets. The choice of dynamical splitting is not unique.
\end{defi}
\begin{figure}[h]
\begin{center}
\includegraphics[height=6cm,width=6cm]{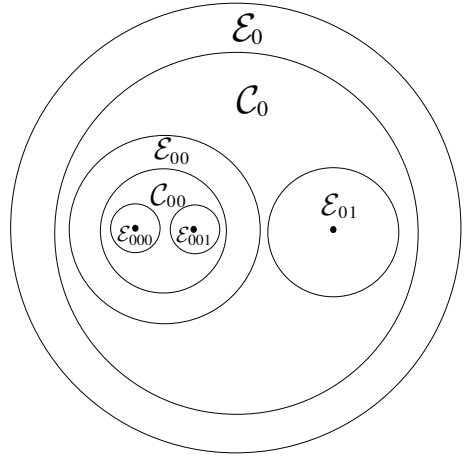}
\end{center}
\caption{Splitting for $X = y(y-{x}^{2})(y-x) \partial/\partial{y}$ in a line $x=x_{0}$} \label{fig5}
\end{figure}

\begin{exa}
\label{exa:dynsplit}
Consider $X=y (y-x^{2}) (y-x) \partial / \partial y$.
Denote $w=y/x$. The vector field $X$ has the form $x^{2} w (w-x) (w-1) \partial / \partial w$
in coordinates $(x,w)$. The polynomial vector field $X_{0}(\lambda)$ associated to the seed
${\mathcal T}_{0}=B(0,\delta) \times \overline{B(0,\epsilon)}$ is equal to
$\lambda^{2} w^{2}  (w-1) \partial / \partial w$.

The exterior and compact-like sets
associated to ${\mathcal T}_{0}$ are
${\mathcal E}_{0}={\mathcal T}_{0} \cap \{ |y| \geq \rho_{0} |x| \}$ and
${\mathcal C}_{0}=\{ |y| \leq \rho_{0} |x| \} \setminus (\{ |w|<\eta_{00} \} \cup \{ |w_{1}|<\eta_{01} \} )$
respectively. The set ${\mathcal C}_{0}$ encloses the
seeds ${\mathcal T}_{00}=\{ |w| \leq \eta_{00} \}$ and
${\mathcal E}_{01} = {\mathcal T}_{01}= \{ |w_{1}| \leq \eta_{01} \}$.
The seed ${\mathcal T}_{01}$ is terminal since it only contains one irreducible component of
$\mathrm{Sing} X$.

Denote $w'=w/x$. We have
$X=x^{3} w' (w'-1) (xw'-1) \partial / \partial w'$
in coordinates $(x,w')$. Thus
$- \lambda^{3} w' (w'-1) \partial / \partial w'$
is the polynomial vector field $X_{00}(\lambda)$
associated to ${\mathcal T}_{00}$. The seed ${\mathcal T}_{00}$ contains an exterior set
${\mathcal E}_{00} = {\mathcal T}_{00} \cap \{ |w| \geq \rho_{00} |x| \}$ for
$\rho_{00}>>1$, a compact-like set
${\mathcal C}_{00} =
\{ |w'| \leq \rho_{00} \} \setminus ( \{ |w_{0}'|<\eta_{000} \} \cup \{ |w_{1}'|<\eta_{001} \} )$
and two terminal seeds ${\mathcal E}_{000}={\mathcal T}_{000}=\{ |w_{0}'| \leq \eta_{000} \}$
and ${\mathcal E}_{001}={\mathcal T}_{001}= \{ |w_{1}'| \leq \eta_{001} \}$ for some
$0 < \eta_{000},\eta_{001} <<1$. We have $e({\mathcal E}_{0})=0$,
$e({\mathcal C}_{0})=e({\mathcal E}_{00})=e({\mathcal E}_{01})=2$
and
$e({\mathcal C}_{00})=e({\mathcal E}_{000})=e({\mathcal E}_{001})=3$.
\end{exa}
\begin{rem}
This construction reminds the Fulton-MacPherson compactification
of the configuration space of $n$ distinct labeled points in a nonsingular algebraic variety $X$
\cite{Fulton-MacPherson}.
Indeed the analogue of the {\it seeds} are the Fulton-MacPherson's {\it screens}.
\end{rem}
\section{Dynamics of polynomial vector fields}
\label{sec:dpvf}
Given a vector field $X \in \Xt$ we divide a neighborhood of the origin in
exterior and compact-like basic sets.
The dynamics of $\Re (X)_{|x=x_{0}}$ in an exterior set is simple, namely an attractor,
a repeller or a  Fatou flower (see Section \ref{sec:dynext}).
The dynamics of $\Re (X)$ in compact-like sets determines the dynamics of
$\Re (X)$ in a neighborhood of the origin.  It turns out that the behavior of
$\Re (X)_{|{\mathcal C}}$ for a compact-like basic set ${\mathcal C}$
can be described in terms of
the polynomial vector field $X_{\mathcal C}(\lambda)$ associated to ${\mathcal C}$
(see Definition \ref{def:pol}). In this section we study polynomial vector fields and
their stability properties.

Polynomial vector fields have been studied by Douady, Estrada and Sentenac \cite{DES}.
We include here the main properties and their proofs for the sake of completeness.
\begin{defi}
Let $Y = P(w) \partial / \partial w \in {\mathbb C}[w] \partial / \partial w$ be
a polynomial vector field. We define $\nu (Y) = \deg (P) -1$.
We consider polynomial vector fields of degree greater than $1$.
\end{defi}
\begin{defi}
We define the set $Tr_{\to \infty}(Y)$
of trajectories $\gamma:(c,d) \to {\mathbb C}$ of $\Re(Y)$
such that $c \in {\mathbb R} \cup \{-\infty\}$, $d \in {\mathbb R}$
and $\lim_{\zeta \to d} \gamma(\zeta) = \infty$. In an analogous way
we define $Tr_{\leftarrow \infty}(Y)=Tr_{\to \infty}(-Y)$.
We define $Tr_{\infty}(Y) = Tr_{\leftarrow \infty}(Y) \cup Tr_{\to \infty}(Y)$.
\end{defi}
\begin{rem}
\label{rem:frinfty}
Let $Y = P(w) \partial / \partial w \in {\mathbb C}[w] \partial / \partial w$
with $\nu (Y) \geq 1$.
The vector field $Y$ is analytically conjugated to
$\nu(Y)^{-1} z^{1-\nu(Y)} \partial / \partial z$ in the neighborhood of
$\infty$ by a change of coordinates of the form
$w = c/z +O(1)$ for some $c \in {\mathbb C}$
(see \cite{JR:mod}, p. 348).
We have $Tr_{\to \infty}(\partial/\partial{z})= {\mathbb R}^{+}$
and $Tr_{\leftarrow \infty}(\partial/\partial{z})= {\mathbb R}^{-}$.
Since $\nu(Y)^{-1} z^{1-\nu(Y)} \partial / \partial z$ is equal to
$(z^{\nu(Y)})^{*} (\partial / \partial z)$
each of the sets $Tr_{\to \infty}(Y)$, $Tr_{\leftarrow \infty}(Y)$
contains $\nu (Y)$ trajectories of $\Re (Y)$ in a neighborhood of $\infty$.
\end{rem}
\begin{defi}
\label{def:angle}
The complementary of the set $Tr_{\infty}(Y)  \cup \{ \infty \}$
has $2 \nu(Y)$ connected components in the neighborhood of $w=\infty$.
Each of these components is called an {\it angle}, the boundary of an angle
contains exactly one $\to \infty$-trajectory and one
$\leftarrow \infty$-trajectory.
\end{defi}
\begin{defi}
\label{def:homtr}
We say that $\Re(Y)$ has $\infty$-connections or homoclinic trajectories if
$Tr_{\to \infty}(Y) \cap Tr_{\leftarrow \infty}(Y) \neq \emptyset$, i.e.
there exists a trajectory
$\gamma:(c_{-1},c_{1}) \to {\mathbb C}$
of $\Re(Y)$ such that $c_{-1},c_{1} \in {\mathbb R}$ and
$\lim_{\zeta \to c_{s}} \gamma(\zeta)=\infty$ for any $s \in \{-1,1\}$.
The notion of homoclinic trajectory has been introduced in \cite{DES}
for the study of deformations of elements of $\diff{1}{}$.
\end{defi}
\begin{defi}
We define the $\alpha$ and $\omega$-limits $\alpha^{Y}(P)$ and
$\omega^{Y}(P)$ respectively of a point $P \in {\mathbb C}$
by the vector field $\Re(Y)$. If $P \in Tr_{\to \infty}(Y)$
we denote $\omega^{Y}(P) = \{\infty\}$ whereas if
$P \in Tr_{\leftarrow \infty}(Y)$ we denote $\alpha^{Y}(P)=\{\infty\}$.
\end{defi}
\begin{lem}
\label{lem:infom}
Let $Y = P(w) \partial / \partial w$ be a polynomial vector field
such that $\nu(Y) \geq 1$. Then $\omega^{Y}(w_{0})=\{ \infty \}$
is equivalent to $w_{0} \in Tr_{\to \infty}(Y)$. Analogously
$\alpha^{Y}(w_{0})=\{ \infty \}$
is equivalent to $w_{0} \in Tr_{\leftarrow \infty}(Y)$
\end{lem}
\begin{proof}
The vector field $Y$ is a ramification of a regular vector field
in a neighborhood of $\infty$. Thus there exists an open neighborhood $V$
of $\infty$ and $c \in {\mathbb R}^{+}$ such that
\[ \mathrm{exp}(c Y)(V \setminus Tr_{\to \infty}(Y)) \cap V = \emptyset
\ \ \mathrm{and} \ \
\mathrm{exp}(-c Y)(V \setminus Tr_{\leftarrow \infty}(Y)) \cap V = \emptyset. \]
We are done since $w_{0} \not \in Tr_{\to \infty}(Y)$ implies
$\omega^{Y}(w_{0}) \cap (\pn{1} \setminus V) \neq \emptyset$.
\end{proof}
\begin{lem}
\label{lem:doml}
Let $Y = P(w) \partial / \partial w$ be a polynomial vector field
such that $\nu(Y) \geq 1$. Let $w_{1} \in {\mathbb C}$ be a point
such that $w_{1} \in \omega^{Y} (w_{0})$ for some $w_{0} \in {\mathbb C}$.
Then either $\omega^{Y}(w_{1}) = \{ \infty \}$ or $w_{1} \in \mathrm{Sing} (Y)$
and $\omega^{Y}(w_{0}) = \{w_{1}\}$ or $w_{1}$
belongs to a closed trajectory of $\Re (Y)$. Moreover, in the latter
case $w_{0}$ and $w_{1}$ belong to the same trajectory of
$\Re (Y)$.
\end{lem}
\begin{proof}
If $\omega^{Y} (w_{1})$ contains a singular point $\tilde{w}$
then $\omega^{Y} (w_{0}) = \{\tilde{w}\}$ and $w_{1}=\tilde{w}$ since
basins of attractions of attractors and parabolic points
are open sets.

Suppose that
$\omega^{Y} (w_{1}) \neq \{\infty\}$ and
$\omega^{Y} (w_{1}) \cap \mathrm{Sing} (Y) = \emptyset$. Then $\omega^{Y} (w_{1})$
contains a regular point $w_{2}$ of
$Y$ by Lemma \ref{lem:infom}.
Consider a transversal $T$ to $\Re (Y)$ passing through $w_{2}$.
Trajectories through points in $\alpha$ or $\omega$-limits intersect
connected transversals at most once (Proposition 2, p. 246 \cite{Hirsch-Smale}).
Thus $w_{2}$ is in the same trajectory $\gamma$ of $\Re (Y)$ as $w_{1}$
and $\gamma$ is a closed trajectory.
The neighborhood of a closed trajectory of $\Re (Y)$ is composed
by closed trajectories of the same period by the isolated zeros principle.
We deduce that $w_{0}$ and $w_{1}$ both belong to $\gamma$.
\end{proof}
\begin{cor}
\label{cor:trninf}
Let $Y = P(w) \partial / \partial w$ be a polynomial vector field
with $\nu(Y) \geq 1$. Assume that $\infty \not \in \omega^{Y} (w_{0})$.
Then either $\omega^{Y} (w_{0})$ is a singleton contained in $\mathrm{Sing} (Y)$
or $w_{0}$ belongs to a closed trajectory.
\end{cor}
\begin{proof}
Either $\omega^{Y} (w_{0})$ is a singleton or
there exists $w_{1} \in \omega^{Y} (w_{0}) \cap ({\mathbb C} \setminus \mathrm{Sing} (Y))$.
We have $\infty \not \in \omega^{Y} (w_{1})$, otherwise we obtain
$\infty \in \omega^{Y} (w_{0})$.
Lemma \ref{lem:doml} implies that either $w_{1}$ is singular and
$\omega^{Y} (w_{0}) = \{w_{1}\}$ or $w_{0}$ belongs to a closed trajectory.
\end{proof}
The previous corollary has an analogous version for elements of $\Xnt$.
\begin{defi}
\label{def:alom}
Let $X \in \Xnt$.
We define $\alpha^{X}(P)$ as the $\alpha$-limit of $\Gamma(X, P, {\mathcal T}_{0})$
for any $P \in B(0,\delta) \times \overline{B(0,\epsilon)}$
such that ${\mathcal I}(X,P,{\mathcal T}_{0})$
contains $(- \infty,0)$. Otherwise we define $\alpha^{X}(P)=\{ \infty \}$.
We define $\omega^{X}(P)$ in an analogous way.
\end{defi}
\begin{pro}
\label{pro:lxp1}
Let $X \in \Xnt$. Consider a point $P \in B(0,\delta) \times B(0,\epsilon)$
such that $\omega^{X}(P) \neq \{\infty\}$. Then either
$\omega^{X}(P)$ is a singleton contained in $\mathrm{Sing} (X)$ or
the trajectory of $\Re (X)$ through $P$ is closed.
\end{pro}
The proof is analogous to the proof of Corollary \ref{cor:trninf}.
The analogue for $P$ of the condition
$\infty \not \in \omega^{Y} (w_{0})$
is $\omega^{X}(P) \neq \{\infty\}$.

Our next goal is showing that the dynamics of the real flow of a
polynomial vector field of degree greater than $1$ is simple if
there are no homoclinic trajectories.
\begin{lem}
\label{lem:aolnht}
Let $Y = P(w) \partial / \partial w$ be a polynomial vector field
such that $\nu(Y) \geq 1$. Suppose that $\Re (Y)$ has no homoclinic trajectories.
Then either $\omega^{Y}(w_{0})=\{\infty\}$ or $\omega^{Y}(w_{0})$ is a singleton
contained in $\mathrm{Sing} (Y)$. In particular $\Re (Y)$ does not have periodic
trajectories.
\end{lem}
\begin{proof}
We claim that $\omega^{Y}(w_{0})$ can not contain a point
$w_{1} \in {\mathbb C} \setminus \mathrm{Sing} (Y)$ such that $\omega^{Y}(w_{1})=\{\infty\}$.
Otherwise there exists an angle $A$ containing points of
$\Gamma(Y,w_{0},{\mathbb C})[0,\infty)$ in every neighborhood of $\infty$.
The angle $A$ is limited by
a trajectory in $Tr_{\to \infty}(Y)$ and a trajectory $\gamma$ in
$Tr_{\leftarrow \infty}(Y)$.
Moreover $\gamma$ is contained in $\omega^{Y}(w_{0})$.
It satisfies $\gamma \cap Tr_{\to \infty}(Y) = \emptyset$ since
there are no homoclinic trajectories. We obtain a contradiction since
Lemma \ref{lem:doml} implies that
$\gamma$ is a closed orbit.

We obtain that $w_{0} \in Tr_{\to \infty}(Y)$ or $\omega^{Y}(w_{0}) \subset \mathrm{Sing} (Y)$ or
$w_{0}$ belongs to a closed orbit of $\Re (Y)$ for any $w_{0} \in {\mathbb C}$
by Lemma \ref{lem:doml}.

Let us prove that if $\Re (Y)$ has a closed trajectory
$\gamma: [0,a] \to {\mathbb C}$ ($\gamma(0)=\gamma(a)$)
we obtain a contradiction.
Let $D$ be the union of all the closed trajectories
of $\Re (Y)$. We denote by $D_{0}$ the connected component of
$D$ containing $\gamma[0,a]$.
Since $Tr_{\infty}(Y)  \cap D = \emptyset$
we can choose $w_{0} \in \partial D_{0} \setminus \mathrm{Sing} (Y)$.
If $\omega^{Y}(w_{0}) \subset \mathrm{Sing} (Y)$ or $w_{0}$ belongs to a closed trajectory
then the analogous property also holds true for the points in the neighborhood of
$w_{0}$ and $w_{0} \not \in \partial D_{0}$. We deduce that $w_{0}$ belongs to
$Tr_{\to \infty}(Y)$. Analogously $w_{0}$ is contained in $Tr_{\leftarrow \infty}(Y)$.
Thus $w_{0}$ belongs to a homoclinic trajectory.
\end{proof}
The next result is of technical type, it will be used in
the proof of Lemma \ref{lem:inst}.
\begin{lem}
\label{lem:sptcl}
Let $Y = P(w) \partial / \partial w$ be a polynomial vector field
such that $\nu(Y) \geq 1$. Suppose that $\Re (Y)$ has no homoclinic trajectories.
Let $w_{0} \in {\mathbb C} \setminus \mathrm{Sing} (Y)$ be a point such that
$\omega^{Y}(w_{0}) =\{w_{1}\}$ for some $w_{1} \in \mathrm{Sing} (Y)$.
Then there exists a trajectory $\gamma$ in $Tr_{\leftarrow \infty}(Y)$
such that $\omega^{Y}(\gamma) = \{w_{1}\}$.
\end{lem}
\begin{proof}
Let $D = \{ w \in {\mathbb C} \setminus \mathrm{Sing} (Y) : \omega^{Y}(w) =\{w_{1}\}\}$.
We denote by $D_{0}$ the connected component of
$D$ containing $w_{0}$.
Since $Tr_{\to \infty}(Y) \cap D = \emptyset$
we can choose $\tilde{w} \in \partial D_{0} \setminus \mathrm{Sing} (Y)$.
Lemma \ref{lem:aolnht} implies that $\tilde{w} \in Tr_{\to \infty}(Y)$
by the openness of basins of attraction of singular points.
Thus we have $\infty \in \overline{D_{0}}$.
The set $D_{0}$ contains an angle $A$. The angle $A$ is limited by
a trajectory in $Tr_{\to \infty}(Y)$ and a trajectory $\gamma$ in
$Tr_{\leftarrow \infty}(Y)$.
Lemma \ref{lem:aolnht} implies $\omega^{Y}(\gamma)=\{w_{1}\}$.
\end{proof}
Next we study the notion of stability of polynomial vector fields
as introduced in \cite{DES}.
It is crucial in the paper since the rigidity results for conjugacies
between elemens of $\dif{p1}{2}$
are obtained by analyzing the directions of
instability in the parameter space for unfoldings in $\Xntg$ and $\dif{p1}{2}$.
\begin{defi}
\label{def:noinfcon}
We denote by ${\mathcal X}_{stable}\cn{}$ the set of polynomial vector
fields such that $\nu(Y) \geq 1$ and
$\Re (Y)$ is orbitally equivalent to $\Re (\mu Y)$ for any
$\mu \in {\mathbb S}^{1}$ in a neighborhood of $1$.
\end{defi}
\begin{defi}
\label{def:res1}
Let $X$ be a holomorphic vector field defined in a connected domain
$U \subset {\mathbb C}$ such that $X \neq 0$. Consider $P \in \mathrm{Sing} (X)$.
There exists a unique meromorphic differential form
$\omega$ in $U$ such that $\omega(X)=1$. We denote by
$Res(X,P)$ the residue of $\omega$ at the point $P$.
\end{defi}
\begin{defi}
\label{def:res12}
Given $\phi \in \diff{1}{} \setminus \{Id\}$ we consider a convergent normal form
$X=g(y) \partial /\partial y$. We define
$Res_{\phi}(0) =Res (X,0)$. The definition does not depend on the choice of $X$.
Let $\varphi \in \dif{p1}{2}$; we define $Res_{\varphi}(0,0) = Res_{\varphi_{|x=0}}(0)$.
\end{defi}
\begin{defi}
\label{def:res2}
Let $X=f(x,y) \partial/\partial y \in \Xnt$.
Given $(x_{0},y_{0}) \in \mathrm{Sing} (X)$ such that
$\{x=x_{0}\} \not \subset \mathrm{Sing} (X)$ we define
$Res(X,(x_{0},y_{0}))=Res(f(x_{0},y) \partial/\partial y,y_{0})$.
\end{defi}

We introduce the main result in this section.
\begin{pro}
\label{pro:DES}
(See \cite{DES})
Let $Y = P(w) \partial / \partial w$ be a polynomial vector field
such that $\nu(Y) \geq 1$. Then $Y$ belongs to ${\mathcal X}_{stable}\cn{}$
if and only if $\Re (Y)$  has no homoclinic trajectories.
\end{pro}
\begin{proof}
Let $\Omega$ be the meromorphic $1$-form such that $\Omega (Y)=1$.
Let $\gamma:(c_{-1},c_{1}) \to {\mathbb C}$ be a homoclinic trajectory
of $\Re (\mu Y)$ (see Definition \ref{def:homtr}). We obtain
\[ {\mathbb R}^{+} \ni c_{1}-c_{-1} = \int_{\gamma(c_{-1},c_{1})} \frac{\Omega}{\mu} =
\frac{1}{\mu} 2 \pi i \sum_{\omega \in E} Res(Y, w) \]
where $E$ is the set of singular points of $Y$ enclosed by $\gamma$.
The set of directions $\mu \in {\mathbb S}^{1}$ such that
$2 \pi i \sum_{\omega \in E} Res(Y, w) \in {\mathbb R}^{*} \mu$ for some
subset $E$ of $\mathrm{Sing} (Y)$ is finite. Hence $Y$ is not stable if
it has a homoclinic trajectory.

The trajectories $\eta_{1,\mu}$, $\hdots$, $\eta_{2 \nu(Y),\mu}$
of $Tr_{\infty}(\mu Y)$
depend continuously on $\mu$.
Suppose that $\Re (Y)$ has no homoclinic trajectories.
Given a trajectory $\eta_{j,\mu}$ in $Tr_{\to \infty}(Y)$ (resp. $Tr_{\leftarrow \infty}(Y)$)
we have that $\alpha^{Y}(\eta_{j,\mu})$ (resp. $\omega^{Y}(\eta_{j,\mu})$) is a singleton contained
in $\mathrm{Sing} (Y)$ by Lemma \ref{lem:aolnht}.
Since basins of attractions are open sets then the previous limits do not depend on $\mu$
for $\mu \in {\mathbb S}^{1}$ in a neighborhood of $1$.
By extending an analytical conjugacy defined in a neighborhood of $\infty$
we obtain that there exists a topological orbital equivalence between
$\Re (Y)$ and $\Re (\mu Y)$ defined in $F$ where
$F(\mu) = Tr_{\to \infty}(\mu Y) \cup Tr_{\leftarrow \infty}(\mu Y) \cup \{\infty\} \cup \mathrm{Sing} (Y)$.
The set $F(\mu)$ depends continuously on $\mu$.
The functions $\alpha^{\mu Y}$ and $\omega^{\mu Y}$ are constant
in each component of $\pn{1} \setminus F(\mu)$.
Moreover
\[ \begin{array}{cccc}
(\alpha, \omega) : &
(\pn{1} \times V) \setminus \cup_{\mu \in V} (F(\mu) \times \{\mu\}) & \to & \mathrm{Sing} (Y) \times \mathrm{Sing} (Y) \\
& (w,\mu) & \mapsto & (\alpha^{\mu Y}(w), \omega^{\mu Y}(w))
\end{array}
\]
is constant in the connected components of
$(\pn{1} \times V) \setminus \cup_{\mu \in V} (F(\mu) \times \{\mu\})$
for some connected open neighborhood $V$ of $1$ in ${\mathbb S}^{1}$.
As a consequence the topological orbital equivalence can be extended to
$\pn{1}$.
\end{proof}
\begin{defi}
\label{def:plevels}
 Let $Y = P(w) \partial / \partial w$ be a polynomial vector field
with $\nu(Y) \geq 1$. We define
${\mathcal U}_{Y}^{1} = \left\{{ \lambda \in {\mathbb S}^{1}   : \lambda Y \not \in
{\mathcal X}_{stable} \cn{} }\right\}$.
\end{defi}
\begin{defi}
\label{def:levels}
  Let $X \in \Xnt$. Consider the compact-like sets
${\mathcal C}_{1}$, $\hdots$, ${\mathcal C}_{q}$ associated to $X$.
Let $X_{j}(\lambda) = \lambda^{e({\mathcal C}_{j})} P_{j}(w) \frac{\partial}{\partial w}$ be the
polynomial vector field associated to ${\mathcal C}_{j}$ and $X$ for $1 \leq j \leq q$.
We define
\[ {\mathcal U}_{X}^{j,1} = \left\{{ \lambda \in {\mathbb S}^{1}   : X_{j}(\lambda) \not \in
{\mathcal X}_{stable} \cn{} }\right\}, \ \
{\mathcal U}_{X}^{1} = \cup_{1 \leq j \leq q} {\mathcal U}_{X}^{j,1} . \]
\end{defi}
On the one hand the dynamics of $\Re (X)$ in an exterior set is trivial.
On the other hand the behavior of $\Re (X)$ in ${\mathcal C}_{j}$ is controlled
by the vector field $X_{j}(\lambda)$. Hence the dynamics of $\Re (X)$ is stable in
the neighborhood of the directions in
${\mathbb S}^{1} \setminus {\mathcal U}_{X}^{1}$.
\begin{defi}
\label{def:dir}
Let $A$ be a subset of ${\mathbb C} \setminus \{0\}$. Consider the blow-up mapping
$\pi : ({\mathbb R}^{+} \cup \{0\}) \times {\mathbb S}^{1} \to {\mathbb C}$
defined by $\pi (r,\lambda)=r \lambda$. We denote $A_{\pi}$ the subset
$\overline{\pi^{-1}(A)}$ of $({\mathbb R}^{+} \cup \{0\}) \times {\mathbb S}^{1}$.
We say that $A$ adheres the directions in $A_{\pi} \cap (\{0\} \times {\mathbb S}^{1})$.
\end{defi}
\begin{lem}
\label{lem:inst}
 Let $Y = P(w) \partial / \partial w$ be a polynomial vector field
with $\nu(Y) \geq 1$.
Then ${\mathcal U}_{Y}^{1} = \emptyset$ if and only if
$\sharp \mathrm{Sing} (Y) = 1$.
\end{lem}
\begin{proof}
If $\sharp \mathrm{Sing} (Y) = 1$ then we have $Y = w^{\nu(Y) +1} \partial / \partial w$
for some $\nu \geq 1$ up to an affine change of coordinates.
The vector fields $Y$ and $\mu Y$ are analytically conjugated by
a linear change of coordinates for any $\mu \in {\mathbb S}^{1}$.

Suppose ${\mathcal U}_{Y}^{1} = \emptyset$.
The dynamics of $\Re (\mu Y)$ depends continuously on $\mu$.
Consider the trajectories $\eta_{1,\mu}$, $\hdots$, $\eta_{2 \nu(Y),\mu}$
of $Tr_{\infty}(\mu Y)$ ordered in
counter clock wise sense. We can suppose that $\eta_{j,\mu} \in Tr_{\leftarrow \infty}(\mu Y)$
if $j$ is odd whereas $\eta_{j,\mu} \in Tr_{\to \infty}(\mu Y)$ if $j$ is even.
The trajectory $\eta_{j,e^{i \theta}}$ adheres to a direction
$\lambda_{0} e^{-i \theta / \nu} e^{j \pi i/\nu}$ for some $\lambda_{0} \in {\mathbb S}^{1}$
and all $\theta \in {\mathbb R}$ and $1 \leq j \leq 2 \nu$.
When we follow the path $e^{i \theta}$ for $\theta \in [0,\pi]$ the
direction $\lambda_{0}   e^{j \pi i/\nu}$ becomes $\lambda_{0} e^{-i \pi / \nu} e^{j \pi i/\nu}$
and $Y$ becomes $-Y$. Hence the trajectories $\eta_{j,1}$ and $\eta_{j+1,-1}$ are the same as sets.
We obtain
\[ \omega^{Y}(\eta_{j,1})  = \alpha^{Y}(\eta_{j+1,1}) \
\mathrm{for} \ j \ \mathrm{odd} \ \mathrm{and} \
\alpha^{Y}(\eta_{j,1}) = \omega^{Y}(\eta_{j+1,1}) \
\mathrm{for} \ j \ \mathrm{even} . \]
This implies that there exists $w_{1} \in \mathrm{Sing} (Y)$
such that $\omega^{Y}(\eta_{j,1})=\{w_{1}\}$ if $j$ is odd and
$\alpha^{Y}(\eta_{j,1})=\{w_{1}\}$ if $j$ is even.

We claim that $w_{1}$ is the unique point of $\mathrm{Sing} (Y)$.
Consider $w_{2} \in \mathrm{Sing} (Y)$. The point $w_{2}$ is a parabolic point,
i.e $P' (w_{2})=0$. Otherwise there exists $\mu_{0}$ such that
$w_{2}$ is a center of $\Re (\mu_{0} Y)$ and $\mu_{0}$ belongs to
${\mathcal U}_{Y}^{1}$ by Lemma \ref{lem:aolnht}.
The set
$G = \{ w \in {\mathbb C} \setminus \mathrm{Sing} (Y) : \omega^{Y}(w) = \{w_{2}\} \} \cap Tr_{\leftarrow \infty}(Y)$
is non-empty by Lemma \ref{lem:sptcl}.
The discussion in the previous paragraph implies $w_{2}=w_{1}$.
\end{proof}

\begin{cor}
\label{cor:exir}
Let $X \in \Xnt$. Then ${\mathcal U}_{X}^{j,1} = \emptyset$ iff
$\sharp \mathrm{Sing} (X_{j}(1)) = 1$.
Moreover we have ${\mathcal U}_{X}^{1} = \emptyset$ iff $N(X)=1$.
\end{cor}
Corollary \ref{cor:exir} implies that there are instability phenomena for
any $X =g(x,y) \partial /\partial y$ in $\Xntg$ with $N(X) \neq 1$.
The remaining case $N(X) = 1$
is a topological product.
More precisely $\Re (X)$ is topologically conjugated to
$\Re (g(0,y) \partial /\partial y)$ by a mapping of the form
$\sigma(x,y) = (x,\tilde{\sigma}(x,y))$.
\begin{defi}
\label{def:unstext}
Let $X \in \Xnt$. Consider a non-parabolic exterior set
${\mathcal E}$. We define
${\mathcal U}_{X}^{{\mathcal E},1}=
\{ \lambda \in {\mathbb S}^{1} : \lambda^{e({\mathcal E})} h(0) \in i {\mathbb R}\}$
where
$X = x^{e({\mathcal E})}
h(x) t \partial / \partial{t}$
in adapted coordinates in ${\mathcal E}$.
We define ${\mathcal U}_{X}^{{\mathcal E},1} = \emptyset$ if ${\mathcal E}$
is a parabolic exterior set.
\end{defi}
\begin{rem}
 Let us stress that if ${\mathcal C}_{j}$ is the compact-like set enclosing
the non-parabolic exterior set ${\mathcal E}$ then
${\mathcal U}_{X}^{{\mathcal E},1} \subset {\mathcal U}_{X}^{j,1}$.
The point $t=0$ is an attractor in
$\{ x \in B(0,\delta) \setminus \{0\} : Re(\lambda^{e({\mathcal E})} h(x)) <0 \}$
where $x=|x| \lambda$.
\end{rem}
Let us analyze the dynamics of $\Re (X)$ in the stable directions.
The next result is Lemma 6.13 of \cite{JR:mod}.
\begin{lem}
\label{lem:goins}
Let $X \in \Xntg$. Let $K$ be a compact subset of
${\mathbb S}^{1} \setminus {\mathcal U}_{X}^{1}$.
Then
$[0,\infty)$ is contained in
${\mathcal I}(X, P_{0},{\mathcal T}_{0}^{\epsilon})$
and $\lim_{\zeta \to \infty} \mathrm{exp}(\zeta X)(P_{0}) \in \mathrm{Sing} (X)$
for any
$P_{0} \in [0,\delta(\epsilon,K)) K \times \partial{B(0,\epsilon)}$
such that $\Re (X)$ does not point towards
${\mathbb C}^{2} \setminus {\mathcal T}_{0}^{\epsilon}$ at $P_{0}$.
Moreover,
there exists a dynamical splitting $\digamma_{K}$ such that
the intersection of $\mathrm{exp}((0,\infty) X)(P_{0})$ with
every compact-like or exterior set is connected for any
$P_{0} \in [0,\delta(\epsilon,K)) K \times \partial{B(0,\epsilon)}$.
\end{lem}
Let us remark that the dynamical splitting $\digamma_{K}$
depends on $K$ but
it does not depend on $\epsilon$.
Indeed  stable behavior degrades as we approach the directions in
${\mathcal U}_{X}^{1}$ in the parameter space.
Therefore a unique dynamical splitting does not satisfy
the result in the theorem for any compact set
$K \subset {\mathbb S}^{1} \setminus {\mathcal U}_{X}^{1}$.
On the other hand instability of the dynamics of $\Re (X)$ is related
to a finite set of data, namely the polynomial vector fields
restricted to the finitely many directions in ${\mathcal U}_{X}^{1}$.
Hence we can choose a unique dynamical splitting to describe instability
phenomena.
\begin{cor}
\label{cor:stdir}
Let $X \in \Xntg$. Let $K$ be a compact subset of ${\mathbb S}^{1} \setminus {\mathcal U}_{X}^{1}$.
Given any $\epsilon>0$ small there exists $\delta_{0}=\delta_{0}(\epsilon,K) \in {\mathbb R}^{+}$
such that
\begin{itemize}
\item There is no $P \in [0,\delta_{0}) K \times B(0,\epsilon)$
such that $\alpha^{X}(P)=\omega^{X}(P) = \{\infty\}$.
\item $\alpha^{X}(P) \subset \mathrm{Sing} (X)$ or
$\omega^{X}(P) \subset \mathrm{Sing} (X)$ for any $P \in [0,\delta_{0}) K \times B(0,\epsilon)$.
\item There are no closed trajectories or centers of $\Re (X)$ in
$[0,\delta_{0}) K \times B(0,\epsilon)$.
\end{itemize}
\end{cor}
\begin{proof}
Let $\delta_{0} \in {\mathbb R}^{+}$ be the constant provided by Lemma \ref{lem:goins}.
Suppose there exists $P \in [0,\delta_{0}) K \times B(0,\epsilon)$ 
such that $\alpha^{X}(P) = \{\infty\}$.
Consider $(a,b) = {\mathcal I}(X, P, U_{\epsilon})$ (see Definition \ref{def:ue}).
Let $P_{0} = \Gamma (X, P, {\mathcal T}_{0})(a)$.
Lemma \ref{lem:goins} implies $b = \infty$
and $\omega^{X}(P_{0}) \subset \mathrm{Sing} (X)$.
Analogously $\omega^{X}(P) = \{\infty\}$ implies
$\alpha^{X}(P) \subset \mathrm{Sing} (X)$.

It remains to consider the case $\alpha^{X}(P) \neq \{\infty\}$,
$\omega^{X}(P) \neq \{\infty\}$.
We have that either both $\alpha^{X}(P)$, $\omega^{X}(P)$ are contained
in $\mathrm{Sing} (X)$ or $P$ belongs to a closed trajectory $\gamma$ by Proposition \ref{pro:lxp1}.
Suppose that $\gamma$ is contained in $x=x_{0}$.
Consider the union $D$ of closed trajectories of $\Re (X)$
in $\{x_{0}\} \times B(0,\epsilon)$. Let
$(x_{0},y_{0})$ be a point in
$(\partial D_{0} \setminus \mathrm{Sing} (X)) \cap (\{x_{0}\} \times   B(0,\epsilon))$ where
$D_{0}$ is the connected component of $D$ containing $\gamma$.
The point $(x_{0},y_{0})$ satisfies
$\alpha^{X}(x_{0},y_{0})=\omega^{X}(x_{0},y_{0}) = \{\infty\}$.
This contradicts the first part of the corollary.
\end{proof}
\section{Dynamics in exterior basic sets}
\label{sec:dynext}
We study the dynamics of diffeomorphisms $\varphi \in \diff{p1}{2}$.
The idea is comparing the dynamics of $\varphi$ with the dynamics of
a convergent normal form ${\rm exp}(X)$ (see Definition \ref{def:normal}).
This section is intended to describe the
behavior of $\Re (X)$ in exterior sets.

The main technique to prove the results in the paper is the study of
special orbits for $\varphi \in \dif{p1}{2}$. Part of the proof is
showing that such orbits are close to
trajectories of $\Re (X)$ where $X$ is a normal form of $\varphi$.
In order to compare the orbits of $\varphi$
and ${\mathfrak F}_{\varphi}=\mathrm{exp}(X)$ we need some estimates that
are introduced below.
\begin{defi}
Let $X \in \Xnt$ and $P \in U_{\epsilon}$. Given $M \in {\mathbb R}^{+}$ we
define $B_{X}(P,M) = \cup_{z \in B(0,M)} \{ \mathrm{exp}(z X)(P) \}$.
\end{defi}
\subsection{Exterior sets}

\begin{defi}
\label{def:psiext}
Let $X \in \Xnt$.
Let ${\mathcal E} = \{   \eta \geq |t| \geq \rho|x| \}$
be an   exterior set associated to a seed ${\mathcal T}$.
The vector field $X_{\mathcal E}$ is of the form
\[ X_{\mathcal E} =v(x,t) (t-\gamma_{1}(x))^{s_{1}} \hdots (t-\gamma_{p}(x))^{s_{p}}
\partial / \partial{t} \]
where $v$ is a function never vanishing in ${\mathcal T}$. Denote $\gamma_{\mathcal E}=\gamma_{1}$.
Denote $\psi_{\mathcal E}^{0}$  a Fatou coordinate   of
$v(0,t-\gamma_{\mathcal E}(x)) (t-\gamma_{\mathcal E}(x))^{\nu({\mathcal E})+1}
\partial / \partial{t}$
defined in the neighborhood of ${\mathcal E} \setminus \mathrm{Sing} X$.
\end{defi}
The idea behind the definitions in this section is that the
dynamics of  the vector fields $X_{\mathcal E}$ and
$v(0,t-\gamma_{\mathcal E}(x)) (t-\gamma_{\mathcal E}(x))^{\nu({\mathcal E})+1}
\partial / \partial{t}$ in ${\mathcal E}$ are analogous. But the latter vector field is much
simpler since it is conjugated to
$v(0,t) t^{\nu({\mathcal E})+1} \partial / \partial t$, that does not depend on $x$,
by a diffeomorphism  $(x,t+\gamma_{\mathcal E}(x))$.
\begin{rem}
\label{rem:psiext0}
Suppose $\nu({\mathcal E})=0$.
We have $X_{\mathcal E} = h(x) t \partial / \partial t$.
The function $\psi_{\mathcal E}^{0}$ is of the form
\[ \psi_{\mathcal E}^{0} =
\frac{1}{h(0)} \ln t + b(x) \]
where $b(x)$ is a holomorphic function in the neighborhood of $0$.
The Fatou coordinates are useful to study trajectories of $\Re (X)$
intersecting the boundaries of the basic sets.
We will use determinations of $\psi_{\mathcal E}^{0}$ that are bounded by above in
the exterior boundary of ${\mathcal E}$.
\end{rem}
\begin{rem}
\label{rem:psiext}
Suppose $\nu({\mathcal E})>0$.
The function $\psi_{\mathcal E}^{0}$ is of the form
\[ \psi_{\mathcal E}^{0} =
\frac{-1}{\nu({\mathcal E}) v(0,0)} \frac{1}{(t-\gamma_{\mathcal E}(x))^{\nu({\mathcal E})}}
+ Res(X_{\mathcal E},(0,0)) \ln (t-\gamma_{\mathcal E}(x)) + h(t-\gamma_{\mathcal E}(x)) + b(x) \]
where $h(z)$ is a $O(1/z^{\nu({\mathcal E})-1})$ meromorphic function
and $b(x)$ is a holomorphic function in the neighborhood of $0$.
We make analogous choices of determinations of $\psi_{\mathcal E}^{0}$
as in the case $\nu({\mathcal E})=0$. We obtain that
given $\zeta>0$
there exists $C_{\zeta} \in {\mathbb R}^{+}$ such that
\[ \frac{1}{C_{\zeta}}  \frac{1}{|t-\gamma_{\mathcal E}(x)|^{\nu({\mathcal E})}} \leq
|\psi_{\mathcal E}^{0}|(x,t) \leq  C_{\zeta} \frac{1}{|t-\gamma_{\mathcal E}(x)|^{\nu({\mathcal E})}} \]
in ${\mathcal E} \cap \{t - \gamma_{\mathcal E}(x) \in {\mathbb R}^{+} e^{i[-\zeta, \zeta]}\}
\cap \{x \in B(0,\delta(\zeta)) \}$.
\end{rem}
\begin{defi}
\label{def:psiE}
Let ${\mathcal E} = \{   \eta \geq |t| \geq \rho|x| \}$
be a parabolic exterior set associated to $X \in \Xnt$.
Denote by $\psi_{\mathcal E}$  a Fatou coordinate of
$X_{\mathcal E}$ defined in the neighborhood of ${\mathcal E} \setminus \mathrm{Sing} X$ such that
$\psi_{\mathcal E}(0,\lambda,t) \equiv \psi_{\mathcal E}^{0}(0,\lambda,t)$.
The function $\psi_{\mathcal E}$ is multi-valued.
\end{defi}
\subsubsection{Non-parabolic exterior sets}
The next lemma is used in Proposition \ref{pro:boufespre} to show
that far away from ${\mathcal U}_{X}^{{\mathcal E},1}$ (see Definition \ref{def:unstext})
the orbits of $\varphi$ track, i.e. are very close to, orbits of the
normal form ${\mathfrak F}_{\varphi}$.
The lemma will be applied to the function $\Delta_{\varphi}$ (see Definition \ref{def:delta})
that measures how good is the approximation of $\varphi$ provided by ${\mathfrak F}_{\varphi}$.
\begin{lem}
\label{lem:cansumnp}
Let $X \in \Xnt$ with $N \geq 1$.
Fix a non-parabolic exterior set ${\mathcal E}$.
Consider a function $\Delta=O({x}^{a} t)$ where
$X = x^{e({\mathcal E})}
h(x) t \partial / \partial{t}$
in adapted coordinates. Fix
a closed subset $S$ of
$\{ \lambda \in {\mathbb S}^{1} : Re(\lambda^{e({\mathcal E})} h(0)) <0 \}$
and $b \in {\mathbb N}$.
Then there exists $C \in {\mathbb R}^{+}$ such that
$|\Delta| \leq C {|x|}^{a}/ |\psi_{{\mathcal E}}|^{b}$ in
a neighborhood of any trajectory $\gamma$ of $\Re (X)$ in
${\mathcal E} \cap \{ x \in (0,\delta) S\}$.
\end{lem}
In Lemma \ref{lem:cansumnp}  we consider neighborhoods of the form
$B_{X}(\gamma,M) = \cup_{P \in \gamma} B_{X}(P,M)$ for some
a priori fixed $M \in {\mathbb R}^{+}$, for instance $M=1$.
In such neighborhoods the function $\psi_{\mathcal E}$ is uni-valuated.
\begin{proof}
We have
\[  X_{\mathcal E} =
h(x) t \frac{\partial}{\partial{t}}
\implies \psi_{\mathcal E} = \frac{\ln t}{h(x)}  .  \]
We obtain
$\Delta = O({x}^{a} e^{h(x) \psi_{\mathcal E}})$. We have
($x = r \lambda$ with $r \in {\mathbb R}^{+} \cup \{0\}$ and $\lambda \in {\mathbb S}^{1}$)
\[ (h(x)  \psi_{\mathcal E}) \circ \mathrm{exp}(z X)(x,y)  =
(h(x)  \psi_{\mathcal E})(x,y) +
r^{e({\mathcal E})} z h(x) {\lambda}^{e({\mathcal E})}  \]
for any $z \in {\mathbb C}$. The trajectories of $\Re (X)$ are obtained by
considering $z \in {\mathbb R}$. We obtain
\[ | e^{h(x) \psi_{\mathcal E}} | =
e^{Re(h(x) \psi_{\mathcal E})} \leq
C_{1} e^{- C_{2} |\psi_{\mathcal E}|} \]
along any trajectory of $\Re (X)$. The constants
$C_{1},C_{2} \in {\mathbb R}^{+}$ do not depend on the trajectory
because of the choice of $\psi_{\mathcal E}$ (see Remark
\ref{rem:psiext0} and Definition \ref{def:psiE}). It is clear
that $e^{-C_{2} |\psi_{\mathcal E}|} = O(1/\psi_{\mathcal E}^{b})$.
\end{proof}
\subsection{Parabolic exterior sets}
Let $X \in \Xnt$. Consider a parabolic exterior set ${\mathcal E}$.
The  approximation of
$X_{\mathcal E}$ with
$v(0,t-\gamma_{\mathcal E}(x)) (t-\gamma_{\mathcal E}(x))^{\nu({\mathcal E})+1}
\partial / \partial{t}$ is accurate.
\begin{lem}
(Lemma 6.5 \cite{JR:mod})
\label{lem:itf}
Let ${\mathcal E} = \{(x,t) \in B(0,\delta) \times {\mathbb C} :  \eta \geq |t| \geq \rho|x| \}$
be a parabolic exterior set associated to $X \in \Xnt$.
Let $\upsilon>0$, $\zeta>0$. Suppose ${\mathcal E}$ is terminal.
Then $|\psi_{\mathcal E}/\psi_{\mathcal E}^{0} -1| \leq \upsilon$ in
${\mathcal E} \cap \{t - \gamma_{\mathcal E}(x) \in {\mathbb R}^{+} e^{i[-\zeta, \zeta]}\}
\cap \{x \in B(0,\delta(\upsilon,\zeta)) \}$ for some $\delta(\upsilon,\zeta) \in {\mathbb R}^{+}$.
The same inequality is true for a non-terminal ${\mathcal E}$ if $\rho >0$ is big enough.
\end{lem}
The behavior of a multi-transversal flow in a parabolic exterior set is also
analogous to a Fatou flower from a quantitative point of view. In particular we prove
that the spiraling behavior is bounded in exterior basic sets.
\begin{pro}
\label{pro:estext}
Let $X \in \Xnt$ and let
${\mathcal E} = \{ \eta \geq |t| \geq \rho |x| \}$ be a parabolic exterior set
associated to $X$.  Consider a trajectory
$\Gamma =
\Gamma( \lambda^{e({\mathcal E})}   X_{\mathcal E},(r, \lambda,t),{\mathcal E})$
for $r \lambda$ in a neighborhood of $0$. Then $\Gamma$ is contained
in a sector centered at $t=\gamma_{\mathcal E}(r \lambda)$ (see Definition \ref{def:psiext}) of angle less
than $\zeta$ for some $\zeta>0$ independent of $r, \lambda$ and $\Gamma$.
\end{pro}
Let us explain the statement. Consider the universal covering
\[ (r,\lambda,\gamma_{\mathcal E}(r \lambda) + e^{z}):
 {\mathcal E}^{\flat} \to {\mathcal E} \setminus \mathrm{Sing} X . \]
Let ${\Gamma}^{\flat}$ the lifting of $\Gamma$ by $(r,\lambda,e^{z})$.
We claim that the set $(Im(z))({\Gamma}^{\flat})$ is contained in an interval of length $\zeta$.
\begin{proof}
We have
\[ X = x^{e({\mathcal E})} v(x,t) (t-\gamma_{1}(x))^{s_{1}} \hdots (t-\gamma_{p}(x))^{s_{p}}
\partial / \partial{t}  \]
where we consider $\gamma_{\mathcal E} \equiv \gamma_{1}$.
We denote
\[ \psi_{\mathcal E}^{00} =
\frac{-1}{\nu({\mathcal E}) v(0,0)} \frac{1}{(t-\gamma_{\mathcal E}(x))^{\nu({\mathcal E})}} =
\frac{-1}{\nu({\mathcal E}) v(0,0)} \frac{1}{(t-\gamma_{1}(x))^{\nu({\mathcal E})}}. \]
Given $\upsilon>0$ and $\zeta_{0}>0$ we can consider $\eta>0$ small to obtain that
$|\psi_{\mathcal E}^{0}/\psi_{\mathcal E}^{00} -1| < \upsilon$ in the set
$\{ (x,t) \in {\mathcal E} :  |\arg(t-\gamma_{\mathcal E}(x))| \leq \zeta_{0} \}$
(see Remark \ref{rem:psiext}).
Therefore we obtain
$|\psi_{\mathcal E}/\psi_{\mathcal E}^{00} -1| < \upsilon$ in
$\{ (x,t) \in {\mathcal E} :  |\arg(t-\gamma_{1}(x))| \leq \zeta_{0} \}$
by considering $\eta >0$ small enough and $\rho >0$ big enough if ${\mathcal E}$ is not terminal
(Lemma \ref{lem:itf}).

We have that either
\[ (\psi_{\mathcal E}/\lambda^{e({\mathcal E})})(\Gamma) \cap i ({\mathbb R}^{+} \cup \{0\}) = \emptyset \
\mathrm{or} \ (\psi_{\mathcal E}/\lambda^{e({\mathcal E})})(\Gamma) \cap i ({\mathbb R}^{-} \cup \{0\}) = \emptyset . \]
Thus $(\psi_{\mathcal E}/\lambda^{e({\mathcal E})})(\Gamma)$ lies in a sector of angle of angle $2 \pi$.
Since $\psi_{\mathcal E} / \psi_{\mathcal E}^{00} \sim 1$ then $\Gamma$  lies in a sector of
center $t=\gamma_{\mathcal E}(r, \lambda)$ and angle close to $2 \pi/\nu({\mathcal E})$.
\end{proof}
The next result plays an analogous role as Lemma \ref{lem:cansumnp}
for parabolic exterior sets.
\begin{lem}
\label{lem:cansum}
Let $X \in \Xnt$ with $N \geq 1$.
Fix a parabolic exterior set ${\mathcal E}$.
Consider a function $\Delta=O({x}^{a} {f'}^{b})$ where
$X = x^{e({\mathcal E})} f' \partial / \partial{t}$
in adapted coordinates.
Then $\Delta$ is of the form $O({x}^{a}/ \psi_{{\mathcal E}}^{b})$ in ${\mathcal E}$.
\end{lem}
\begin{proof}
Let
\[
 X = x^{e({\mathcal E})}
v(x,t) (t-\gamma_{1}(x))^{s_{1}} \hdots (t-\gamma_{p}(x))^{s_{p}}
\frac{\partial}{\partial{t}}
\]
be the expression of $X$ in adapted coordinates in the exterior set ${\mathcal E}$.
Let   $\nu = \nu_{{\mathcal E}} (X)$.
We obtain $\Delta = O({x}^{a} {(t-\gamma_{1}(x))}^{b (\nu+1)})$ in ${\mathcal E}$.
We have
\[ \psi_{{\mathcal E}} \sim \psi_{{\mathcal E}}^{0} \sim 1 /(t-\gamma_{1}(x))^{\nu} \]
by Lemma  \ref{lem:itf}. Let us remark that the previous property holds true
at a sector centered at $t=\gamma_{1}(r \lambda)$ and angle uniformly
bounded.
We obtain $\Delta = O({x}^{a}/\psi_{{\mathcal E}}^{be})$ for $e = (\nu+1)/\nu$
by Proposition \ref{pro:estext}.
\end{proof}
\section{Long Trajectories}
\label{sec:lt}
\begin{defi}
Consider a subset $\beta$ of ${\mathbb C}$. We say that
$\upsilon: \beta \to {\mathbb C}^{2}$ is a section if
there exists a continuous function
$\tilde{\upsilon}: \beta \to {\mathbb C}$
such that $\upsilon(x) \equiv (x, \tilde{\upsilon}(x))$.
\end{defi}
The proof of Theorem \ref{teo:main} depends on
the instability properties of elements $\varphi$ of
$\dif{p1}{2}$. Roughly speaking, we consider sections
$\upsilon: \beta \cup \{0\} \to U_{\epsilon}$, where
$\beta$ is a connected set with $0 \in \overline{\beta}$,
such that the limit of the orbits of $\varphi$ through
$\upsilon(x)$ splits in two orbits in the limit.
One of the orbits is obviously the orbit through $\upsilon(0)$
whereas the other orbit is composed of points $(0,y_{-})$
such that to go from $\upsilon(x)$ to a neighborhood of $(0,y_{-})$
we have to iterate $\varphi$ a number of times $T(x)$ that tends to
$\infty$ when $x \to 0$. These so called Long Orbits appear when
$N(\varphi) >1$ even if for simplicity we only consider the case
$N(\varphi) >1$, $m(\varphi)=0$. Anyway, this
is a non-generic phenomenon since the parameters containing Long Orbits adhere
directions of ${\mathcal U}_{X}^{1}$ (Proposition \ref{pro:udlo}).
Next we introduce the rigorous definition of Long Orbits and its
analogue (Long Trajectories) for vector fields.

Let $X \in \Xntg$ be a vector field defined in $B(0,\delta) \times B(0,\epsilon)$.
Consider a set  $\beta \subset {\mathbb C}^{*}$ adhering $0 \in {\mathbb C}$
and a point $y_{+} \in B(0,\epsilon) \setminus \{0\}$ such that
$\omega^{X} (0,y_{+}) = \{(0,0)\}$.  Assume that
$T: \beta \to {\mathbb R}^{+}$ is a continuous function
such that $\lim_{x \in \beta, \ x \to 0} T(x) = \infty$.
We are interested in describing the limit of
the trajectory
$\Gamma (X, (x,y_{+}), U_{\epsilon})$ when $x \in \beta$ and $x \to 0$.
\begin{defi}
\label{def:wlt}
We say that ${\mathcal O}=(X, y_{+},\beta,T)$ generates a {\it weak Long Trajectory}
if there exist a submersion $\vartheta_{\mathcal O}: \beta \to {\mathcal S}_{\mathcal O}$ where
${\mathcal S}_{\mathcal O}$ is a connected subset of ${\mathbb R}$ containing $0$ and a
section
$\upsilon_{\mathcal O}: \beta \cup \{0\} \to U_{\epsilon}$
with $\upsilon_{\mathcal O}(0)=(0,y_{+})$ such that
\begin{itemize}
\item $\vartheta_{\mathcal O}^{-1}(s)$ is a germ of connected curve for any $s \in {\mathcal S}_{\mathcal O}$.
\item $[0,T(x)] \subset {\mathcal I}(\Gamma (X, \upsilon_{\mathcal O}(x), U_{\epsilon}))$
for any $x \in \beta$.
\item Given a compact subset $K$ of ${\mathcal S}_{\mathcal O}$
there exists $\epsilon_{K}>0$ such that
\[ \Gamma (X, \upsilon_{\mathcal O}(x), U_{\epsilon})(T(x)) \not \in
U_{\epsilon_{K}} \]
for any $x \in \vartheta_{\mathcal O}^{-1}(K)$ close to $0$.
\item Given
any $\epsilon''>0$ there exists $M \in {\mathbb N}$ such that
\[ \Gamma (X, \upsilon_{\mathcal O}(x), U_{\epsilon})[M,T(x)-M] \subset
U_{\epsilon''} \]
for any $x \in \beta$ in a neighborhood of $0$.
\end{itemize}
\end{defi}
\begin{defi}
\label{def:lt}
We say that ${\mathcal O}=(X, y_{+},\beta,T)$ generates a {\it Long Trajectory}
if ${\mathcal O}$ generates a weak Long Trajectory and
there exists a continuous $\chi_{\mathcal O}:i {\mathcal S}_{\mathcal O} \to U_{\epsilon}(0)$
(see Definition \ref{def:ue})
such that
\[ \chi_{\mathcal O}(i s) =
\lim_{x \in \beta, \ \vartheta_{\mathcal O}(x) \to s, \ \ x \to 0}
\Gamma (X, (x,y_{+}), U_{\epsilon})(T(x)) \]
for any $s \in {\mathcal S}_{\mathcal O}$ and
$\chi_{\mathcal O}(i s)  = \mathrm{exp}(i s X)(\chi_{\mathcal O}(0))$
for any $s \in {\mathcal S}_{\mathcal O}$.
\end{defi}
Fix $s \in {\mathcal S}_{\mathcal O}$ for
a Long Trajectory ${\mathcal O}$.
The trajectory
$\Gamma (X, (x,y_{+}), U_{\epsilon})[0,T(x)]$ converges to
$\Gamma (X, (0,y_{+}), U_{\epsilon})[0,\infty) \cup \{(0,0) \} \cup
\Gamma (X, \chi_{\mathcal O}(i s), U_{\epsilon})(-\infty,0]$
when $x \in \vartheta_{\mathcal O}^{-1}(s)$ tends to $0$.
We consider the Hausdorff topology for compact sets.
Moreover $\Gamma (X, \chi_{\mathcal O}(i s), U_{\epsilon})$
describes all trajectories in a petal of $\Re (X)_{|U_{\epsilon}(0)}$ if
${\mathcal S}_{\mathcal O}={\mathbb R}$.

Consider a weak Long Trajectory. An accumulation
point of  $\Gamma (X, (x,y_{+}), U_{\epsilon})[0,T(x)]$ when
$x \in \vartheta_{\mathcal O}^{-1}(s)$ tends to $0$ is a union
of two trajectories of $\Re (X)_{|U_{\epsilon}(0)}$ and the origin by the last two properties of
Definition \ref{def:wlt}. Anyway the accumulation set is
not necessarily unique. The definition of Long Trajectory was introduced in
\cite{rib-mams}. There the section $\upsilon_{\mathcal O}$ is defined in a curve
denoted by $\beta \cup \{0\}$ in \cite{rib-mams} and whose analogue in this paper
is $\vartheta_{\mathcal O}^{-1}(0) \cup \{0\}$. Then the evolution of the Long
Trajectories is studied when that curve varies in a family as
${\{\vartheta_{\mathcal O}^{-1}(s)\}}_{s \in {\mathcal S}_{\mathcal O}}$.
In this paper we can deal with all the curves simultaneously since the
proofs have been improved by using the properties of polynomial vector fields.
\begin{rem}
\label{rem:udll}
Corollary \ref{cor:stdir} implies the non-existence of weak Long Trajectories ${\mathcal O}$
such that ${\mathcal S}_{\mathcal O}$ is compact and $\beta$ adheres directions in
${\mathbb S}^{1} \setminus {\mathcal U}_{X}^{1}$.
More precisely we have
$\lim_{n \to \infty} \Gamma (X, \upsilon_{\mathcal O}(x_{n}), U_{\epsilon})(T(x_{n}))=(0,0)$
for any sequence $x_{n}$ in $\beta$ such that $x_{n} \to 0$ and $x_{n}/|x_{n}|$ converges to
a point in ${\mathbb S}^{1} \setminus {\mathcal U}_{X}^{1}$.
Thus any set $\beta$ supporting a weak Long Trajectory with ${\mathcal S}_{\mathcal O}$ compact
adheres to a  unique direction in the finite set
${\mathcal U}_{X}^{1}$.
\end{rem}
Next we introduce the analogue of Long Trajectories for diffeomorphisms.
Let $\varphi \in \diff{p1}{2}$. Let $y_{+} \neq 0$ be a point such that
$\varphi^{j}(0,y_{+})$ is well-defined and belongs to $U_{\epsilon}$ for any $j \in {\mathbb N}$
and $\lim_{j \to \infty} \varphi^{j}(0,y_{+}) = (0,0)$. Consider a germ of set
$\beta \subset {\mathbb C}^{*}$ at $0$ and a continuous function
$T: \beta \to {\mathbb R}^{+}$ with $\lim_{x \in \beta, \ x \to 0} T(x) = \infty$.
We denote by $[s]$ and $\lceil s \rceil$ the integer part and the ceiling of $s \in {\mathbb R}$
respectively. Let us remind that
$\lceil s \rceil$ is the smallest integer not less than $s$.
\begin{defi}
\label{def:lo}
We say that ${\mathcal O}=(\varphi, y_{+},\beta,T)$ generates a Long Orbit if
there exist a submersion $\vartheta_{\mathcal O}: \beta \to {\mathcal S}_{\mathcal O}$ where
${\mathcal S}_{\mathcal O}$ is a connected subset of ${\mathbb R}$ containing $0$ and
continuous $\upsilon_{\mathcal O}: \beta \cup \{0\} \to U_{\epsilon}$
and $\chi_{\mathcal O}: [0,1]+ i {\mathcal S}_{\mathcal O} \to U_{\epsilon}(0) \setminus \{(0,0)\}$
such that
\begin{itemize}
\item $\vartheta_{\mathcal O}^{-1}(s)$ is a germ of connected curve for any $s \in {\mathcal S}_{\mathcal O}$.
\item $\upsilon_{\mathcal O}$ is a section, $\upsilon_{\mathcal O}(0)=(0,y_{+})$
and $\chi_{\mathcal O} (1+is) = \varphi(\chi_{\mathcal O}(is))$ for any $s \in {\mathcal S}_{\mathcal O}$.
\item $\varphi^{j}(\upsilon_{\mathcal O}(x))$ is well-defined and belongs to $U_{\epsilon}$ for any $0 \leq j \leq [T(x)]+1$
and any $x \in \beta$.
\item Given any $z=s +i u \in [0,1] + i {\mathcal S}_{\mathcal O}$ and
a sequence $\{x_{n}\}$ in $\beta$
with $x_{n} \to 0$ and
\[ s = \lim_{n \to \infty} ( \lceil T(x_{n}) \rceil - T(x_{n})), \ \
\lim_{n \to \infty} \vartheta_{\mathcal O}(x_{n}) = u \]
we obtain
$\lim_{n \to \infty} \varphi^{\lceil T(x_{n}) \rceil} (\upsilon_{\mathcal O}(x_{n})) = \chi_{\mathcal O}(z)$.
\item Given any $\epsilon'>0$
there exists $M \in {\mathbb N}$ such that
\[ \{\varphi^{M}(\upsilon(x)), \hdots, \varphi^{\lceil T(x) \rceil -M}(\upsilon(x))\} \]
is contained in
$U_{\epsilon'}$ for any $x \in \beta$ in a neighborhood of $0$.
\end{itemize}
\end{defi}
Notice that if
$(X, y_{+},\beta,T)$ generates a Long Trajectory then
$({\rm exp}(X), y_{+},\beta,T)$ generates a Long Orbit.
The definitions of Long Trajectories and Orbits are analogous.
Obviously the definition for flows is a bit simpler since
for diffeomorphisms we can only iterate an integer number of times.
\begin{rem}
Long Orbits are topological invariants.
\end{rem}
\begin{rem}
\label{rem:idgnc}
In the previous definitions $\beta$ is a germ of set at $0$ if
${\mathcal S}_{\mathcal O}$ is compact. Otherwise we identify
$\beta$ and $\tilde{\beta}$ if the germs of
$\vartheta_{\mathcal O}^{-1}[-n,n] \cap \beta$ and
$\vartheta_{\mathcal O}^{-1}[-n,n] \cap \tilde{\beta}$ coincide for
any $n \in {\mathbb N}$.
\end{rem}
\begin{defi}
\label{def:trim}
Suppose that $(X, y_{+},\beta,T)$ generates a weak Long Trajectory.
Let $(0,y_{+}') = \mathrm{exp}(MX) (0,y_{+})$ for some $M \in {\mathbb N}$.
Then $(X,y_{+}',\beta,T-2M)$ generates a weak Long Trajectory.
We say that the latter weak Long Trajectory is obtained by {\it trimming} the former one.
Given $\epsilon'>0$ any weak Long Trajectory is contained in $U_{\epsilon'}$ up to trimming
by the last condition in Definition \ref{def:wlt}.

Analogously suppose that $(\varphi, y_{+},\beta,T)$ generates a Long Orbit.
Let $(0,y_{+}') = \varphi^{M} (0,y_{+})$ for some $M \in {\mathbb N}$.
Then $(\varphi,y_{+}',\beta,T-2M)$ generates a Long Orbit.
We say that the latter Long Orbit is obtained by trimming the former one.
\end{defi}
Trimming does not change the fundamental properties of a Long Trajectory.
Moreover it is easy to define germs of Long Trajectory.
Trimming maps a Long Trajectory to another one in the same equivalence class.
\subsection{The residue formula}
\label{subsec:res}
The quantitative properties of the Long Trajectories are obtained by
applying the residue formula. It allows to calculate the ``length" of the
Long Trajectories or more precisely the function $T$ (see Definition \ref{def:lt}).

Consider a vector field $Z=a(y) \partial /\partial y$ defined in a neighborhood of
$\overline{B(0,\epsilon')}$ such that $\mathrm{Sing} (Z) \cap \partial B(0,\epsilon') = \emptyset$.
Let $\gamma:[0,c] \to {\mathbb C}$ a trajectory of $\Re (Z)$ such that
$\gamma(0) \in \partial B(0,\epsilon') \ni \gamma(c)$ and $\gamma (0,c) \subset B(0,\epsilon')$.
Let $\kappa$ be a path in $\partial B(0,\epsilon')$ going from $\gamma(0)$ to $\gamma (c)$
in counter clock wise sense.
Consider the bounded connected component $C_{-}$ of ${\mathbb C} \setminus \gamma \kappa^{-1}$.
We denote $E_{-} = C_{-} \cap \mathrm{Sing} (Z)$.

Consider a Fatou coordinate $\psi_{+}$ of $Z$ defined in a neighborhood of $\gamma(0)$.
We define $\psi_{-}$ and $\psi_{-}'$ Fatou coordinates of $Z$ defined in the neighborhood
of $\gamma(c)$. More precisely, $\psi_{-}$ and $\psi_{-}'$ are obtained by analytic continuation
of $\psi_{+}$ along the paths $\kappa$ and $\gamma$ respectively.
We have
\[ {\mathbb R}^{+} \ni c = \psi_{-}'(\gamma(c)) - \psi_{+}(\gamma(0)) =
\psi_{-}(\gamma(c)) - 2 \pi i \sum_{y \in E_{-}} Res (Z,y) - \psi_{+}(\gamma(0)) . \]
This is the residue formula. Of course it can be extended to other setups.
For instance if $\gamma:[0,c] \to {\mathbb C}$ is a trajectory of $\Re (Z)$
such that there exists $c_{1},c_{2} \in [0,c]$ with
$\gamma(0,c_{1}) \subset {\mathbb C} \setminus \overline{B(0,\epsilon')}$,
$\gamma(c_{1},c_{2}) \subset B(0,\epsilon')$ and
$\gamma(c_{2},c) \subset {\mathbb C} \setminus \overline{B(0,\epsilon')}$
the we obtain
\begin{equation}
\label{equ:res}
\psi_{-}(\gamma(c)) - 2 \pi i \sum_{y \in E_{-}} Res (Z,y) - \psi_{+}(\gamma(0))  = c \in {\mathbb R}^{+}
\end{equation}
where $\psi_{+}$ is a Fatou coordinate of $Z$ defined in a neighborhood of $\gamma(0)$ and
$\psi_{-}$ is the analytic continuation of $\psi_{+}$ along $\gamma[0,c_{1}] \kappa \gamma[c_{2},c]$.

Eq. (\ref{equ:res}) is interesting to study weak Long Trajectories.
Given a Long Trajectory ${\mathcal O}$ (see Definition \ref{def:lt}) we can define a holomorphic
Fatou coordinate $\psi_{+}$ of $X$ in the neighborhood of $(0,y_{+})$.
Then we can consider the Fatou coordinates $\psi_{-}$ and $\psi_{-}'$ defined in
the neighborhood of $\chi_{\mathcal O}(0)$. The Fatou coordinate $\psi_{-}$ is holomorphic
in the neighborhood of $\chi_{\mathcal O}(0)$ whereas $\psi_{-}'$ is equal to $\infty$ in $x=0$.
Thus the length and properties of Long Trajectories are intimately related to the
properties of the meromorphic residue functions.
\begin{rem}
Suppose that ${\mathcal O}=(X, y_{+},\beta,T)$ generates a Long Trajectory.
Then the trajectory $\Gamma(X,\upsilon_{\mathcal O}(x),U_{\epsilon})[0,T(x)]$
establishes a division of $\mathrm{Sing} (X)$ in sets
$E_{-}(x)$ and $E_{+}(x) = (\mathrm{Sing} (X))(x) \setminus E_{-}(x)$
as described above. Moreover the sets $E_{-}(x)$ and $E_{+}(x)$
depend continuously on $x$.
We say that $(E_{-},E_{+})$ is the division of $\mathrm{Sing} (X)$ induced by ${\mathcal O}$.
\end{rem}
\subsection{Behavior of trajectories in adapted coordinates}
The Long Trajectories of an element in $\Xntg$ with $N>1$ are obtained
by analyzing the dynamics in the most exterior compact-like set
${\mathcal C}_{j_{0}}$ such that ${\mathcal U}_{X,j_{0}}^{1} \neq \emptyset$.
This section is devoted to describe the dynamics of $\Re (X)$ in the
basic sets enclosing ${\mathcal C}_{j_{0}}$.

We study the properties of the sets of tangencies between $\Re (X)$ and the
boundaries of the basic sets in the next results. This is useful to
understand the topological behavior of $\Re (X)$. Moreover
the set of tangencies determines the dynamics of $\Re (X)$
for some simple basic sets and in particular for the basic sets that are
the subject of this section.
\begin{defi}
Let $X \in \Xnt$.
Consider an exterior set
\[ {\mathcal E} = \{(x,t) \in B(0,\delta) \times {\mathbb C} :  \eta \geq |t| \geq \rho|x| \} \]
associated to $X$ with $\eta >0$  and $\rho \geq 0$.
We define $T{\mathcal E}_{X}^{\eta}(r, \lambda)$ the set of tangent points between
$|t|=\eta$ and $\Re(\lambda^{e({\mathcal E})} X_{\mathcal E})_{|x=r \lambda}$
for $(r,\lambda) \in {\mathbb R}_{\geq 0} \times {\mathbb S}^{1}$.
We denote $T_{X}^{\epsilon}(r \lambda)= T {\mathcal E}_{X}^{\epsilon}(r, \lambda)$
for the particular case ${\mathcal E}= {\mathcal E}_{0}$.
\end{defi}
\begin{rem}
We have $X = r^{e({\mathcal E})} \lambda^{e({\mathcal E})} X_{\mathcal E}$.
Thus $T{\mathcal E}_{X}^{\eta}(r, \lambda)$ is the set of tangent points between
$\Re (X)_{|x=r \lambda}$ and $|t|=\eta$ for $r \neq 0$. The definition allows to
extend the concept to $r=0$ in adapted coordinates.
\end{rem}
\begin{defi}
\label{def:taintpt0}
Let $X \in \Xnt$.
Consider a compact-like set
\[ {\mathcal C} =\{
(x,w) \in B(0,\delta) \times \overline{B(0,\rho)} \} \setminus
(\cup_{\zeta \in S_{\mathcal C}}
\{ (x,w_{\zeta}) \in  B(0,\delta) \times B(0,\eta_{{\mathcal C}, \zeta}) \}) \]
associated to $X$. We denote $T{\mathcal C}_{X}^{\rho}(r, \lambda)$ the set of
tangent points between the exterior boundary
$|w|=\rho$ of ${\mathcal C}$
and $\Re(\lambda^{e({\mathcal C})}   X_{\mathcal C})_{|x=r \lambda}$.
\end{defi}
\begin{defi}
Let ${\mathcal B}$ a basic set.
We say that a point $P \in T{\mathcal B}_{X}(r, \lambda)$
is {\it convex} if the germ of trajectory of
$\Re(\lambda^{e({\mathcal B})} X_{\mathcal B})_{|x=r \lambda}$
through $P$ is contained in ${\mathcal B}$.
\end{defi}
We describe the tangent sets $T{\mathcal B}_{X}(r, \lambda)$ for parabolic
exterior sets and compact-like sets.
\begin{lem}
(See \cite{JR:mod})
\label{lem:tgpt20}
Let $X \in \Xnt$. Let
${\mathcal E}=\{ \eta \geq |t| \geq \rho|x| \}$ be
a parabolic exterior set associated to $X$
with $0 < \eta <<1$ and $\rho \geq 0$. Then the set
$T{\mathcal E}_{\mu X}^{\eta}(r, \lambda)$ is composed of
$2 \nu({\mathcal E})$ convex points for all
$(\lambda,\mu) \in {\mathbb S}^{1} \times {\mathbb S}^{1}$
and $r$ close to $0$. Each connected component of
$(\partial_{e} {\mathcal E})(r, \lambda) \setminus T{\mathcal E}_{\mu X}^{\eta}(r, \lambda)$
contains a unique point of $T{\mathcal E}_{\mu' X}^{\eta}(r, \lambda)$ for
all $r \in [0,\delta)$, $\lambda$, $\mu$, $\mu' \in {\mathbb S}^{1}$
such that
$\mu' \in {\mathbb S}^{1} \setminus \{ -\mu, \mu \}$.
\end{lem}
\begin{lem}
(See \cite{JR:mod})
\label{lem:tgpt30}
Let $X \in \Xt$ and a compact-like set
\[ {\mathcal C} =\{
(x,w) \in B(0,\delta) \times \overline{B(0,\rho)} \} \setminus
(\cup_{\zeta \in S_{\mathcal C}}
\{ (x,w_{\zeta}) \in  B(0,\delta) \times B(0,\eta_{{\mathcal C}, \zeta}) \}) \]
associated to $X$ with $\rho >>0$. Then $T{\mathcal C}_{\mu X}^{\rho}(r, \lambda)$
is composed of $2 \nu({\mathcal C})$ convex points for all
$(\lambda,\mu) \in {\mathbb S}^{1} \times {\mathbb S}^{1}$
and $r$ close to $0$. Moreover each connected component of
$(\partial_{e} {\mathcal C})(r, \lambda) \setminus T{\mathcal C}_{\mu X}^{\rho}(r, \lambda)$
contains a unique point of $T {\mathcal C}_{\mu' X}^{\rho}(r, \lambda)$ for
all $r \in [0,\delta)$, $\lambda$, $\mu$, $\mu' \in {\mathbb S}^{1}$
such that
$\mu' \in {\mathbb S}^{1} \setminus \{ -\mu, \mu \}$.
\end{lem}
\begin{defi}
\label{def:bdtr}
Let $X \in \Xnt$ with $N>1$.  Let ${\mathcal B}$ be a basic set.
The set
$(\partial_{e}  {\mathcal B} \setminus T_{X} {\mathcal B})(x)$ has $2 \nu ({\mathcal B})>0$
connected components if $\nu ({\mathcal B})>0$ (Lemmas \ref{lem:tgpt20} and \ref{lem:tgpt30}).
Otherwise ${\mathcal B}$ is a non-parabolic exterior set and
$(\partial_{e}  {\mathcal B} \setminus T_{X} {\mathcal B})(x)$ is either empty
or coincides with $(\partial_{e}  {\mathcal B})(x)$.
We say that a set is a {\it boundary transversal} if it is the closure of a connected
component of $(\partial_{e}  {\mathcal B} \setminus T_{X} {\mathcal B})$.
for some basic set ${\mathcal B}$.

Suppose $\nu ({\mathcal B})>0$. We define
the set $\partial_{\downarrow} {\mathcal B}$ of points in $\partial_{e} {\mathcal B}$
where $\Re (X)$ does not point towards the exterior of ${\mathcal B}$. It is the union of
$\nu ({\mathcal B})$ boundary transversals.
\end{defi}
\begin{defi}
Let $X \in \Xnt$ with $N(X)>1$.
Corollary \ref{cor:exir} implies that there exists a sequence of
basic sets
${\mathcal E}_{0}$, ${\mathcal C}_{1}$, ${\mathcal E}_{1}$, ${\mathcal C}_{2}$,
$\hdots$, ${\mathcal E}_{j_{0}-1}$,  ${\mathcal C}_{j_{0}}$
such that
\[ \partial_{I} {\mathcal E}_{0} = \partial_{e} {\mathcal C}_{1}, \
\partial_{I} {\mathcal C}_{1} = \partial_{e} {\mathcal E}_{1},   \hdots, \
\partial_{I} {\mathcal C}_{j_{0}-1} = \partial_{e} {\mathcal E}_{j_{0}-1}, \
\partial_{I} {\mathcal E}_{j_{0}-1} = \partial_{e} {\mathcal C}_{j_{0}}, \]
${\mathcal U}_{X,j}^{1} = \emptyset$ for any $1 \leq j < j_{0}$ and
${\mathcal U}_{X,j_{0}}^{1} \neq \emptyset$. We deduce
$\nu({\mathcal E}_{0}) =    \hdots   = \nu({\mathcal C}_{j_{0}})$.
We say that
${\mathcal E}_{0}$, ${\mathcal C}_{1}$, ${\mathcal E}_{1}$, ${\mathcal C}_{2}$,
$\hdots$, ${\mathcal E}_{j_{0}-1}$,  ${\mathcal C}_{j_{0}}$ is a
{\it simple sequence} associated to $X$.
\end{defi}
Let ${\mathcal B}$ be a basic set in the sequence with ${\mathcal B} \neq {\mathcal C}_{j_{0}}$.
The set ${\mathcal B}(r_{0},\lambda_{0})$ is an annulus that does not contain singular points
of $X$. Moreover
the number of tangent points between $\Re (X)$ and $\partial_{e} {\mathcal B}$
coincides with the
number of tangent points between $\Re (X)$ and $\partial_{I} {\mathcal B}$
and it is equal to $2 \nu (X)$ in
any line $(r,\lambda)=(r_{0},\lambda_{0})$ with $r_{0} \in [0,\delta)$ and
$\lambda_{0} \in {\mathbb S}^{1}$. Both sets of tangent points are composed
of convex points. It is easy to show that in this setting
the dynamics of $\Re (\lambda^{e({\mathcal B})} X_{\mathcal B})_{|x= r \lambda}$
is as described in Figure (\ref{EVfig6}) (see Proposition 6.1 and Corollary 6.1 of \cite{JR:mod}).
The dynamics of $\Re (X)$ in ${\mathcal B}$ is a truncated Fatou flower.
\begin{figure}[h]
\begin{center}
\includegraphics[height=6cm,width=7cm]{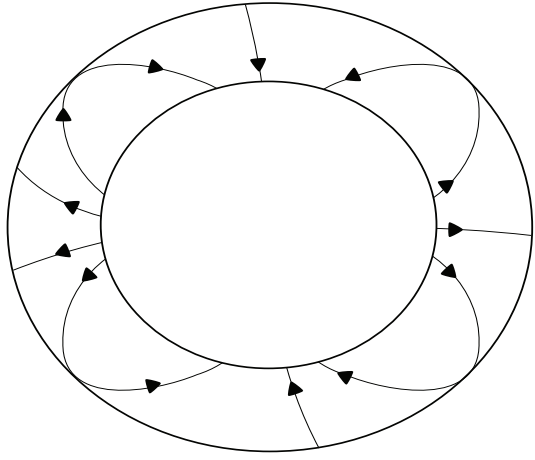}
\end{center}
\caption{Dynamics in basic set ${\mathcal B} \neq {\mathcal C}_{j_{0}}$ of a simple sequence} \label{EVfig6}
\end{figure}
\begin{defi}
\label{def:hit}
Let $X \in \Xnt$ with $N>1$.
Let ${\mathcal E}_{0}$, ${\mathcal C}_{1}$,  $\hdots$,    ${\mathcal C}_{j_{0}}$
be a simple sequence
associated to $X$. Consider a continuous section
$\tau: [0,\delta) \times {\mathbb S}^{1}   \to \partial_{\downarrow} {\mathcal E}_{0}$.
Denote $\Gamma_{r,\lambda} =
\Gamma (X, \tau(r,\lambda), {\mathcal E}_{0} \cup {\mathcal C}_{1} \cup \hdots \cup {\mathcal E}_{j_{0}-1})$.
The dynamics of $\Re (X)$
implies $\sup {\mathcal I}(\Gamma_{r,\lambda}) < \infty$ and
$\Gamma_{r,\lambda}(\sup {\mathcal I}(\Gamma_{r,\lambda})) \in \partial_{\downarrow} {\mathcal C}_{j_{0}}$
for any $(r,\lambda) \in  (0,\delta) \times {\mathbb S}^{1}$.
The formula
\[ \partial \tau (r,\lambda) = \Gamma_{r,\lambda}(\sup {\mathcal I}(\Gamma_{r,\lambda})) \]
defines a continuous section
$\partial \tau: (0,\delta) \times {\mathbb S}^{1} \to \partial_{\downarrow} {\mathcal C}_{j_{0}}$.
\end{defi}
In other words $\partial \tau (r,\lambda)$ is the first point of the positive trajectory
of $\Re (X)$ through $\tau (r,\lambda)$ that belongs to ${\mathcal C}_{j_{0}}$.
\begin{rem}
A section $\tau$ as introduced in Definition \ref{def:hit} satisfies that
$\tau ([0,\delta) \times {\mathbb S}^{1})$ is contained in a component $\kappa$
of $\partial_{\downarrow} {\mathcal E}_{0}$.
The set $(\partial \tau) (B(0,\delta) \setminus \{0\})$ is contained in
a connected component $\tilde{\kappa}$ of
$\partial_{\downarrow} {\mathcal C}_{j_{0}}$.
Moreover $\tilde{\kappa}$ depends only on $\kappa$ and the mapping
$\kappa \to \tilde{\kappa}$ is a bijection from the boundary transversals of
$\partial_{\downarrow} {\mathcal E}_{0}$ onto the boundary transversals of
$\partial_{\downarrow} {\mathcal C}_{j_{0}}$.
\end{rem}
The next proposition shows that $\partial \tau$ is continuous in adapted coordinates.
\begin{pro}
\label{pro:hit}
Let $X \in \Xnt$ with $N>1$.
Let ${\mathcal E}_{0}$,  $\hdots$,    ${\mathcal C}_{j_{0}}$
be a simple sequence.
Consider a continuous section
$\tau: [0,\delta) \times {\mathbb S}^{1}   \to \partial_{\downarrow} {\mathcal E}_{0}$.
Then $\partial \tau$ admits a continuous extension to
$[0,\delta) \times {\mathbb S}^{1}$ in the adapted
coordinates of ${\mathcal C}_{j_{0}}$.
Moreover we have $(\partial \tau)(0,\lambda) \in Tr_{\leftarrow \infty}(X_{j_{0}}(\lambda))$
for any $\lambda \in {\mathbb S}^{1}$.
The mapping $(\partial \tau)_{|r=0}$ depends only on the connected component
$\kappa$ of $\partial_{\downarrow} {\mathcal E}_{0}$ containing $\tau([0,\delta) \times {\mathbb S}^{1})$.
The mapping $\kappa \to (\partial \tau)_{|r=0}$ is a bijection from
$\partial_{\downarrow} {\mathcal E}_{0}$ onto the continuous sections of
$Tr_{\leftarrow \infty}(X_{j_{0}}(\lambda))$.
\end{pro}
\begin{proof}
We denote
\[ {\mathcal E}_{j_{0}-1}^{\rho'} =
\{(x,t) \in B(0,\delta) \times {\mathbb C} :  \eta \geq |t| \geq \rho' |x| \} \]
in adapted coordinates associated to ${\mathcal E}_{j_{0}-1}$.
We have ${\mathcal E}_{j_{0}-1}={\mathcal E}_{j_{0}-1}^{\rho}$. Analogously we denote
\[ {\mathcal C}_{j_{0}}^{\rho'} =\{
(x,w) \in B(0,\delta) \times \overline{B(0,\rho')} \} \setminus
(\cup_{\zeta \in S_{\mathcal C}}
\{ (x,w_{\zeta}) \in  B(0,\delta) \times B(0,\eta_{{\mathcal C}, \zeta}) \}) \]
in adapted coordinates associated to ${\mathcal C}_{j_{0}}$ ($t=xw$).
We obtain ${\mathcal C}_{j_{0}}={\mathcal C}_{j_{0}}^{\rho}$.
Consider the section $(\partial \tau)^{\rho'}$ associated to
$\tau$ and
${\mathcal E}_{0}$, ${\mathcal C}_{1}$,  $\hdots$,    ${\mathcal C}_{j_{0}-1}$,
${\mathcal E}_{j_{0}-1}^{\rho'}$,  ${\mathcal C}_{j_{0}}^{\rho'}$ for $\rho' \geq \rho$.
The image of $(\partial \tau)^{\rho'}$ is contained in a connected component
$\kappa_{\rho'}$ of $\partial_{\downarrow} {\mathcal C}_{j_{0}}^{\rho'}$.
Moreover $\kappa_{\rho'}$ depends continuously on $\rho'$ for $\rho' \geq \rho$.

Consider $\rho' > \rho$ and $\lambda  \in {\mathbb S}^{1}$.
We define
$\iota_{\lambda}^{\rho'}: \kappa_{\rho'}(0,\lambda) \to  \kappa_{\rho}(0,\lambda)$ as
the mapping given by the formula
\[ \iota_{\lambda}^{\rho'}(P) = \Gamma_{P}^{\rho'}
(\sup {\mathcal I}(\Gamma_{P}^{\rho'}))  \
\mathrm{where} \ \Gamma_{P}^{\rho'} = \Gamma ( X_{j_{0}}(\lambda), P,
\overline{B(0,\rho')}  \setminus  B(0,\rho) ).
\]
Notice that all the accumulation points of sequences of the form
$(\partial \tau)(r_{n},\lambda_{n})$ with $(r_{n},\lambda_{n}) \to (0,\lambda)$
are contained in
$\iota_{\lambda}^{\rho'}( \kappa_{\rho'}(0,\lambda))$ for any $\rho' \geq \rho$.
Hence any accumulation point belongs to
$E=\{ P \in \kappa_{\rho}(0,\lambda) : \Gamma_{P}^{\rho'}
(\inf {\mathcal I}(\Gamma_{P}^{\rho'})) \in \partial B(0,\rho') \ \forall \rho' > \rho \}$.
A neighborhood of $\infty$ in ${\mathbb C}$ is a union of angles of $\Re (X_{j_{0}}(\lambda))$
(see Remark \ref{rem:frinfty} and Definition \ref{def:angle}) limited by trajectories in
$Tr_{\infty}(X_{j_{0}}(\lambda))$.
As a consequence the set $E$ is the singleton
$\kappa_{\rho}(0,\lambda) \cap Tr_{\leftarrow \infty}(X_{j_{0}}(\lambda))$
for any $\lambda \in {\mathbb S}^{1}$.
\end{proof}
\begin{rem}
Let $X \in \Xntg$ with $N>1$.
Let ${\mathcal E}_{0}$,  $\hdots$,    ${\mathcal C}_{j_{0}}$
be a simple sequence
associated to $X$. By applying Proposition \ref{pro:hit} to $X$ and $-X$ we obtain
the existence of a bijection
between the connected components of
$(B(0,\delta) \times \partial B(0,\epsilon)) \setminus T_{X}^{\epsilon}$
and the continuous sections of $Tr_{\infty}(X_{j_{0}}(\lambda))$.
\end{rem}
\begin{defi}
\label{def:parj}
Let ${\mathcal P}_{1}$,  $\hdots$, ${\mathcal P}_{2 \nu(X)}$
be the petals of $\Re (X)_{|U_{\epsilon}(0)}$ (see Definition \ref{def:atpetvf}).
Let $a_{1}$, $\hdots$, $a_{2 \nu(X)}$ be the $2 \nu (X)$ connected components of
$(B(0,\delta) \times \partial B(0,\epsilon)) \setminus T_{X}^{\epsilon}$.
We can enumerate them so we have $a_{j}(0) \subset \overline{{\mathcal P}_{j}}$ for
any $1 \leq j \leq 2 \nu (X)$.
Moreover if ${\mathcal P}_{j}$ is an attracting petal we have
$a_{j} \subset \partial_{\downarrow} {\mathcal E}_{0}$.
We denote by $\partial_{j}:{\mathbb S}^{1} \to \partial_{e} {\mathcal C}_{j_{0}}$ the mapping
$(\partial \tau)_{|r=0}$ for $\tau (B(0,\delta)) \subset a_{j}$.
\end{defi}
\begin{rem}
 There exists a bijection between attracting (resp. repelling) petals
of $\Re (X)_{|x=0}$ and continuous sections of $Tr_{\leftarrow \infty}(X_{j_{0}}(\lambda))$
(resp. $Tr_{\to \infty}(X_{j_{0}}(\lambda))$).
\end{rem}
Let us explain the remark.
A section $\tau_{1}: B(0,\delta) \to B(0,\delta) \times B(0,\epsilon)$
with $\tau_{1}(0)$ in an attracting petal ${\mathcal P}_{j}$ induces
a  continuous  section
$\partial \tau_{1}: [0,\delta) \times {\mathbb S}^{1} \to \partial_{e} {\mathcal C}_{j_{0}}$
in adapted coordinates such that
$\partial_{j}(\lambda) = (\partial \tau_{1})(0,\lambda)$ for any $\lambda \in {\mathbb S}^{1}$.
More precisely, consider $\epsilon'<0$ such that $\epsilon' < |\tau_{1}(0)|$.
Then $\tau_{1}$ induces a section $\tau_{1}':B(0,\delta) \to  B(0,\delta) \times \partial B(0,\epsilon')$
where $\tau_{1}'(x)$ is the first point in $B(0,\delta) \times \partial B(0,\epsilon')$
of the positive trajectory of $\Re (X)$ through $\tau_{1}(x)$.
In the same way we can define $\tau'$ for $\tau (B(0,\delta)) \subset a_{j}$.
We define $\partial \tau_{1} = \partial \tau_{1}'$ and we have
$\partial \tau = \partial \tau'$.
Since
$\tau'(x)$ and $\tau_{1}'(x)$ are contained in the same component of
$(B(0,\delta) \times \partial B(0,\epsilon')) \setminus T_{X}^{\epsilon'}$
for any $x \in B(0,\delta)$ we obtain $\partial_{j} \equiv (\partial \tau_{1})_{|r=0}$
by Proposition \ref{pro:hit}.
\subsection{Existence of Long Trajectories}
Next we prove the existence of Long Trajectories for $N>1$. The main idea is that
Long Trajectories appear naturally in the neighborhood of some homoclinic
trajectories of polynomial vector fields associated to the unfolding.
\begin{pro}
\label{pro:exll}
Let $X \in \Xntg$ with $N>1$.
There exist $y_{+} \in B(0,\epsilon) \setminus \{0\}$,
a germ of set $\beta$ at $0$ and a continuous function
$T: \beta \to {\mathbb R}^{+}$ such that
$(X,y_{+},\beta,T)$ generates a Long Trajectory with
${\mathcal S}_{\mathcal O}={\mathbb R}$.
\end{pro}
\begin{proof}
The first part of the proof is intended to introduce the objects that
define the Long Trajectory of $\Re (X)$.
Let ${\mathcal E}_{0}$, $\hdots$,    ${\mathcal C}_{j_{0}}$
be a simple sequence
associated to $X$.
Consider $\lambda_{0} \in {\mathcal U}_{X,j_{0}}^{1}$ and a homoclinic trajectory
$\Gamma = \Gamma(X_{j_{0}}(\lambda_{0}), w_{0}, {\mathbb C})$ (Proposition \ref{pro:DES}).
We have ${\mathcal I}(\Gamma) = (s_{0},s_{1})$.
The choice of the compact-like set
\[ {\mathcal C}_{j_{0}} =\{
(x,w) \in B(0,\delta) \times \overline{B(0,\rho)} \} \setminus
(\cup_{\zeta \in S_{j_{0}}}
\{ (x,w_{\zeta}) \in  B(0,\delta) \times B(0,\eta_{\zeta}) \}) \]
implies $\sharp (\Gamma \cap \partial B(0,\rho)) =2$ because of the local
dynamics of $\Re (X_{j_{0}}(\lambda_{0}))$ in the
neighborhood of $\infty$.
Indeed we have $\Gamma \cap \partial B(0,\rho) = \{ \Gamma(t_{0}), \Gamma(t_{1}) \}$
for some $s_{0} < t_{0} < t_{1} < s_{1}$. There exists a unique attracting petal
${\mathcal P}_{+}={\mathcal P}_{j}$ of $\Re (X)_{|U_{\epsilon}(0)}$ such that
$\partial_{j}(\lambda_{0}) = \Gamma (t_{0})$ (see Definition \ref{def:parj}).
Denote $\partial_{+}=\partial_{j}$.
Analogously, there exists a  unique repelling petal ${\mathcal P}_{-}$ of
$\Re (X)_{|x=0}$ such that $\partial_{-}(\lambda_{0}) = \Gamma (t_{1})$.
Let $\psi$ be a Fatou coordinate of $X_{j_{0}}(\lambda_{0})$ defined in the
neighborhood of $\infty$ and such that $\psi(\infty)=0$.
There exists a unique connected $C_{-}$ component of
${\mathbb C} \setminus \Gamma(s_{0},s_{1})$ such that
$\Gamma$ parametrizes $\partial C_{-}$ in clock wise sense.
Denote $E_{-} = C_{-} \cap \mathrm{Sing} (X_{j_{0}}(1))$. We obtain
\begin{equation}
\label{equ:resp}
{\mathbb R}^{+} \ni s_{1}-s_{0} =
\psi(\infty) - \psi(\infty) - 2 \pi i \sum_{P \in E_{-}} Res(X_{j_{0}}(\lambda_{0}),P)
\end{equation}
by the residue formula.
Consider the set $E_{-}(r,\lambda) \subset (\mathrm{Sing} (X))(r,\lambda)$ that varies continuously
with respect to $(r,\lambda)$ and satisfies $E_{-}(0,1)=E_{-}$.
We have
\begin{equation}
\label{equ:resa}
\sum_{P \in E_{-}(x)} Res(X,P) = \frac{1}{|x|^{e({\mathcal C}_{j_{0}})}}
\frac{\lambda_{0}^{e({\mathcal C}_{j_{0}})}}{\mu^{e({\mathcal C}_{j_{0}})}}
\left( \sum_{P \in E_{-}} Res(X_{j_{0}}(\lambda_{0}),P) + o(1) \right)
\end{equation}
where $\mu = x/|x|$. The function $\sum_{P \in E_{-}(x)} Res(X,P)$
is meromorphic (Proposition 5.2 of \cite{UPD}) and it has a pole of
order greater than $0$.

Fix $(0,y_{+}) \in {\mathcal P}_{+}$.
Consider $(0,y_{-}) \in {\mathcal P}_{-}$
such that $\mathrm{exp}(z X)(0,y_{-})$ is well-defined and belongs to
$U_{\epsilon}$ for any $z \in i {\mathbb R}$.
Given any $(0,y_{-}') \in {\mathcal P}_{-}$ the point
$\mathrm{exp}(-j X)(0,y_{-}')$ satisfies the previous property for
some $j \in {\mathbb N}$ big enough.
Let $\psi_{+}$ be a Fatou coordinate of $X$ defined in the neighborhood of
${\mathcal P}_{+}$ in ${\mathbb C}^{2}$. We define a Fatou coordinate
$\psi_{-}$ of $X$ in a neighborhood of ${\mathcal P}_{-}$ as in Subsection
\ref{subsec:res}. We define
\begin{equation}
\label{equ:deftf}
T_{0}(x) =
\psi_{-}(0,y_{-}) - \psi_{+}(0,y_{+})  - 2 \pi i \sum_{P \in E_{-}(x)} Res(X,P), \ \
T_{s}(x) =T_{0}(x) + is
\end{equation}
for $s \in {\mathbb R}$.
Eqs. (\ref{equ:resp}) and (\ref{equ:resa}) imply that there exists a curve
$\beta_{s}$ adhering $\lambda_{0}$ at $0$ (see Definition \ref{def:dir})
and contained in $T_{s}^{-1}({\mathbb R}^{+})$
for any $s \in {\mathbb R}$. We define
$\beta$ as a connected set such that $\beta \subset \cup_{s \in {\mathbb R}} \beta_{s}$,
$\beta_{\pi} \cap (\{0\} \times {\mathbb S}^{1}) =\{(0,\lambda_{0})\}$
(see Definition \ref{def:dir}) and contains the germ of $\beta_{s}$ for any
$s \in {\mathbb R}$. Let us clarify that we do not define
$\beta = \cup_{s \in {\mathbb R}} \beta_{s}$ straight up
because then $\beta_{\pi} \cap (\{0\} \times {\mathbb S}^{1}) =\{(0,\lambda_{0})\}$
does not hold true. We define $T = Re (T_{0})$,
$\upsilon_{\mathcal O}(x)=(x,y_{+})$, $(\vartheta_{\mathcal O})_{|\beta_{s}} \equiv s$
and $\chi_{\mathcal O}(z) = \mathrm{exp}(z X)(0,y_{-})$.
Our goal is proving that ${\mathcal O}=(X,y_{+},\beta,T)$ generates
a Long Trajectory.

There exists a continuous section $\upsilon^{1}: \beta \to U_{\epsilon}$
such that
\[ \psi_{-}(\upsilon^{1}(x)) - \psi_{+}(x,y_{+}) \equiv
\psi_{-}(0,y_{-}) - \psi_{+}(0,y_{+}) + i \vartheta_{\mathcal O}(x) \]
and
$\lim_{x \in \beta, \ \vartheta_{\mathcal O}(x) \to s, \ \ x \to 0}
\upsilon^{1}(x) = \chi_{\mathcal O}(i s)$ for any $s \in {\mathbb R}$.
Notice that
\[ \psi_{-}(\upsilon^{1}(x)) - \psi_{+}(x,y_{+}) - 2 \pi i \sum_{Q \in E_{-}(x)} Res(X,Q) = T(x) \in {\mathbb R}^{+} \]
for any $x \in \beta$.
 We define
\[ \Gamma_{0}=\Gamma(X,(x,y_{+}),U_{\epsilon}), \ \
\Gamma_{1}=\Gamma(X,\upsilon^{1}(x),U_{\epsilon}) . \]
We claim that $\upsilon^{1}(x) = \Gamma_{0}(T(x))$ for $x \in \beta$.
Let $u_{j}(x)$ be the smallest positive real number such that
$\Gamma_{j}((-1)^{j} u_{j}(x)) \in \partial_{e} {\mathcal C}_{j_{0}}$ for
$j \in \{0,1\}$.
%
Denote
$\kappa_{+}(x) = \Gamma_{0}(u_{0}(x))$,
$\kappa_{-}(x) = \Gamma_{1}(-u_{1}(x))$.
Moreover we get
$\lim_{x \in \beta, x \to 0} \kappa_{+}(x) = \partial_{+}(\lambda_{0})$ and
$\lim_{x \in \beta, x \to 0} \kappa_{-}(x) = \partial_{-}(\lambda_{0})$.
Since $\partial_{+}(\lambda_{0})$ and $\partial_{-}(\lambda_{0})$ belong to
the same trajectory  of $\Re (X_{j_{0}}(\lambda_{0}))$ there exists
$u_{2}(x) \in {\mathbb R}^{+}$ such that
$\Gamma(X,\kappa_{+}(x), U_{\epsilon})(0,u_{2}(x)) \in \accentset{\circ}{{\mathcal C}_{j_{0}}}$
and
\[ \tilde{\kappa}_{+}(x) \stackrel{def}{=} \Gamma(X,\kappa_{+}(x),U_{\epsilon})(u_{2}(x)) \in
\partial_{e} {\mathcal C}_{j_{0}}. \]
 We have $\lim_{x \in \beta, x \to 0} \tilde{\kappa}_{+}(x) = \partial_{-}(\lambda_{0})$.
Since $\psi_{-} - 2 \pi i \sum_{P \in E_{-}(x)} Res(X,P)$ is an analytic continuation of
$\psi_{+}$ along $\Gamma_{0}$ then there exists a Fatou coordinate
$\tilde{\psi}_{+}$ defined in a neighborhood of $(r,\lambda,w)=(0,\lambda_{0},\partial_{-}(\lambda_{0}))$
such that $\tilde{\psi}_{+}(\kappa_{-}(x)) - \tilde{\psi}_{+}(\tilde{\kappa}_{+}(x)) \in {\mathbb R}$
for any $x \in \beta$.
The point $\partial_{-}(\lambda_{0})$ does not belong to $T ({\mathcal C}_{j_{0}})_{X}^{\rho}(0,\lambda_{0})$.
Hence $\tilde{\kappa}_{+}(x)$ and $\kappa_{-}(x)$ belong to a common connected transversal to
$\Re (X)$ for any $x \in \beta$. We deduce that $\tilde{\kappa}_{+}(x)=\kappa_{-}(x)$
for any $x \in \beta$.

We have that $\upsilon^{1}(x) = \Gamma_{0}(T(x))$ for any $x \in \beta$.
Given $\epsilon'>0$ small there exists a continuous function
$v_{0}: \beta  \to {\mathbb R}^{+} \cup \{0\}$
such that $\Gamma_{0}[0,v_{0}(x)) \cap \overline{U_{\epsilon'}} = \emptyset$ and
$\Gamma_{0}(v_{0}(x)) \in \overline{U_{\epsilon'}}$.
The function  $v_{0}$ is bounded by above.
Moreover there exists $v_{1} \in {\mathbb R}^{+}$ such that
${\mathcal I}(\Gamma(X,{\rm exp}(-v_{1}X)(\chi_{\mathcal O}(i s)),U_{\epsilon'}))$
contains $(-\infty,0]$ for
any $s \in {\mathbb R}$.
Proposition \ref{pro:hit} applied to
$\Gamma_{0}(v_{0})$, $\Gamma_{1}(-v_{1})$ and the construction of $\Gamma_{0}$
imply that
$\Gamma_{0}(v_{0}(x),T(x)-v_{1})$ is contained in $U_{\epsilon'}$.
Thus ${\mathcal O}=(X,y_{+},\beta,T)$ generates
a Long Trajectory.
\end{proof}
\begin{rem}
\label{rem:fhevol}
Consider the setting in Proposition \ref{pro:exll}.
All the Long Trajectories of points of ${\mathcal P}_{+}$ with respect to
$\beta$ have an analogous behavior.
Let $y_{+}' \in {\mathcal P}_{+}$ and $y_{-}' \in {\mathcal P}_{-}$
such that
\[ \psi_{-}(0,y_{-}') - \psi_{+}(0,y_{+}') =
\psi_{-}(0,y_{-}) - \psi_{+}(0,y_{+}) + j \]
for some $j \in {\mathbb Z}$. If
$\mathrm{exp}(i s X)(0,y_{-}') \in U_{\epsilon}$ for any $s \in {\mathbb R}$ then
$\tilde{\mathcal O}=(X,y_{+}',\beta,T+j)$ generates a Long Orbit and
$\chi_{\tilde{\mathcal O}}(0)=(0,y_{-}')$ by the proof of Proposition \ref{pro:exll},
see Eq. (\ref{equ:deftf}). In general there exists
$j_{0} \in {\mathbb N}$   such that
$\breve{\mathcal O}=(X,y_{+}',\beta,T+j-j_{0})$ generates a Long Orbit and
$\chi_{\breve{\mathcal O}}(0)=\mathrm{exp}(-j_{0} X)(0,y_{-}')$.
Hence, up to replace $T$ with $T-j_{1}$ for some $j_{1} \in {\mathbb N} \cup \{0\}$,
the Long Trajectory of a point of ${\mathcal P}_{+}$ with respect to
$\beta$ is always non-empty.
\end{rem}
It is natural to ask if we can choose $y_{+}$ in any attracting
petal ${\mathcal P}_{+}$ of $\Re(X)_{|U_{\epsilon}(0)}$.
The answer is positive and the proof exploits the symmetries of
the polynomial vector fields associated to $X$.
\begin{pro}
\label{pro:exll2}
Let $X \in \Xntg$ with $N>1$. Let ${\mathcal P}_{+}$ be an attracting petal
of $Re(X)_{|x=0}$. Then there exists a germ of set $\beta$ at $0$
such that the Long Trajectory associated to $X$, $y_{+}$
with respect to $\beta$ is not empty for any $y_{+} \in {\mathcal P}_{+}$.
Moreover given a repelling petal ${\mathcal P}_{-}$ and a point
$(0,y_{-}) \in  {\mathcal P}_{-}$ there exist $d \in {\mathbb N} \cup \{0\}$
and a Long Trajectory ${\mathcal O} = (X,y_{+}',\beta,T)$
such that $\chi_{\mathcal O}(0) = {\rm exp}(-d X)(0,y_{-})$.
\end{pro}
\begin{proof}
Up to a ramification $(x,y) \mapsto (x^{l},y)$
we can suppose $X \in \Xtg$. We have
\[ X = u(x,y) (y - \gamma_{1}(x))^{n_{1}} \hdots (y - \gamma_{p}(x))^{n_{p}} \frac{\partial}{\partial y} \]
where $u \in {\mathbb C}\{x,y\}$ is a unit and $\nu(X) + 1 = n_{1} + \hdots + n_{p}$.
Denote $\nu = \nu (X)$.

Consider the notations in Proposition \ref{pro:exll}.
We have ${\mathcal P}_{j}={\mathcal P}_{+}$ and
${\mathcal P}_{k}={\mathcal P}_{-}$ for some
$j,k \in {\mathbb Z}/(2 \nu {\mathbb Z})$.
We have $X_{j_{0}}(\lambda) = \lambda^{j_{0} \nu} P_{j_{0}}(w) \partial / \partial w$.
Notice that $u(0,0)$ is the
the coefficient of highest degree in $P_{j_{0}}$.
The trajectories in $Tr_{\infty} (X_{j_{0}}(\lambda))$ adhere to the directions in
$\lambda^{j_{0} \nu} u(0,0) w^{\nu} \in {\mathbb R}$ at $\infty$.
These directions rotate at a speed of $-j_{0}$ (in other words
if $\lambda$ rotates an angle of $\theta$ then the directions rotate an angle of
$-j_{0} \theta$).
In particular the tangent directions
to  $Tr_{\infty} (X_{j_{0}}(\lambda e^{2 \pi i s/(j_{0} \nu)}))$ are obtained
by rotating an angle of $-2 \pi s/\nu$ the tangent directions to
$Tr_{\infty} (X_{j_{0}}(\lambda))$ for $s \in [0,1]$.
Since $X_{j_{0}}(\lambda_{0}) = X_{j_{0}}(\lambda_{0} e^{2 \pi i /(j_{0} \nu)})$
every trajectory in $Tr_{\to \infty} (X_{j_{0}}(\lambda))$
(resp. $Tr_{\leftarrow \infty} (X_{j_{0}}(\lambda))$)
is transformed into the previous one when $s$ goes from $0$ to $1$.
The previous discussion implies (see Definition \ref{def:parj})
\[ \partial_{j+2} (\lambda_{0} e^{2 \pi i /(j_{0} \nu)}) = \Gamma(t_{0}), \ \mathrm{and} \
\partial_{k+2} (\lambda_{0} e^{2 \pi i /(j_{0} \nu)}) = \Gamma(t_{1})  \]
where $\Gamma$ is the homoclinic trajectory such that
$\partial_{j}(0,\lambda_{0}) \in \Gamma \ni \partial_{k}(0,\lambda_{0})$.
Consider any points $(0,y_{+}') \in {\mathcal P}_{j+2}$ and $(0,y_{-}') \in {\mathcal P}_{k+2}$.
Up to replace $(0,y_{-}')$ with $\mathrm{exp}(-d X)(0,y_{-}')$ if necessary there
exists
a set $\beta'$ tangent to $\lambda_{0} e^{2 \pi i /(j_{0} \nu)}$ such that
$\tilde{\mathcal O}=(X,y_{+}',\beta',T')$ generates a Long Trajectory for some
$T':\beta' \to {\mathbb R}^{+}$ with
$\chi_{\tilde{\mathcal O}}(0)=(0,y_{-}')$
by Proposition \ref{pro:exll}.
Long Trajectories of points of ${\mathcal P}_{j+2}$ with respect to $\beta'$
are not empty by Remark \ref{rem:fhevol}.

We proved that if the result is true for ${\mathcal P}_{l}$
then it is also true for ${\mathcal P}_{l+2}$.
Hence it is satisfied for any   petal of $\Re (X)_{|x=0}$.
\end{proof}
\section{Tracking Long Orbits}
\label{sec:tracking}
Let $\varphi \in \dif{p1}{2}$ and
a convergent normal form $X$ of $\varphi$.
Let ${\mathcal O}=(X,y_{+},\beta,T)$ be a Long Trajectory. It is natural to ask whether
$(\varphi,y_{+},\beta,T)$ generates a Long Orbit.
A priori this is not clear since orbits of $\varphi$ could be
(and are!) very different than orbits of ${\mathfrak F}_{\varphi}$.
Anyway
the orbits ${\{ \varphi^{j}(\upsilon_{\mathcal O}(x))\}}_{0 \leq j \leq \lceil T(x) \rceil}$
and ${\{ {\mathfrak F}_{\varphi}^{j}(\upsilon_{\mathcal O}(x))\}}_{0 \leq j \leq \lceil T(x) \rceil}$
remain close for $x \in \beta$. The dynamics of $\varphi$ ``tracks"  the dynamics of
${\mathfrak F}_{\varphi}$ along Long Trajectories of $\Re (X)$ (Proposition \ref{pro:tracking}).
The idea is that Long Trajectories change of basic set a number of times
that is bounded by above uniformly and that in basic sets the tracking property
is simple to prove.

We study topological conjugacies $\sigma$ between elements $\varphi$, $\eta$
of $\dif{p1}{2}$.
Long Orbits are topological invariants but $\sigma$
does not conjugate ${\mathfrak F}_{\varphi}$ and ${\mathfrak F}_{\eta}$
and does not preserve the dynamical splitting in general. Hence
it is not clear that the image of a Long Orbit of $\varphi$ by $\sigma$
is close to a Long Trajectory of $Y$ where $Y$ is a convergent normal form of $\eta$.
We prove in Section \ref{sec:Rolle} that Long Orbits are always in the neighborhood of
Long Trajectories of the normal form since the latter one satisfies a sort of Rolle property.
The tracking phenomenon allows to generalize the
residue formula for diffeomorphisms (Propositions \ref{pro:ltlo} and \ref{pro:lores}).
\begin{defi}
Let $X \in \Xntg$ with $N \geq 1$.
Let $\beta$ be a germ of connected set at $0 \in {\mathbb C}$.
Consider a family of sub-trajectories
$\Gamma_{x}:[0,T'(x)] \to U_{\epsilon}(x)$ of $\Re (X)$
defined for $x \in \beta$.
We say that ${\{ \Gamma_{x} \}}_{x \in \beta}$ is
{\it stable} if
$\{ x \in \beta: \Gamma_{x} \cap {\mathcal E} \neq \emptyset \}$
does not adhere the directions in
${\mathcal U}_{X}^{{\mathcal E},1}$ (see Definition \ref{def:unstext})
for any   exterior set ${\mathcal E}$.
\end{defi}
The orbits of $\varphi$ and ${\mathfrak F}_{\varphi}$ are very
different in the neighborhood of the indifferent fixed points of $\varphi$.
Roughly speaking a stable family is a family far away from indifferent
fixed points.
\begin{defi}
Let $\varphi \in \diff{p1}{2}$ with
$N \geq 1$. Fix a convergent normal form $X$ of $\varphi$
and a basic set ${\mathcal B}$.
%
%
We define $\nu_{{\mathcal B}}(\Delta_{\varphi}) \in {\mathbb N} \cup \{0\}$ as the integer such that
$\Delta_{\varphi}$ (see Definition \ref{def:delta}) is of the form
$x^{\nu_{{\mathcal B}}(\Delta_{\varphi})} g(x,t)$ in the adapted coordinates
$(x,t)$ associated to ${\mathcal B}$ where $g(0,t) \not \equiv 0$.
\end{defi}
The next propositions provide the tracking properties for basic sets.
\begin{pro}
\label{pro:boufespre}
Let $\varphi \in \diff{p1}{2}$ with
$N \geq 1$. Fix a convergent normal form $X$ of $\varphi$.
Let $\beta$ be a germ of connected set at $0 \in {\mathbb C}$.
Consider an exterior set
${\mathcal E} = \{(x,t) \in B(0,\delta) \times {\mathbb C} :  \eta \geq |t| \geq \rho|x| \}$.
Let $\nu =\nu_{{\mathcal E}}(\Delta_{\varphi}) - e({\mathcal E})$
(see Definitions \ref{def:rvfe}, \ref{def:rvfc}).
Fix a closed set $S \subset {\mathbb S}^{1} \setminus {\mathcal U}_{X}^{{\mathcal E},1}$
(see Definition \ref{def:unstext}).
There exists  $\xi > 0$
such that the properties
$\mathrm{exp}([0,j]X)(x,t_{1}) \subset \cup_{x \in (0,\delta) S} {\mathcal E}(x)$
for some $j \in {\mathbb N} \cup \{ 0 \}$ and
$(x,t_{2}) \in B_{X}((x,t_{1}),1)$
imply
\begin{equation}
\label{equ:trackes}
|\psi_{X} \circ {\varphi}^{j+1}(x,t_{2}) - \psi_{X}  \circ {{\mathfrak F}}_{\varphi}^{j+1}(x,t_{1})| \leq
|\psi_{X}(x,t_{2}) - \psi_{X}(x,t_{1})| + \xi {|x|}^{\nu}.
\end{equation}
Moreover we can choose $\xi>0$ as small as desired by considering a small $\eta>0$.
\end{pro}
\begin{proof}
We denote
\[ G = (\psi_{X} \circ {\varphi}^{j+1}(x,t_{2}) - \psi_{X}  \circ {{\mathfrak F}}_{\varphi}^{j+1}(x,t_{1})) -
(\psi_{X}(x,t_{2}) - \psi_{X}(x,t_{1})) . \]
We have $G=\sum_{k=0}^{j} \Delta_{\varphi}(\varphi^{j}(x,t_{2}))$.
We obtain
\[ |G| \leq  \sum_{k=0}^{\infty}
\frac{C |x|^{\nu_{\mathcal E}(\Delta_{\varphi})}}{(\min_{0 \leq l \leq j}
|\psi_{\mathcal E}({\mathfrak F}_{\varphi}^{l}(x,t_{1}))|
+k |x|^{e({\mathcal E})})^{2}} =
O \left( \frac{|x|^{\nu_{\mathcal E}(\Delta_{\varphi})-e({\mathcal E})}}{\min_{0 \leq l \leq j}
|\psi_{\mathcal E}({\mathfrak F}_{\varphi}^{l}(x,t_{1}))|} \right) \]
for some $C \in {\mathbb R}^{+}$ by Lemmas \ref{lem:cansumnp} and \ref{lem:cansum}.
Since
$\psi_{\mathcal E} (x,t) \to \infty$ uniformly in ${\mathcal E}$
when $\eta \to 0$ we get Eq. (\ref{equ:trackes}).
\end{proof}
\begin{pro}
\label{pro:bouis}
Let $\varphi \in \diff{p1}{2}$ with
$N > 1$. Fix a convergent normal form $X$ of $\varphi$.
Fix   a compact-like basic set ${\mathcal C}$.
Let $\nu =\nu_{{\mathcal C}}(\Delta_{\varphi}) - e({\mathcal C})$.
Fix $B \in {\mathbb R}^{+}$.
There exists a constant $C'>0$ such that
$\mathrm{exp}([0,j]X)(P) \subset U_{\epsilon}(x) \cap {\mathcal C}$
for some $x \in B(0,\delta)$, $0 \leq j \leq B/|x|^{e({\mathcal C})}$ and $Q \in B_{X}(P,1)$ imply
\begin{equation}
\label{equ:trackcs}
 |\psi_{X} \circ {\varphi}^{j+1}(Q) - \psi_{X} \circ {{\mathfrak F}}_{\varphi}^{j+1}(P)| \leq
|\psi_{X}(Q) - \psi_{X}(P)| + C' {|x|}^{\nu}.
\end{equation}
\end{pro}
\begin{proof}
We denote
\[ G = (\psi_{X} \circ {\varphi}^{j+1}(Q) - \psi_{X}  \circ {{\mathfrak F}}_{\varphi}^{j+1}(P)) -
(\psi_{X}(Q) - \psi_{X}(P)) . \]
We have $G=\sum_{k=0}^{j} \Delta_{\varphi}(\varphi^{j}(Q))$.
The inequalities
\[ |G| \leq |x|^{\nu_{\mathcal C}(\Delta_{\varphi})}
C \sum_{k=0}^{j} 1 \leq B C |x|^{\nu_{\mathcal C}(\Delta_{\varphi})-e({\mathcal C})}   \]
lead us to Eq. (\ref{equ:trackcs}).
\end{proof}
\begin{defi}
\label{def:ab}
Let $X \in \Xnt$ with $N \geq 1$.
Let $\beta$ be a germ of connected set at $0 \in {\mathbb C}$.
Consider a family of sub-trajectories
$\Gamma_{x}:[0,T'(x)] \to U_{\epsilon}(x)$ of $\Re (X)$
defined for $x \in \beta$. We say that the family is $(A,B)$ bounded if
\begin{itemize}
\item $\Gamma_{x}$ changes at most
$A$ times of basic set and
\item $\Gamma_{x}[j,j'] \subset {\mathcal C}$ and
$0 \leq j \leq j' \leq T'(x) \implies j'-j \leq B/|x|^{e({\mathcal C})}$
\end{itemize}
for any compact-like set ${\mathcal C}$ and any $x \in \beta$.
\end{defi}
We compare orbits of $\varphi$ and ${\mathfrak F}_{\varphi}=\mathrm{exp}(X)$
by analyzing the sub-orbits contained in the basic sets of the dynamical
splitting. The first condition in Definition \ref{def:ab} is a natural finiteness property.
The second property
allows to apply Proposition \ref{pro:bouis} to
${\{\Gamma_{x}[0,T'(x)]\}}_{x \in \beta}$. They assure that the dynamics of
$\varphi$ and $\mathrm{exp}(X)$ are similar in a neighborhood of
${\{\Gamma_{x}[0,T'(x)]\}}_{x \in \beta}$.
\begin{rem}
It is clear that the trajectories of $\Re (X)$ ($N >1$) provided by
Propositions \ref{pro:exll} and \ref{pro:exll2}
are $(A,B)$ bounded for some values $A,B \in {\mathbb R}^{+}$.
Indeed all Long Trajectories are $(A,B)$ bounded (Proposition \ref{pro:uniform}).
\end{rem}
Let $\varphi \in \dif{p1}{2}$ with $N>1$.
A posteriori the Long Orbits of $\varphi$ are $(A,B)$ bounded for some values
$A,B \in {\mathbb R}^{+}$ that do not depend on the Long Orbit.
We introduce next these a priori bounds.
\begin{defi}
\label{def:B}
Let $X \in \Xnt$.  Consider the dynamical splitting associated to
$X$ in Section \ref{sec:dynspl}.
The number of boundary transversals (see Definition \ref{def:bdtr}) is bounded by a number
$A_{X} \in {\mathbb N}$ depending only on $X$.
Let ${\mathcal C}_{j}$ be a compact-like set and
$\lambda_{0} \in {\mathbb S}^{1}$. We define
$B_{j,\lambda_{0}}^{0}$ as the maximum of the periods of the closed
trajectories of $X_{j}(\lambda_{0})$  (Definition \ref{def:levels}).
We define $B_{j,\lambda_{0}}^{1}$ as the maximum of the values $s \in {\mathbb R}^{+}$
such that exists a trajectory $\Gamma[0,s]$ of $\Re (X_{j}(\lambda_{0}))$ contained in
${\mathcal C}_{j}(0,\lambda_{0})$ and not contained in a closed trajectory of
$\Re (X_{j}(\lambda_{0}))$ in ${\mathcal C}_{j}(0,\lambda_{0})$.
We define
$B_{\digamma}^{k}= 1+\max_{1 \leq j \leq q, \ \lambda \in {\mathcal U}_{X}^{1}} B_{j,\lambda}^{k}$
for $k \in \{0,1\}$ and
$B_{\digamma} = \max (B_{\digamma}^{0},B_{\digamma}^{1})$.
\end{defi}
\begin{rem}
The polynomial vector fields associated to a
convergent normal form $X$ of $\varphi \in \diff{p1}{2}$
only depend on $\varphi$. Thus the constant $B_{\digamma}$
depends only on $\varphi$ and the dynamical splitting
$\digamma$.
\end{rem}
\begin{defi}
\label{def:A}
Let $\varphi \in \diff{p1}{2}$ with
convergent normal form $X$. The number $A_{X}$
depends on $\varphi$ but not on $X$. We denote $A_{\varphi}=A_{X}$.
We denote $B_{\varphi} = B_{\digamma_{X}}$.
Of course $B_{\varphi}$ depends on the choice of $\digamma_{X}$
(see Definition \ref{def:dynfx}).
\end{defi}
\begin{lem}
\label{lem:abist}
Let $X \in \Xnt$ with $N \geq 1$.
Let $\beta$ be a germ of connected set at $0 \in {\mathbb C}$.
Consider a family of sub-trajectories
${\{\Gamma_{x}[0,T'(x)]\}}_{x \in \beta}$ of $\Re (X)$ such that
$\lim_{x \to \beta} \Gamma_{x}(0)$ exists and it is not $(0,0)$.
Suppose that the family is $(A,B)$ bounded.
Then ${\{\Gamma_{x}[0,T'(x)]\}}_{x \in \beta}$ is stable.
\end{lem}
A family ${\{\Gamma_{x}[0,T'(x)]\}}_{x \in \beta}$ that is $(A,B)$
bounded and stable satisfies the hypotheses in Propositions
\ref{pro:boufespre} and \ref{pro:bouis} that guarantee tracking
in basic sets. The lemma shows that the stability condition
is superfluous.
\begin{proof}
Suppose that it is not stable.
There exists a sequence $x_{n} \in \beta$, $x_{n} \to 0$
such that $\Gamma_{x_{n}}[0,T'(x_{n})] \cap {\mathcal E} \neq \emptyset$,
and $x_{n}/|x_{n}|$ tends to
$\lambda_{0} \in {\mathcal U}_{X}^{{\mathcal E},1}$
for some non-parabolic exterior set ${\mathcal E}$.
The exterior set ${\mathcal E}$ is enclosed by a compact-like
set ${\mathcal C}$. Let $X_{\mathcal C}(\lambda)$ be the polynomial
vector field associated to ${\mathcal C}$.
Since the point in $(\mathrm{Sing} X \cap {\mathcal E})(0,\lambda_{0})$ is indifferent
for $X_{\mathcal C}(\lambda_{0})$ then $\Gamma_{x_{n}}[0,T'(x_{n})]$
adheres in (adapted coordinates) to all periodic trajectories
in ${\mathcal C}(0,\lambda_{0})$
enclosing $(\mathrm{Sing} (X) \cap {\mathcal E})(0,\lambda_{0})$.
Periodic trajectories in ${\mathcal C}$ never quit ${\mathcal C}$.
Thus the family does not satisfy the last condition in
Definition \ref{def:ab}.
\end{proof}
\begin{defi}
\label{def:fam}
Let $\varphi \in \diff{p1}{2}$ with
$N > 1$. Fix a convergent normal form $X$ of $\varphi$.
Suppose that ${\mathcal O}$ is either a weak Long Trajectory
$(X, y_{+},\beta,T)$ or a Long Orbit
$(\varphi, y_{+},\beta,T)$.
Consider $\Gamma_{x}=\Gamma(X, \upsilon_{\mathcal O}(x), U_{\epsilon})$ for $x \in \beta$.
A sub-family associated to ${\mathcal O}$ is a family of the form
${\{ \Gamma_{x}[0,T_{1}(x)]\}}_{x \in \beta}$ for some function
$0 \leq T_{1} \leq T$.
If $[0,T(x)] \subset {\mathcal I}(\Gamma_{x})$ for any
$x \in \beta$ we say that
${\{ \Gamma_{x}[0,T(x)]\}}_{x \in \beta}$ is the family associated to ${\mathcal O}$
\end{defi}
In order to prove that a Long Orbit tracks its associated family it
suffices to show that it is $(A,B)$ bounded by Lemma \ref{lem:abist}.
We  briefly outline the proof. First we see that
there is tracking for $(A,B)$ bounded sub-families (Proposition \ref{pro:tracking}).
Then we prove that $(A,B)$ boundness plus tracking implies that the families
associated to Long Orbits satisfy a Rolle property (Proposition \ref{pro:rolle}).
Finally if the families associated to Long Orbits are not $(A,B)$ bounded we
construct $(A,B)$-bounded subfamilies that do not satisfy the Rolle property,
obtaining a contradiction. Along the way we obtain a formula for the
length of Long Orbits (Propositions \ref{pro:ltlo} and \ref{pro:lores}).
\subsection{The residue formula for diffeomorphisms}
\label{subsec:resdif}
In this section we show that Long Trajectories of a convergent normal form
induce Long Orbits of a diffeomorphism. We also obtain a generalization of the
residue formula.
\begin{pro}
\label{pro:tracking}
Let $\varphi \in \dif{p1}{2}$ with
$N > 1$. Fix a convergent normal form $X$ of $\varphi$.
Suppose that ${\mathcal O}$ is either a weak Long Trajectory
$(X, y_{+},\beta,T)$ or a Long Orbit
$(\varphi, y_{+},\beta,T)$.
Then, up to trimming ${\mathcal O}$, any $(A,B)$ bounded sub-family
${\{ \Gamma_{x}[0,T_{1}(x)]\}}_{x \in \beta}$
of ${\mathcal O}$ satisfies that (see Definition \ref{def:normal})
\[ \varphi^{j}(\upsilon_{\mathcal O}(x)) \in
B_{X}({\mathfrak F}_{\varphi}^{j}(\upsilon_{\mathcal O}(x)),1) \]
for all $0 \leq j \leq \lceil T_{1}(x) \rceil$ and $x \in \beta$.
Moreover, if ${\{ \Gamma_{x}[0,T(x)]\}}_{x \in \beta}$ is $(A,B)$ bounded
and $\lim_{n \to \infty} \varphi^{\lceil T(x_{n} \rceil}(\upsilon_{\mathcal O}(x_{n}))$
converges to $(0,y_{-})$ for some sequence $x_{n} \in \beta$ then
\begin{equation}
\label{equ:noise}
\sum_{j=0}^{\lceil T(x_{n} \rceil -1} \Delta_{\varphi}(\varphi^{j}(\upsilon_{\mathcal O}(x_{n}))) -
\sum_{j=0}^{\infty} \Delta_{\varphi}(\varphi^{j}(0,y_{+}))
-
\sum_{j=1}^{\infty} \Delta_{\varphi}(\varphi^{-j}(0,y_{-}))
\end{equation}
converges to $0$ when $n \to \infty$ (see Definition \ref{def:delta}).
\end{pro}
The idea is that Long Orbits of elements of $\dif{p1}{2}$ have good tracking properties
if their associated families are $(A,B)$ bounded. The convergence to $0$ of the
expression in Eq. (\ref{equ:noise})
is key to generalize the residue formula for diffeomorphisms.
\begin{rem}
Let $\varphi \in \diff{p1}{2}$. Consider
a convergent normal form $X$ of $\varphi$.
There exists $\epsilon'>0$ such that
\[ \left| \sum_{j \geq 0} \Delta_{\varphi} \circ \varphi^{j}(0,y_{+}) \right| < \frac{1}{4}
\ \ \mathrm{and} \ \
 \left| \sum_{j \geq 1} \Delta_{\varphi} \circ \varphi^{-j}(0,y_{-})  \right| < \frac{1}{4} . \]
for all $(0,y_{+})$ in an attracting petal of $\Re (X)_{|U_{\epsilon'}(0)}$ and
$(0,y_{-})$ in a repelling petal.
From now on and up to trimming we suppose that Long Orbits are contained in
$U_{\epsilon'}$.
\end{rem}
\begin{proof}[Proof of Proposition \ref{pro:tracking}]
We have $\nu_{\mathcal B} (\Delta_{\varphi}) - e({\mathcal B})>0$
for any basic set ${\mathcal B}$ different than the first exterior set ${\mathcal E}_{0}$.
On the other hand we have $\nu_{{\mathcal E}_{0}} (\Delta_{\varphi}) - e({\mathcal E}_{0})=0$
if $m=0$.
Let us use the notations for families and sub-families in Definition \ref{def:fam}.

Fix $0 < \xi <1/(4(A+1))$. The first exterior set is parabolic,
so we can choose $\epsilon'' >0$ such that
Eq. (\ref{equ:trackes}) in Proposition \ref{pro:boufespre}
holds for trajectories contained in
\[ {\mathcal E} = \{(x,y) \in B(0,\delta) \times {\mathbb C} :  \epsilon'' \geq |y| \geq \rho|x| \}. \]
We claim that there exists $M'>0$ such that
$\Gamma_{x}[M',\min(T_{1}(x),T(x)-M')]$ is contained in $U_{\epsilon''}$ for any $x \in \beta$.
This is obvious if ${\mathcal O}$ is a weak Long Trajectory. Let us prove it for Long Orbits.
We choose $M' \in {\mathbb N}$ satisfying
that
\begin{equation}
\label{equ:auxu}
\{\varphi^{M'}(\upsilon(x)), \hdots, \varphi^{\lceil T(x) \rceil -M'}(\upsilon(x))\}
\subset U_{\tilde{\epsilon}}
\end{equation}
for some $\tilde{\epsilon} >0$ such that
$\cup_{z \in B(0,2)} \mathrm{exp}(z X)(\overline{U_{\tilde{\epsilon}}}) \subset U_{\epsilon''}$.
If the property does not hold true we define
\[ T_{2}(x) = \min \{ s \in [M',\min(T_{1}(x),T(x)-M')]: \Gamma_{x}(s) \not \in U_{\epsilon''} \}; \]
it is well-defined for a sequence $x_{n} \in \beta$, $x_{n} \to 0$.
The family ${\{\Gamma_{x}[0,T_{2}(x)]\}}_{x \in \beta}$ is stable by
Lemma \ref{lem:abist}.
Propositions \ref{pro:boufespre} and \ref{pro:bouis} imply that
\[ |\psi_{X} (\varphi^{[T_{2}(x_{n})]}(\upsilon_{\mathcal O}(x_{n}))) -
\psi_{X}(\mathrm{exp}([T_{2}(x_{n})] X)(\upsilon_{\mathcal O}(x_{n})))| \leq \frac{1}{4} + A \xi < 1 \
\mathrm{for} \ n >>0. \]
The left hand side of the previous equation is greater than $2-1=1$
by Eq. (\ref{equ:auxu}) and the choice of $T_{2}$.
We obtain a contradiction.

The family ${\{\Gamma_{x}[0,T_{1}(x)]\}}_{x \in \beta}$ is stable by
Lemma \ref{lem:abist}.
Propositions \ref{pro:boufespre} and \ref{pro:bouis} imply
\[ |\psi_{X} (\varphi^{j}(\upsilon_{\mathcal O}(x))) -
\psi_{X}({\mathfrak F}_{\varphi}^{j}(\upsilon_{\mathcal O}(x)))| < \frac{1}{4} + \frac{1}{4} + A \xi < 1  \]
for all $0 \leq j \leq \lceil T_{1}(x) \rceil$ and $x \in \beta$ in a neighborhood of $0$.

Suppose that   ${\{ \Gamma_{x}[0,T(x)]\}}_{x \in \beta}$ is $(A,B)$ bounded and
$\lim_{n \to \infty} \varphi^{\lceil T(x_{n} \rceil}(\upsilon_{\mathcal O}(x_{n})) =(0,y_{-})$.
We denote by $G(x_{n})$ the expression in Eq. (\ref{equ:noise}).
We define
\[ G_{0} = \sum_{j=0}^{\infty} \Delta_{\varphi}(\varphi^{j}(0,y_{+})) \ \mathrm{and} \
G_{1} = \sum_{j=1}^{\infty} \Delta_{\varphi}(\varphi^{-j}(0,y_{-})). \]
We have
\[ G(x_{n})=\psi_{X} (\varphi^{\lceil T(x_{n}) \rceil}(\upsilon(x_{n}))) -
\psi_{X}(\mathrm{exp}(\lceil T(x_{n}) \rceil X)(\upsilon(x_{n})))
- G_{0} - G_{1} \]
Given $0 < \xi <1/(4(A+1))$
Propositions \ref{pro:boufespre} and \ref{pro:bouis} imply
$|G(x_{n})| \leq o(1) + A \xi$ for $n>>1$.
We deduce that $\lim_{n \to \infty} G(x_{n}) =0$.
\end{proof}
The next proposition is the analogue of Remark \ref{rem:udll} for Long Orbits. The non-existence
of Long Orbits is a generic phenomenon in the parameter space.
\begin{pro}
\label{pro:udlo}
Let $\varphi \in \dif{p1}{2}$ with
$N > 1$. Fix a convergent normal form $X$ of $\varphi$.
Consider a Long Orbit
${\mathcal O}=(\varphi, y_{+},\beta,T)$
such that ${\mathcal S}_{\mathcal O}$ is compact.
Then $\beta$ adheres a unique direction in ${\mathcal U}_{X}^{1}$.
\end{pro}
\begin{proof}
Fix $\lambda_{0} \in {\mathbb S}^{1} \setminus {\mathcal U}_{X}^{1}$.
Consider a compact connected small neighborhood $K$ of $\lambda_{0}$
in ${\mathbb S}^{1} \setminus {\mathcal U}_{X}^{1}$ and
$\tilde{\beta} = (0,\delta) K$.
Up to trimming the Long Orbit
we can suppose that $(0,y_{+})$ is in
an attracting petal of $\Re (X)_{|U_{\epsilon}(0)}$.
Fix the dynamical splitting $\digamma_{K}$ provided by
Lemma \ref{lem:goins}. Thus
given $\epsilon''>0$
there exists
$M \in {\mathbb N}$ such that
${\mathfrak F}_{\varphi}^{j}(x,y_{+}) \in U_{\epsilon''}$
for all $j \geq M$ and $x \in \tilde{\beta}$ close to $0$.
Corollary \ref{cor:stdir} implies
$\lim_{n \to \infty} {\mathfrak F}_{\varphi}^{n}(x,y_{+}) \in \mathrm{Fix} (\varphi)$
for any $x \in \tilde{\beta}$ close to $0$.
Consider any family of sub-trajectories
$\Gamma_{x}:[0,T'(x)] \to U_{\epsilon}(x)$ of $\Re (X)$
defined for $x \in \tilde{\beta}$ and such that
$\lim_{x \in \tilde{\beta}, \ x \to 0} \Gamma_{x}(0) =(0,y_{+})$.
Lemma \ref{lem:goins} implies that ${\{ \Gamma_{x}\}}_{x \in \tilde{\beta}}$
is $(A,1+\max_{1 \leq j \leq q} B_{j,\lambda_{0}}^{1})$ bounded for some $A \in {\mathbb R}^{+}$
that depends only on $X$
(see Definition \ref{def:B}).
We can proceed as in the proof of Proposition \ref{pro:tracking} to
show that
$\varphi^{j}(\upsilon_{\mathcal O}(x)) \in
B_{X}({\mathfrak F}_{\varphi}^{j}(\upsilon_{\mathcal O}(x)),1)$
for all $j \geq 0$ and $x \in \tilde{\beta}$.
We deduce that $\tilde{\beta} \cap \beta$ does not contain any
point in the neighborhood of $0$.
Hence $\beta_{\pi} \cap  (\{0\} \times {\mathbb S}^{1})$ is a singleton
since it
is a connected
set contained in $\{0\} \times {\mathcal U}_{X}^{1}$.
\end{proof}
The residue formula (\ref{equ:deftf}) for Long Trajectories of
$\Re (X)$ with $X \in \Xnt$ involves Fatou coordinates
$\psi_{+}$ and $\psi_{-}$ of $X$. In order to obtain a generalization
for Long Orbits of $\varphi \in \diff{p1}{2}$
it is natural to replace the previous functions with
Fatou coordinates of $\varphi_{|x=0}$.
\begin{defi}
\label{def:fatl}
Let $\varphi \in \dif{p1}{2}$ with
$N > 1$. Fix a convergent normal form $X$ of $\varphi$.
Consider an attracting petal ${\mathcal P}_{+}'$
and a repelling petal ${\mathcal P}_{-}'$ of $\varphi_{|x=0}$. We define
\[ \psi_{{\mathcal P}_{+}'}^{\varphi}(0,y) = \psi_{+}(0,y) + \sum_{j=0}^{\infty} \Delta_{\varphi}(\varphi^{j}(0,y)), \ \
\psi_{{\mathcal P}_{-}'}^{\varphi}(0,y) = \psi_{-}(0,y) - \sum_{j=1}^{\infty} \Delta_{\varphi}(\varphi^{-j}(0,y)) \]
in ${\mathcal P}_{+}'$ and ${\mathcal P}_{-}'$ respectively
where $\psi_{+}$, $\psi_{-}$ are Fatou coordinates of $X$.
The function $\psi_{{\mathcal P}_{j}'}^{\varphi}$
is a Fatou coordinates of $\varphi_{|{\mathcal P}_{j}'}$, i.e
$\psi_{{\mathcal P}_{j}'}^{\varphi} \circ \varphi \equiv \psi_{{\mathcal P}_{j}'}^{\varphi} +1$ for
$j \in \{+,-\}$.
\end{defi}
We introduce the main result of this section.
\begin{pro}
\label{pro:ltlo}
Let $\varphi \in \dif{p1}{2}$ with
$N > 1$. Fix a convergent normal form $X$ of $\varphi$.
Suppose that ${\mathcal O}=(X, y_{+},\beta,T)$ is a weak Long Trajectory and
that ${\{ \Gamma_{x}[0,T(x)]\}}_{x \in \beta}$ is $(A,B)$ bounded.
Suppose that, up to trimming ${\mathcal O}$,
$\varphi^{\lceil T(x_{n}) \rceil}(\upsilon_{\mathcal O}(x_{n}))$
converges to $(0,y_{-}) \neq (0,0)$
for some sequence $x_{n} \in \beta$, $x_{n} \to 0$. Then we obtain
\begin{equation}
\label{equ:ltlo1}
\psi_{{\mathcal P}_{-}'}^{\varphi}(0,y_{-}) - \psi_{{\mathcal P}_{+}'}^{\varphi}(0,y_{+}) =
\lim_{n \to \infty} \left( \lceil T(x_{n}) \rceil + 2 \pi i \sum_{Q \in E_{-}(x_{n})} Res(X,Q) \right)
\end{equation}
where $(E_{-},E_{+})$ is the division of $\mathrm{Sing}(X)$ induced by ${\mathcal O}$.
Suppose now that ${\mathcal O}$ is a Long Trajectory. Then
${\mathcal O}'=(\varphi, y_{+},\beta,T)$
is a Long Orbit. Moreover,
${\mathcal O}'$ satisfies ${\mathcal S}_{{\mathcal O}'} = {\mathcal S}_{\mathcal O}$ and
\[ \psi_{{\mathcal P}_{-}'}^{\varphi}(\chi_{{\mathcal O}'}(s+i u))
- \psi_{{\mathcal P}_{+}'}^{\varphi}(0,y_{+}) =
\lim_{\vartheta_{\mathcal O}(x) \to u, \ x \to 0}^{\lceil T(x) \rceil - T(x) \to s}
\left(
\lceil T(x) \rceil + 2 \pi i \sum_{Q \in E_{-}(x)} Res(X,Q) \right) \]
for any $s+i u \in [0,1] + i {\mathcal S}_{{\mathcal O}'}$.
In particular we get
$\psi_{{\mathcal P}_{-}'}^{\varphi}(\chi_{{\mathcal O}'}(z)) =
\psi_{{\mathcal P}_{-}'}^{\varphi}(\chi_{{\mathcal O}'}(0)) +z$
for any $z \in [0,1] + i {\mathcal S}_{{\mathcal O}'}$.
\end{pro}
Let us remark that $(0,y_{+})$ belongs to an attractive petal ${\mathcal P}_{+}'$
and all possible limits of sequences of the form
$\varphi^{\lceil T(x_{n}) \rceil}(\upsilon_{\mathcal O}(x_{n}))$
are contained in a repelling petal ${\mathcal P}_{-}'$ of $\varphi_{|U_{\epsilon}(0)}$.

The family associated to a weak Long Trajectory
is always $(A_{\varphi},B_{\varphi})$ bounded.
A direct proof is not difficult and it is also a consequence
of Proposition \ref{pro:uniform}.
The corresponding hypothesis in Proposition \ref{pro:ltlo} is
a posteriori unnecessary.
Long Trajectories of $\Re (X)$ always induce Long Orbits of $\varphi$.
\begin{proof}
Denote $\upsilon^{1}(x)=\mathrm{exp}(T(x)X)(\upsilon_{\mathcal O}(x))$.
By defining $\psi_{+}$ and $\psi_{-}$ as in Subsection \ref{subsec:res}
we obtain
\begin{equation}
\label{equ:ltlo3}
T(x) = \psi_{-}(\upsilon^{1}(x)) - \psi_{+}(\upsilon_{\mathcal O}(x))  -
2 \pi i \sum_{Q \in E_{-}(x)} Res(X,Q) \ \ \forall  x \in \beta
\end{equation}
for  some division $(E_{-},E_{+})$ of $\mathrm{Sing} (X)$.
Denote $G(x)=- 2 \pi i \sum_{Q \in E_{-}(x)} Res(X,Q)$.
We want to express $T$ as a function of data depending on $\varphi$.
Since
\[ \psi_{-} \circ \varphi^{\lceil T(x) \rceil}(\upsilon_{\mathcal O}(x)) -
\psi_{-} \circ  \mathrm{exp}(\lceil T(x) \rceil X) (\upsilon_{\mathcal O}(x)) =
\sum_{j=0}^{\lceil T(x) \rceil -1} \Delta_{\varphi}(\varphi^{j}(\upsilon_{\mathcal O}(x))) \]
(see Definition \ref{def:delta}) we obtain
\[ \lceil T(x) \rceil= \psi_{-} \circ \varphi^{\lceil T(x) \rceil}(\upsilon_{\mathcal O}(x)) -
\sum_{j=0}^{\lceil T(x) \rceil-1} \Delta_{\varphi}(\varphi^{j}(\upsilon_{\mathcal O}(x)))
 - \psi_{+}(\upsilon_{\mathcal O}(x))  + G(x) \]
for  any $x \in \beta$. We obtain
$\varphi^{\lceil T(x) \rceil}(\upsilon_{\mathcal O}(x)) \in
B_{X}({\mathfrak F}_{\varphi}^{\lceil T(x) \rceil}(\upsilon_{\mathcal O}(x)),1)$
by the tracking phenomenon. Thus any
$\varphi^{\lceil T(x_{n}) \rceil}(\upsilon_{\mathcal O}(x_{n}))$ has a convergent subsequence.
Suppose that $\varphi^{\lceil T(x_{n}) \rceil}(\upsilon_{\mathcal O}(x_{n}))$
converges to $(0,y_{-}) \neq (0,0)$.
Proposition \ref{pro:tracking} implies Eq. (\ref{equ:ltlo1}), see Definition \ref{def:fatl}.

Suppose that ${\mathcal O}$ is a Long Trajectory. We have
\[ \lim_{\vartheta_{\mathcal O}(x) \to u, \ x \to 0}^{\lceil T(x) \rceil - T(x) \to s}
\left(
\lceil T(x) \rceil + 2 \pi i \sum_{Q \in E_{-}(x)} Res(X,Q) \right)=
\psi_{-}(\chi_{\mathcal O}(iu)) -\psi_{+}(0,y_{+}) + s \]
for all $s \in [0,1]$ and $u \in {\mathcal S}_{\mathcal O}$.
We obtain
\begin{equation}
\label{equ:ltlo2}
\psi_{{\mathcal P}_{-}'}^{\varphi}(\chi_{{\mathcal O}'}(s+i u)) - \psi_{{\mathcal P}_{+}'}^{\varphi}(0,y_{+}) =
\psi_{-}(\chi_{\mathcal O}(iu)) -\psi_{+}(0,y_{+}) + s .
\end{equation}
by Eq. (\ref{equ:ltlo1}) since
$\psi_{{\mathcal P}_{-}'}^{\varphi}$ is injective ${\mathcal P}_{-}'$.
Moreover $\chi_{{\mathcal O}'}$ satisfies
\[ \psi_{{\mathcal P}_{-}'}^{\varphi}(\chi_{{\mathcal O}'}(z)) -
\psi_{{\mathcal P}_{-}'}^{\varphi}(\chi_{{\mathcal O}'}(0)) = z \]
for any $z \in [0,1] + i {\mathcal S}_{{\mathcal O}'}$.
We deduce $\chi_{{\mathcal O}'}(1+iu) = \varphi(\chi_{{\mathcal O}'}(iu))$
for any $u \in {\mathcal S}_{{\mathcal O}'}$.
\end{proof}
\begin{rem}
The Long Trajectories provided in Propositions \ref{pro:exll} and \ref{pro:exll2}
are $(A,B)$ bounded. Thus given $\varphi \in \dif{p1}{2}$ with $N>1$ and a point
$(0,y_{+})$ contained in an attracting petal of $\varphi_{|x=0}$
there exists a Long Orbit $(\varphi,y_{+},\beta,T)$.
\end{rem}
\subsection{The Rolle property}
\label{sec:Rolle}
Let $\varphi \in \dif{p1}{2}$ with $N>1$ and let $X$ be a convergent normal form.
Long Trajectories of $\Re (X)$ induce Long Orbits of $\varphi$
but the reciprocal is not clear.
If $\varphi(x,y)=(x,f(x,y))$ is multi-parabolic, i.e. if
$(\partial f/\partial y)_{|\mathrm{Fix} (\varphi)} \equiv 1$ the situation is much simpler.
Indeed trajectories of $\Re (X)$ satisfy the Rolle property, i.e.
they do not intersect twice connected transversals (Proposition 2.1.1 of \cite{rib-mams}).
In particular $\Re (X)$ has no closed trajectories.
We deduce that any family of trajectories of $\Re (X)$ is $(A_{\varphi},B_{\varphi})$
bounded. This situation is quite special and corresponds to the case when
orbits of $\varphi$ {\it always} track orbits of ${\mathfrak F}_{\varphi}$.
In the general case the Rolle property still holds true for families associated to Long Orbits.
\begin{defi}
\label{def:rolle}
Let $\varphi \in \dif{p1}{2}$ with
$N > 1$. Fix a convergent normal form $X$ of $\varphi$.
Suppose that ${\mathcal O}=(\varphi, y_{+},\beta,T)$ generates a Long Orbit.
We say that a sub-family
${\{ \Gamma_{x}[0,T_{1}(x)]\}}_{x \in \beta}$
of ${\mathcal O}$ satisfies the Rolle property if there is no choice of
a sequence $x_{n} \in \beta$, $x_{n} \to 0$ such that
\begin{itemize}
\item There exist $0 \leq T_{2}(x_{n}) < T_{3}(x_{n}) \leq T_{1}(x_{n})$
for any $n \in {\mathbb N}$
such that $\lim_{n \to \infty} \Gamma_{x_{n}}(T_{j}(x_{n})) = (0,0)$ for
$j \in \{2,3\}$.
\item There exists a trajectory $\gamma_{n}:[0,T_{4}(x_{n})] \to U_{\epsilon}(x_{n})$
of $\Re (iX)$ or $\Re (-iX)$ such that $\gamma_{n}(0)=\Gamma_{x_{n}}(T_{2}(x_{n}))$
and $\gamma_{n}(T_{4}(x_{n}))=\Gamma_{x_{n}}(T_{3}(x_{n}))$ for any $n \in {\mathbb N}$.
\item Given any $\epsilon''>0$ there exists $n_{0} \in {\mathbb N}$ such that
$\gamma_{n}[0,T_{4}(x_{n})] \subset U_{\epsilon''}$ for any $n \geq n_{0}$.
\end{itemize}
We can always suppose that
$\Gamma_{x_{n}}[T_{2}(x_{n}), T_{3}(x_{n})] \cup \gamma_{n}[0,T_{4}(x_{n})]$
is a closed simple curve by changing slightly the trajectories.
We denote by $D_{n}$ the bounded component of
$(\{x_{n}\} \times {\mathbb C})
\setminus (\Gamma_{x_{n}}[T_{2}(x_{n}), T_{3}(x_{n})] \cup \gamma_{n}[0,T_{4}(x_{n})])$.
We define $\mathrm{gap}_{n}=T_{4}(x_{n})$.
\end{defi}
We prove that families associated to Long Orbits
are $(A,B)$ bounded by reductio ad absurdum.
More precisely we construct sub-families that are $(A,B)$ bounded
and fail to satisfy the Rolle property. This contradicts the
next proposition.
\begin{pro}
\label{pro:rolle}
Let $\varphi \in \dif{p1}{2}$ with
$N > 1$. Fix a convergent normal form $X$ of $\varphi$.
Suppose that ${\mathcal O}=(\varphi, y_{+},\beta,T)$ generates a Long Orbit.
Suppose that ${\mathcal S}_{\mathcal O}$ is a compact set.
Then, up to trimming ${\mathcal O}$, any $(A,B)$ bounded sub-family
${\{ \Gamma_{x}[0,T_{1}(x)]\}}_{x \in \beta}$
of ${\mathcal O}$ satisfies the Rolle property.
\end{pro}
\begin{proof}
Suppose that the Rolle property is not satisfied.
The set $\chi_{\mathcal O} ([0,1] + i {\mathcal S}_{\mathcal O})$ is compact and it does not
contain $(0,0)$.
The first condition in Definition \ref{def:rolle} and
Proposition \ref{pro:tracking} imply that
$\lim_{n \to \infty} T_{2}(x_{n}) =  \lim_{n \to \infty} (T-T_{3})(x_{n}) = \infty$
and that $\Gamma_{x_{n}}[T_{2}(x_{n}), T_{3}(x_{n})]$ converges to
$\{(0,0)\}$ in the
Hausdorff topology for compact sets. Hence
$\overline{D_{n}}$ converges to $\{(0,0)\}$
by the last condition in Definition \ref{def:rolle}.

If $\Re (X)$ points towards $D_{n}$ at $\gamma_{n}(0,T_{4}(x_{n}))$ then
$D_{n}$ is invariant by the positive flow of $\Re (X)$.
Otherwise $D_{n}$ is invariant by the negative flow of $\Re (X)$.
We claim that we are always in the former situation for $n>>0$.
Otherwise $\Gamma_{x_{n}}(0) \in \overline{D_{n}}$ for a subsequence and
we obtain a contradiction since
\[ (0,y_{+}) = \lim_{n \to \infty} \upsilon_{\mathcal O}(x_{n}) =
\lim_{n \to \infty} \Gamma_{x_{n}}(0) = (0,0) \]
and $y_{+} \neq 0$. The last equality is a consequence of
$\lim_{n \to \infty} \overline{D_{n}} = \{(0,0)\}$.

Consider a subsequence such that $\mathrm{gap}_{n} \leq K'$ for some
$K' \in {\mathbb R}^{+}$. Let us prove that
$\lim_{n \to \infty} (T_{3}-T_{2})(x_{n}) = \infty$.
The vector field $X_{|x=0}$ has a multiple singular point
at $(0,0)$. Thus the diffeomorphism
$(z,x,y) \mapsto (z, \mathrm{exp}(zX)(x,y))$ defined in a neighborhood of
${\mathbb C} \times \{(0,0)\}$ is of
the form $(z,x,y+z u(z,x,y) X(y))$ where $u(z,0,0) \equiv 1$.
Hence given $C \in {\mathbb R}^{+}$
there exists a neighborhood $W$ of $(0,0)$ in ${\mathbb C}^{2}$
such that $z \mapsto \mathrm{exp}(zX)(x,y)$ is injective in $B(0,C)$
for any $(x,y) \in W \setminus \mathrm{Sing} (X)$.
If   a subsequence satisfies $\mathrm{gap}_{n} \leq K'$
and $(T_{3}-T_{2})(x_{n}) < K''$ then
$\partial D_{n}$ and $D_{n}$ are contained in
$B_{X}(\Gamma_{x_{n}}(T_{2}(x_{n})),K'+K''+1)$
where clearly trajectories of $\Re (X)$ can not intersect
twice trajectories of $\Re (iX)$. We obtain
$\lim_{n \to \infty} (T_{3}-T_{2})(x_{n}) = \infty$. We have
\[  {\mathfrak F}_{\varphi}^{[T_{3}(x_{n})-T_{2}(x_{n})]}(\upsilon_{\mathcal O}(x_{n}))=
\mathrm{exp}([T_{3}(x_{n})-T_{2}(x_{n})]X)(\upsilon_{\mathcal O}(x_{n}))
\in B_{X}(\upsilon_{\mathcal O}(x_{n}),K'+1) .  \]
The tracking phenomenon
(Proposition \ref{pro:tracking}) implies that
\[ |\psi_{X} (\varphi^{[T_{3}(x_{n})-T_{2}(x_{n})]}(\upsilon_{\mathcal O}(x_{n}))) -
\psi_{X} ({\mathfrak F}_{\varphi}^{[T_{3}(x_{n})-T_{2}(x_{n})]}(\upsilon_{\mathcal O}(x_{n})))| < 1  . \]
Since
$\lim_{n \to \infty} (T_{3}-T_{2})(x_{n}) = \lim_{n \to \infty} T(x_{n})- (T_{3}-T_{2})(x_{n}) =\infty$
and ${\mathcal O}$ is a Long Orbit we have
$\lim_{n \to \infty} \varphi^{[T_{3}(x_{n})-T_{2}(x_{n})]}(\upsilon_{\mathcal O}(x_{n})) = (0,0)$.
This leads us to
\[ \lim_{n \to \infty} {\mathfrak F}_{\varphi}^{[T_{3}(x_{n})-T_{2}(x_{n})]}(\upsilon_{\mathcal O}(x_{n})) = (0,0) \
\mathrm{and} \
(0,y_{+}) = \lim_{n \to \infty} \upsilon_{\mathcal O}(x_{n}) =(0,0). \]
The last property contradicts $y_{+} \neq 0$.

Resuming $D_{n}$ is invariant by the positive flow of $\Re (X)$ and
$\lim_{n \to \infty} \mathrm{gap}_{n} = \infty$.
Let us suppose that $\mathrm{gap}_{n} \geq 4$ for any $n \in {\mathbb N}$.
Our goal is proving that there exists $s_{n} \in {\mathbb N}$ such that
$\varphi^{j}(B_{X}(\overline{D_{n}},1)) \subset B_{X}(\overline{D_{n}},2)$ and
$\varphi^{s_{n}}(B_{X}(\overline{D_{n}},1)) \subset D_{n}$ for  $0 \leq j <s_{n}$
and $n>>0$.
Consider the segment
\[ \eta_{n} = \Gamma(iX, \Gamma_{x_{n}}(T_{3}(x_{n})), U_{\epsilon})(-\mathrm{gap}_{n}+1,\mathrm{gap}_{n}-1) .\]
It satisfies $\mathrm{exp}(s X)(\eta_{n}) \subset D_{n}$ for all
$s \in {\mathbb R}^{+}$ and $n \in {\mathbb N}$.
We define
\[ \tilde{D}_{n} = \cup_{s \in {\mathbb R}^{+}} \mathrm{exp}(s X)(\eta_{n}), \ \
\iota_{n}^{c} = B_{X}(\Gamma_{x_{n}}[T_{2}(x_{n}), T_{3}(x_{n})],c) . \]

If $Q \in \iota_{n}^{2}$ there exists $s \in [T_{2}(x_{n}), T_{3}(x_{n})]$
such that $Q \in B_{X}(\Gamma_{x_{n}}(s),2)$. The tracking phenomenon implies
that $\varphi^{[T_{3}(x_{n})-s] +4}(Q) \in \tilde{D}_{n}$.

If $Q \in B_{X}(\overline{D_{n}},1) \setminus (\tilde{D}_{n} \cup \iota_{n}^{2} )$
then $Q$ is of the form $\mathrm{exp}(s X)(Q_{0})$ where $s \in [-1,0]$
and $Q_{0}$ belongs to $\gamma_{n}[\sqrt{3},T_{4}(x_{n})-\sqrt{3}]$.
We obtain that $\varphi^{2}(Q)$ belongs to $\tilde{D}_{n}$.
It suffices to prove that $\varphi^{j}$ is well-defined in $\tilde{D}_{n}$
and $\varphi^{j}(\tilde{D}_{n}) \subset D_{n}$
for all $j \geq 0$ and $n \in {\mathbb N}$.
We can define $s_{n}=[T_{3}(x_{n})-T_{2}(x_{n})] +4$.

Suppose that  $\varphi^{j}(Q) \not \in D_{n}$ for some $Q \in \tilde{D}_{n}$
and $j \in {\mathbb N}$. We can assume that
$\varphi^{k}(Q) \in D_{n}$ for $0 <k<j$. We obtain that
$\varphi^{j}(Q) \in B_{X}(\overline{D_{n}},\epsilon_{n})$ where
$\lim_{n \to \infty} \epsilon_{n}=0$.
We claim that $\varphi^{j}(Q) \in \iota_{n}^{1/2}$;
otherwise we obtain $\varphi^{j-1}(Q) \not \in D_{n}$.
The point $\varphi^{j}(Q)$ belongs to $B_{X}(\Gamma_{x_{n}}(s),1/2)$
for some
$s \in [T_{2}(x_{n}), T_{3}(x_{n})]$.
The tracking phenomenon implies
that there exists $0 \leq j' \leq [s-T_{2}(x)]+2$ such
that $\varphi^{-k}(\varphi^{j}(Q)) \in \iota_{n}^{3/4}$ for $0 \leq k <j'$ and
$\varphi^{-j'}(\varphi^{j}(Q)) \not \in D_{n}$. This is impossible since
$\varphi^{-1}(\iota_{n}^{3/4}) \cap \tilde{D}_{n} =\emptyset$.

It is clear that $\varphi^{s_{n}}$ contracts the Poincar\'{e} metric in $D_{n}$.
Thus $\varphi^{j s_{n}}$ converges uniformly to a point $P_{n}$ in
$B_{X}(\overline{D_{n}},1)$. The point $P_{n}$ is an attractor for $\varphi^{s_{n}}$
and then for $\varphi$.
The orbit
${\{ \varphi^{j}(Q)\}}_{j \in {\mathbb N}}$ is contained in
$B_{X}(\overline{D_{n}},2)$ for any $Q \in B_{X}(\overline{D_{n}},1)$
and $\lim_{j \to \infty} \varphi^{j}(Q) = P_{n}$.
The point $\varphi^{\lceil T_{2}(x_{n}) \rceil}(\upsilon_{\mathcal O}(x_{n}))$ belongs to
$B_{X}(\overline{D_{n}},1)$. We deduce that
$\varphi^{\lceil T(x_{n}) \rceil}(\upsilon_{\mathcal O}(x_{n}))$ belongs to $B_{X}(\overline{D_{n}},2)$
for $n>>0$. This implies that
there exists $z \in [0,1]+i {\mathcal S}_{\mathcal O}$
such that $\chi_{\mathcal O} (z)=(0,0)$.
This property contradicts Definition \ref{def:lo}.
\end{proof}
Let us remark that the analogue for Long Trajectories admits a much simpler proof.
It is a version of the proof of the case $\mathrm{gap}_{n} \not \to 0$.
Indeed this condition guarantees that the attracting nature of $D_{n}$
is stable under small deformations.
\begin{pro}
\label{pro:uniform}
Let $\varphi \in \dif{p1}{2}$ with
$N > 1$. Fix a convergent normal form $X$ of $\varphi$.
Suppose that ${\mathcal O}=(\varphi, y_{+},\beta,T)$ generates a Long Orbit.
Then, up to trimming ${\mathcal O}$, the family
${\{ \Gamma_{x}[0,T(x)]\}}_{x \in \tilde{\beta}}$
of ${\mathcal O}$ is $(A_{\varphi},B_{\varphi})$ bounded
(see Definition \ref{def:A})
where $\tilde{\beta}$ is a subset of $\beta$ such that
$\tilde{\mathcal O}=(\varphi, y_{+},\tilde{\beta},T)$
satisfies ${\mathcal S}_{\tilde{\mathcal O}} = {\mathcal S}_{\mathcal O}$.
\end{pro}
The proposition implies that the $(A,B)$ boundness condition is automatic
for Long Orbits.
Let us explain the setup of the proof.

A priori all the trajectories in the family
${\{ \Gamma_{x} \}}_{x \in \beta}$ are labeled {\textbf{(a)}}.
We want to have
\begin{equation}
\label{equ:pba}
\varphi^{j}(\upsilon_{\mathcal O}(x)) \in B_{X}({\mathfrak F}_{\varphi}^{j}(\upsilon_{\mathcal O}(x)),1)
\end{equation}
for any $0 \leq j \leq \lceil T(x) \rceil$.
Suppose that the property does not hold true for some $x \in \beta$.
Then we replace $T(x)$ with $T_{1}(x) \leq T(x)$ such that
Eq. (\ref{equ:pba}) holds true for $0 \leq j < \lceil T_{1}(x)  \rceil$
but it is false for $j= \lceil T_{1}(x)  \rceil$.
In such a case the label {\textbf{(a)}} for $\Gamma_{x}$ is replaced with {\textbf{(b)}}.
Otherwise we define $T_{1}(x)=T(x)$.

Given a trajectory $\Gamma_{x}[0,T_{1}(x)]$ we keep the label if
$\Gamma_{x}[0,T_{1}(x)]$ does not intersect twice a boundary transversal
(see Definition \ref{def:bdtr}).
Otherwise we replace the previous label with {\textbf{(c)}}.
We also replace $T_{1}(x)$ with a smaller or equal value such that
$\Gamma_{x}[0,T_{1}(x)]$ intersects twice a boundary transversal
$\gamma$ but $\Gamma_{x}[0,T_{1}(x))$ does not.
The curve $\gamma$ is contained in the exterior boundary of a basic set ${\mathcal B} \neq {\mathcal E}_{0}$.
It is easy to consider $T_{1}'(x)$ such that
$\Gamma_{x}[0,T_{1}'(x)]$
intersects twice a sub-trajectory of $\Re (iX)$ contained in ${\mathcal B}$
(passing through a point in $T{\mathcal B}_{i X}(x)$)
and either $T_{1}'(x) \leq T_{1}(x)$ or $T_{1}'(x) > T_{1}(x)$ and
$\Gamma_{x}[T_{1}(x),T_{1}'(x)]$ is contained in
${\mathcal B}$ (see Figure (\ref{dr14})).
\begin{figure}[h]
\begin{center}
\includegraphics[height=5cm,width=10cm]{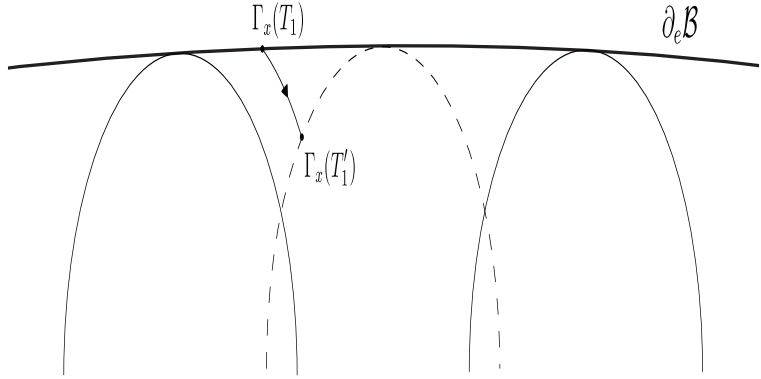}
\end{center}
\caption{The trajectories of $\Re (X)$ and $\Re (iX)$ are solid and dashed
thin curves respectively; the boundary of ${\mathcal B}$ is a thick curve} \label{dr14}
\end{figure}
The trajectories $\Gamma_{x}[0,T_{1}(x)]$ and $\Gamma_{x}[0,T_{1}'(x)]$
change  of basic set a number of times that is bounded
a priori.

We replace the label if   there exists a compact-like set ${\mathcal C}'$,
a sequence $x_{n} \in \beta$, $x_{n} \to 0$ and
sequences $0 \leq j_{n} < j_{n}' \leq T_{1}(x_{n})$ such that
\begin{equation}
\label{equ:auxlb}
\Gamma_{x_{n}}[j_{n},j_{n}'] \subset {\mathcal C}', \ \mathrm{and} \
|x_{n}|^{e({\mathcal C}')} (j_{n}' - j_{n}) > B_{\varphi} \ \forall n \in {\mathbb N}.
\end{equation}
We can refine the sequence and consider a smaller value of $T_{1}$, $T_{1}< T_{2} \leq T$
and a compact-like basic set ${\mathcal C}$ such  that
\[ \Gamma_{x_{n}}(T_{1}(x_{n})) \in \partial {\mathcal C}, \
\Gamma_{x_{n}}[T_{1}(x_{n})), T_{2}(x_{n}))] \subset {\mathcal C}, \
|x_{n}|^{e({\mathcal C})} (T_{2}  - T_{1})(x_{n}) > B_{\varphi} \]
for any $n \in {\mathbb N}$
and there is no choice of subsequences $j_{n}$, $j_{n}'$ and a compact-like set
${\mathcal C}'$ such that Eq. (\ref{equ:auxlb}) holds true.
We replace the previous label with {\textbf{(d)}} for the parameters in
the sequence $x_{n}$.
\begin{proof}
Let us prove that the labels {\textbf{(b)}}, {\textbf{(c)}} and {\textbf{(d)}}
never happen if ${\mathcal S}_{\mathcal O}$ is compact.
The set $\beta$ adheres a unique direction $\lambda_{0}$ in ${\mathcal U}_{X}^{1}$
(Proposition \ref{pro:udlo}).
Let us remark that $T_{1}(x) \leq T(x)$
and $T_{2}(x) \leq T(x)$ for any $x \in \beta$.

Suppose that we are in the situation {\textbf{(d)}}.
The
family ${\{ \Gamma_{x_{n}}[0,T_{1}(x_{n})]\}}_{n \in {\mathbb N}}$ is
$(A_{\varphi},B_{\varphi})$ bounded by the previous construction.
The sequence $x_{n}/|x_{n}|$ tends to $\lambda_{0}$.
We can also suppose that $\Gamma_{x_{n}}(T_{1}(x_{n}))$ converges to
a point $(0,w_{0})$ in the adapted coordinates $(x,w)$ of ${\mathcal C}$ up to consider
a subsequence. Let $X_{\mathcal C}(\lambda)$ be the polynomial vector field
associated to ${\mathcal C}$. Since
$X /|x|^{e({\mathcal C})} \to X_{\mathcal C}(\lambda_{0})$ in ${\mathcal C}$
if $x \to 0$ and $x/|x| \to \lambda_{0}$ the condition
$|x_{n}|^{e({\mathcal C})} (T_{2}  - T_{1})(x_{n}) > B_{\varphi}$ for any $n \in {\mathbb N}$
implies that the trajectory $\Gamma$ of $\Re (X_{\mathcal C}(\lambda_{0}))$ through
$(r,\lambda,w)=(0,\lambda_{0},w_{0})$ in ${\mathcal C}$ is closed.
The period of $\Gamma$ is less or equal than $B_{\varphi}-1$. We can replace
$T_{2}(x_{n})$ with $T_{1}(x_{n}) + B_{\varphi} / |x_{n}|^{e({\mathcal C})}$.
It is clear that $\Gamma_{x_{n}}[0,T_{2}(x_{n})]$ is $(A_{\varphi},B_{\varphi})$ bounded
and does not satisfy the Rolle property for $n >>0$ since it adheres
a periodic trajectory. This contradicts Proposition \ref{pro:rolle}.

Suppose that we are in the situation {\textbf{(c)}}, i.e.
the set $E$ of parameters in $\beta$ having the label {\textbf{(c)}}
adheres $0$.
The families ${\{ \Gamma_{x}[0,T_{1}(x)]\}}_{x \in E}$ and
${\{ \Gamma_{x}[0,T_{1}'(x)]\}}_{x \in E}$ are $(A_{\varphi}+1,B_{\varphi})$ bounded
by construction.
If the set $\{x \in E : T_{1}'(x) < T(x) \}$ adheres $0$
then we obtain a violation of the Rolle property (Proposition  \ref{pro:rolle}).
Otherwise we have $T_{1}(x) \leq T(x) \leq T_{1}'(x)$ for any $x \in \beta$.
The tracking phenomenon
(Proposition \ref{pro:tracking}) implies that
$\varphi^{\lceil T(x_{n}) \rceil}(\upsilon_{\mathcal O}(x_{n}))$ tends to $(0,0)$.
This contradicts the definition of Long Orbit.

Suppose that we are in the situation {\textbf{(b)}}, i.e.
the set $E$ of parameters in $\beta$ having the label {\textbf{(b)}}
adheres $0$.
The family ${\{ \Gamma_{x}[0,T_{1}(x)]\}}_{x \in E}$ is
$(A_{\varphi},B_{\varphi})$ bounded by construction.
This setup is incompatible with the tracking phenomenon
(Proposition \ref{pro:tracking}).

We can suppose that every point $x \in \beta$ has the label {\textbf{(a)}}.
As a consequence the family associated to ${\mathcal O}$ is
$(A_{\varphi},B_{\varphi})$ bounded.

Consider the general case, i.e. ${\mathcal S}_{\mathcal O}$ is not necessarily compact.
Let $\beta_{n}$ be a neighborhood  of $0$ in $\vartheta_{\mathcal O}^{-1}[-n,n]$
such that ${\{ \Gamma_{x}[0,T(x)]\}}_{x \in \beta_{n}}$ is
$(A_{\varphi},B_{\varphi})$ bounded for $n \in {\mathbb N}$. The family
${\{ \Gamma_{x}[0,T(x)]\}}_{x \in \tilde{\beta}}$ is
$(A_{\varphi},B_{\varphi})$ bounded for $\tilde{\beta} = \cup_{n \in {\mathbb N}} \beta_{n}$.
\end{proof}
Next we provide the generalization of the residue formula in the discrete case.
\begin{pro}
\label{pro:lores}
Let $\varphi \in \dif{p1}{2}$ with
$N > 1$. Fix a convergent normal form $X$ of $\varphi$.
Suppose that ${\mathcal O}=(\varphi, y_{+},\beta,T)$ is a Long Orbit.
Then, up to trimming ${\mathcal O}$, $\chi_{{\mathcal O}}$
satisfies
\[ \psi_{{\mathcal P}_{-}'}^{\varphi}(\chi_{{\mathcal O}}(s+iu)) - \psi_{{\mathcal P}_{+}'}^{\varphi}(0,y_{+}) =
\lim_{\vartheta_{\mathcal O}(x) \to u, \ x \to 0}^{\lceil T(x) \rceil - T(x) \to s}
\left(
\lceil T(x) \rceil + 2 \pi i \sum_{Q \in E_{-}(x)} Res(X,Q) \right) \]
for any $s +iu \in [0,1]+i{\mathcal S}_{\mathcal O}$
where $(E_{-},E_{+})$ is the division of $\mathrm{Fix} (\varphi)$ induced by ${\mathcal O}$.
\end{pro}
The concept of division of the fixed points induced by a Long Orbit is introduced during
the proof.
\begin{proof}
The family ${\{\Gamma_{x}[0,T(x)]\}}_{x \in \beta}$ associated to ${\mathcal O}$
is $(A_{\varphi},B_{\varphi})$ bounded by Proposition \ref{pro:uniform}.
The tracking phenomenon (Proposition \ref{pro:tracking}) implies that
$(X,y_{+},\beta,T)$ is a weak Long Trajectory.
Hence it induces a division $(E_{-},E_{+})$ of $\mathrm{Fix} (\varphi) = \mathrm{Sing} (X)$;
it is the division induced by ${\mathcal O}$.
We complete the proof by applying
Eq. (\ref{equ:ltlo1}) in Proposition \ref{pro:ltlo}.
\end{proof}
Long Orbits are invariant under translations in Fatou coordinates.
\begin{pro}
\label{pro:evol}
Let $\varphi \in \dif{p1}{2}$ with
$N > 1$. Fix a convergent normal form $X$ of $\varphi$.
Suppose that ${\mathcal O}=(\varphi, y_{+},\beta,T)$ is a Long Orbit.
Let ${\mathcal P}_{+}'$ and ${\mathcal P}_{-}'$ be the petals of $\varphi_{|x=0}$ containing
$y_{+}$ and $\chi_{\mathcal O}(0)$ respectively.
Consider   $y_{+}' \in  {\mathcal P}_{+}'$ such that
there exists $\chi':[0,1] + i {\mathcal S}_{\mathcal O} \to {\mathcal P}_{-}'$
satisfying
\[ \psi_{{\mathcal P}_{-}'}^{\varphi}(\chi'(z)) -\psi_{{\mathcal P}_{+}'}^{\varphi}(0,y_{+}') =
\psi_{{\mathcal P}_{-}'}^{\varphi}(\chi_{\mathcal O}(z))
- \psi_{{\mathcal P}_{+}'}^{\varphi}(0,y_{+}) \]
for any $z \in [0,1] + i {\mathcal S}_{\mathcal O}$. Then
${\mathcal O}'=(\varphi, y_{+}',\beta,T)$ generates a Long Orbit
such that
$\chi_{{\mathcal O}'} \equiv \chi'$.
\end{pro}
A consequence of the proposition is that
a topological conjugacy $\sigma$ between $\varphi, \eta \in \dif{p1}{2}$
with $N>1$ conjugates translations in
Fatou coordinates of $\varphi_{|x=0}$ and $\eta_{|x=0}$.
We will see that
$\sigma_{|x=0}$ is affine in Fatou coordinates
(Corollary \ref{cor:orp}).
\begin{proof}
Let
$(E_{-},E_{+})$ be the division of $\mathrm{Sing}(X)$ induced by ${\mathcal O}$.
The family associated to ${\mathcal O}$ is $(A_{\varphi},B_{\varphi})$ bounded
by Proposition \ref{pro:uniform}.
Consider $\Gamma_{x} =\Gamma(X,(x,y_{+}'),U_{\epsilon})$ for $x \in \beta$.
We deduce that the family
${\{\Gamma_{x}[0,T(x)]\}}_{x \in \beta}$ is also $(A,B)$ bounded.
It is a weak Long Trajectory by the tracking phenomenon.
The trajectories in the family ${\{\Gamma_{x}[0,T(x)]\}}_{x \in \beta}$
induce the division $(E_{-},E_{+})$. We can apply the residue formula.

Consider ${\mathcal O}'=(\varphi, y_{+}',\beta,T)$.
We have
\[ \psi_{{\mathcal P}_{-}'}^{\varphi}(\chi_{{\mathcal O}}(s+iu)) - \psi_{{\mathcal P}_{+}'}^{\varphi}(0,y_{+}) =
\lim_{\vartheta_{\mathcal O}(x) \to u, \ x \to 0}^{\lceil T(x) \rceil - T(x) \to s}
\left(
\lceil T(x) \rceil + 2 \pi i \sum_{Q \in E_{-}(x)} Res(X,Q) \right) \]
for any $s+iu \in [0,1] + i {\mathcal S}_{\mathcal O}$ (Proposition \ref{pro:lores}).
Equation (\ref{equ:ltlo1}) in Proposition \ref{pro:ltlo} and the definition of $\chi'$
imply
\[ \psi_{{\mathcal P}_{-}'}^{\varphi}(\chi_{{\mathcal O}'}(s+iu)) - \psi_{{\mathcal P}_{+}'}^{\varphi}(0,y_{+}') =
\lim_{\vartheta_{\mathcal O}(x) \to u, \ x \to 0}^{\lceil T(x) \rceil - T(x) \to s}
\left(
\lceil T(x) \rceil + 2 \pi i \sum_{Q \in E_{-}(x)} Res(X,Q) \right) \]
for any $s+iu \in [0,1] + i {\mathcal S}_{\mathcal O}$ and $\chi_{{\mathcal O}'} \equiv \chi'$.
\end{proof}
\section{Topological conjugacies in the generic case}
\label{sec:finite}
We present some consequences of the existence of Long Orbits and their
properties for elements of $\dif{p1}{2}$ with $N >1$.
We are interested in the study of the rigidity properties of topological
conjugacies at the unperturbed line $x=0$.
More precisely we want to describe the behavior of $\sigma_{|x=0}$
where $\sigma$ is a homeomorphism conjugating
$\varphi, \eta \in \dif{p1}{2}$ with $N>1$.
Let us remark that we always consider that the topological conjugation
$\sigma$ preserves the fibration $dx=0$. In other words $\sigma$ is of the form
$\sigma(x,y) = (\sigma_{0}(x), \sigma_{1}(x,y))$.
The dynamics of $\varphi$
can be very rich, containing for instance small divisors phenomena.
It is then natural to think that conjugating all the dynamics of
$\varphi$, $\eta$ for the
lines $x=cte$ should impose heavy restrictions on the conjugacy at
$x=0$. We prove that $\sigma_{|x=0}$ is affine in Fatou coordinates
and generically holomorphic or anti-holomorphic.
\subsection{Affine conjugacies}
In the next proposition we analyze the behavior of topological conjugacies
in a petal of the unperturbed line $x=0$.
\begin{pro}
\label{pro:defh}
Let $\varphi, \eta \in \dif{p1}{2}$ with $N>1$ such that
there exists a homeomorphism $\sigma$ satisfying
$\sigma \circ \varphi = \eta \circ \sigma$.
Consider a petal ${\mathcal P}'$
of $\varphi_{|x=0}$. Let
$\psi_{\sigma({\mathcal P}')}^{\eta}$ be a Fatou coordinate of $\eta$
in the petal $\sigma ({\mathcal P}')$.
Then the function
\[ {\mathfrak h}_{{\mathcal P}'}(y, z) = (\psi_{\sigma({\mathcal P}')}^{\eta} \circ \sigma \circ
\mathrm{exp}(z X_{{\mathcal P}'}^{\varphi}) -
\psi_{\sigma({\mathcal P}')}^{\eta} \circ \sigma)(0,y) \]
(see Definition \ref{def:infgen})
does not depend on $y \in {\mathcal P}'$.
\end{pro}
The mapping $\sigma$ conjugates the translations $\psi + z$ and
$\psi + {\mathfrak h}_{{\mathcal P}'}(z)$ in Fatou coordinates of
$\varphi_{|{\mathcal P}'}$ and $\eta_{|\sigma({\mathcal P}')}$ respectively.
Moreover the proof of the proposition provides an expression for
${\mathfrak h}_{{\mathcal P}'}(z)$.
\begin{proof}
The function ${\mathfrak h}_{{\mathcal P}'}$ satisfies
\[ {\mathfrak h}_{{\mathcal P}'}(\varphi(0,y), z) \equiv {\mathfrak h}_{{\mathcal P}_{-}'}(y, z) \ \
\mathrm{and} \ \
{\mathfrak h}_{{\mathcal P}'}(y, z+1) \equiv {\mathfrak h}_{{\mathcal P}'}(y, z) +1. \]
In particular it suffices to prove the proposition for $z \in [0,1) + i {\mathbb R}$
and $y^{1} \in {\mathcal P}'$ such that
$\mathrm{exp}(z X_{{\mathcal P}'}^{\varphi})(0,y^{1}) \in {\mathcal P}'$
for any $z \in [0,1] + i {\mathbb R}$.

It suffices to prove the result when
${\mathcal P}'={\mathcal P}_{-}'$ is a repelling petal.
Otherwise ${\mathcal P}'$ is a repelling petal of
$\varphi^{-1}$ and $\sigma$ conjugates
$\varphi^{-1}$ and $\eta^{-1}$.
Consider $\epsilon>0$ small enough. In particular
$\sigma(U_{\epsilon})$ is contained in $U_{\tilde{\epsilon}}$
for a small $\tilde{\epsilon}>0$.
Let $X$, $Y$ be convergent normal forms of $\varphi$ and $\eta$ respectively.

By Proposition \ref{pro:exll2} there exists an attracting petal
${\mathcal P}_{+}$ of $\Re(X)_{|U_{\epsilon}(0)}$ and a Long Trajectory
${\mathcal O}=(X,y_{+},\beta,T)$ such that ${\mathcal S}_{\mathcal O}= {\mathbb R}$ and
$\chi_{\mathcal O} (i {\mathbb R}) \subset {\mathcal P}_{-}$.
Consider $y_{+}' \in {\mathcal P}_{+}'$ such that
\[ \psi_{{\mathcal P}_{-}'}^{\varphi}(0,y^{1}) - \psi_{{\mathcal P}_{+}'}^{\varphi}(0,y_{+}')=
\psi_{-}(\chi_{\mathcal O} (0)) - \psi_{+}(0,y_{+}). \]
Such a choice is possible by replacing $T$ with $T-j$
and $(0,y_{+})$ with ${\mathfrak F}_{\varphi}^{j}(0,y_{+})$ for some
$j \in {\mathbb N}$ if necessary.
Consider the Long Orbit  ${\mathcal O}' = (\varphi, y_{+}, \beta, T)$.
We have
\[ \psi_{{\mathcal P}_{-}'}^{\varphi}(\chi_{{\mathcal O}'}(z)) - \psi_{{\mathcal P}_{+}'}^{\varphi}(0,y_{+})
\equiv \psi_{-}(\chi_{\mathcal O} (0)) - \psi_{+}(0,y_{+}) + z \]
by Proposition \ref{pro:ltlo} (see Eq. (\ref{equ:ltlo2})).
Proposition \ref{pro:evol} implies that ${\mathcal O}'' = (\varphi, y_{+}', \beta, T)$
is a Long Orbit such that
$\chi_{{\mathcal O}''}(z) = \mathrm{exp}(z X_{{\mathcal P}'}^{\varphi})(0,y^{1})$
for any $z \in [0,1] + i {\mathbb R}$.

We define $(0,\tilde{y}_{+}) = \sigma(0,y_{+}')$ and
$\tilde{\mathcal O} = (\eta, \tilde{y}_{+}, \sigma(\beta), T \circ \sigma^{-1})$. It is clear
that $\tilde{\mathcal O}$ is a Long Orbit
such that $\chi_{\tilde{\mathcal O}} \equiv \sigma \circ \chi_{{\mathcal O}''}$.
Up to trimming
$(Y, \tilde{y}_{+}, \sigma(\beta), T \circ \sigma^{-1})$ generates a weak Long Trajectory.
Hence $(Y, \tilde{y}_{+}, \sigma(\beta), T \circ \sigma^{-1})$
induces a division $(\tilde{E}_{-},\tilde{E}_{+})$ of $\mathrm{Sing} (Y)$.

Fix $z = s+iu \in [0,1) + i {\mathbb R}$. We consider the sequence ${\{x_{n}^{z}\}}$
contained in $\vartheta_{\mathcal O}^{-1}(u)$ such that
$T(x_{n}^{z})=n-s$. It is defined for $n>>0$.
We have
\begin{equation}
\label{equ:con1}
\lim_{n \to \infty}  \left(
\psi_{{\mathcal P}_{-}'}^{\varphi}(\chi_{{\mathcal O}''}(z)) - \psi_{{\mathcal P}_{+}'}^{\varphi}(0,y_{+}') - n
- 2 \pi i \sum_{Q \in E_{-}(x_{n}^{z})} Res(X,Q) \right) =0
\end{equation}
and
\begin{equation}
\label{equ:con2}
\lim_{n \to \infty}  (
\psi_{\sigma({\mathcal P}_{-}')}^{\eta}( \sigma(\chi_{{\mathcal O}''}(z)))
- \psi_{\sigma({\mathcal P}_{+}')}^{\eta} (0,\tilde{y}_{+}) - n
- 2 \pi i \sum_{Q \in \tilde{E}_{-}(\sigma(x_{n}^{z}))} Res(Y,Q) ) =0
\end{equation}
by Proposition \ref{pro:lores}.
Denote
$L=\psi_{{\mathcal P}_{-}'}^{\varphi}(0,y^{1}) + \psi_{\sigma({\mathcal P}_{+}')}^{\eta}(0,\tilde{y}_{+})
-\psi_{{\mathcal P}_{+}'}^{\varphi}(0,y_{+}')$
and
\[ G_{n}^{z} = 2 \pi i \left( \sum_{Q \in \tilde{E}_{-}(\sigma(x_{n}^{z}))} Res(Y,Q) -
\sum_{Q \in E_{-}(x_{n}^{z})} Res(X,Q) \right) \]
By subtracting Eqs. (\ref{equ:con2}) and (\ref{equ:con1}) we get
\begin{equation}
\label{equ:inv}
\psi_{\sigma({\mathcal P}_{-}')}^{\eta}(\sigma(\mathrm{exp}(z X_{{\mathcal P}_{-}'}^{\varphi})(0,y^{1}))) = z+ L +
\lim_{n \to \infty} G_{n}^{z}
\end{equation}
and then
\[ \psi_{\sigma({\mathcal P}_{-}')}^{\eta}(\sigma(\mathrm{exp}(z X_{{\mathcal P}_{-}'}^{\varphi})(0,y^{1}))) -
\psi_{\sigma({\mathcal P}_{-}')}^{\eta}(\sigma(0,y^{1})) = z+
\lim_{n \to \infty} G_{n}^{z} - \lim_{n \to \infty} G_{n}^{0} . \]
Let us remark that we replace $T$ with $T-j$ for some
$j \in {\mathbb N}$ during the proof. The sequence
$\textmd{x}_{n}^{z}$ that we associate to $T-j$
satisfies $\textmd{x}_{n-j}^{z}=x_{n}^{z}$.
Thus the right hand side of Eq. (\ref{equ:inv})
does not change through the process.
Clearly it
does not depend on $y^{1}$.
\end{proof}
\begin{pro}
\label{pro:rlinear}
Consider the setting of Proposition \ref{pro:defh}.
Then ${\mathfrak h}_{{\mathcal P}_{-}}:{\mathbb C} \to {\mathbb C}$
is a ${\mathbb R}$-linear isomorphism such that
${\mathfrak h}_{{\mathcal P}'}(1)=1$.
\end{pro}
\begin{proof}
We have
\[ {\mathfrak h}_{{\mathcal P}'}(z_{1}+z_{2}) =
 {\mathfrak h}_{{\mathcal P}'}(\mathrm{exp}(z_{2} X_{{\mathcal P}'}^{\varphi})(0,y), z_{1})
 + {\mathfrak h}_{{\mathcal P}'}(y, z_{2})=
 {\mathfrak h}_{{\mathcal P}'}(z_{1}) + {\mathfrak h}_{{\mathcal P}'}(z_{2})\]
for any $z_{1},z_{2} \in {\mathbb C}$.
Since ${\mathfrak h}_{{\mathcal P}'}$ is continuous by definition then
${\mathfrak h}_{{\mathcal P}'}$ is ${\mathbb R}$-linear.
The function ${\mathfrak h}_{{\mathcal P}'}$ is injective in a neighborhood of $0$,
hence ${\mathfrak h}_{{\mathcal P}'}$ is an isomorphism.
The condition ${\mathfrak h}_{{\mathcal P}'}(1)=1$ is obvious.
\end{proof}
\begin{rem}
Consider the setting of Proposition \ref{pro:defh}.
Proposition \ref{pro:rlinear} implies that
${\mathfrak h}_{{\mathcal P}'}(s) =s$ for all
petal ${\mathcal P}'$ of $\varphi_{|x=0}$ and $s \in {\mathbb R}$.
As a consequence $\sigma$ conjugates the real flows of
$X_{{\mathcal P}'}^{\varphi}$ and $X_{\sigma({\mathcal P}')}^{\eta}$.
\end{rem}
So far we proved that
$\psi_{\sigma({\mathcal P}')}^{\eta} \circ \sigma \circ (\psi_{{\mathcal P}'}^{\varphi})^{-1}$
is affine for any petal ${\mathcal P}'$ of $\varphi_{|x=0}$.
Next we show that these mappings do not depend on the petal.
\begin{lem}
\label{lem:eqcon}
Let $\varphi, \eta \in \dif{p1}{2}$ with $N>1$ such that
there exists a homeomorphism $\sigma$ satisfying
$\sigma \circ \varphi = \eta \circ \sigma$.
Then ${\mathfrak h}_{{\mathcal P}'} \equiv {\mathfrak h}_{{\mathcal Q}'}$ for all
petals ${\mathcal P}'$, ${\mathcal Q}'$ of $\varphi_{|x=0}$.
\end{lem}
\begin{proof}
It suffices to prove the lemma for consecutive petals
${\mathcal P}_{+}'$ and ${\mathcal P}_{-}'$ of $\varphi_{|x=0}$.
Let $X$, $Y$ be convergent normal forms of $\varphi$ and $\eta$
respectively.
Consider Fatou coordinates $\psi$, $\tilde{\psi}$ of $X$, $Y$ defined
in the neighborhood of
${\mathcal P}_{+}' \cup {\mathcal P}_{-}'$ and
$\sigma({\mathcal P}_{+}' \cup {\mathcal P}_{-}')$ respectively. We have
\[ \lim_{ |Im(\psi(0,y))| \to \infty}
\sum_{j \in {\mathbb Z}} |\Delta_{\varphi}(\varphi^{j}(0,y))| =0, \ \
\lim_{|Im(\tilde{\psi}(0,y))| \to \infty}
\sum_{j \in {\mathbb Z}} |\Delta_{\eta}(\eta^{j}(0,y))|=0,
\]
see Definition \ref{def:delta}.
The first limit is calculated for
$y \in {\mathcal P}_{+}' \cup {\mathcal P}_{-}'$ whereas the second limit
is calculated for
$y \in \sigma({\mathcal P}_{+}' \cup {\mathcal P}_{-}')$.
The condition $|Im(\psi(0,y))| \to \infty$ is equivalent to
the orbit ${\{ \varphi^{j}(0,y)\}}_{j \in {\mathbb Z}}$ tending
uniformly to the origin. The proof of the previous equations are
analogous to the proof of Proposition \ref{pro:boufespre}. We obtain
\begin{equation}
\label{equ:auxpr}
\lim_{|Im(\psi(0,y))| \to \infty}
(\psi_{{\mathcal P}_{k}'}^{\varphi} - \psi)(0,y) =0,
\lim_{|Im(\tilde{\psi}(0,y))| \to \infty}
(\psi_{\sigma({\mathcal P}_{k}')}^{\eta} - \tilde{\psi})(0,y) =0
\end{equation}
for $k \in \{+,-\}$. We also get
\begin{equation}
\label{equ:auxpr2}
\lim_{|Im(\psi(0,y))| \to \infty}
(\psi_{{\mathcal P}_{+}'}^{\varphi} - \psi_{{\mathcal P}_{-}'}^{\varphi})(0,y) =
\lim_{|Im(\tilde{\psi}(0,y))| \to \infty}
(\psi_{\sigma({\mathcal P}_{+}')}^{\eta} - \psi_{\sigma({\mathcal P}_{-}')}^{\eta})(0,y) =0.
\end{equation}

Consider a sequence $y_{n}$ in ${\mathcal P}_{+}' \cap {\mathcal P}_{-}'$
such that $|Im(\psi(0,y_{n}))| \to \infty$. Fix $z \in {\mathbb C}$.
We have $|Im(\psi(\mathrm{exp}(z X_{{\mathcal P}'}^{\varphi})(0,y_{n})))| \to \infty$
by Eq. (\ref{equ:auxpr}). Let $z_{n}$ be the complex number such that
$\mathrm{exp}(z X_{{\mathcal P}_{-}'}^{\varphi})(0,y_{n})=
\mathrm{exp}(z_{n} X_{{\mathcal P}_{+}'}^{\varphi})(0,y_{n})$.
The sequence $z_{n}$ satisfies $\lim_{n \to \infty} z_{n}=z$
by Eq. (\ref{equ:auxpr2}). We have
\[ {\mathfrak h}_{{\mathcal P}_{-}'}(z) =
(\psi_{\sigma({\mathcal P}_{-}')}^{\eta} \circ \sigma \circ
\mathrm{exp}(z X_{{\mathcal P}_{-}'}^{\varphi}) -
\psi_{\sigma({\mathcal P}_{-}')}^{\eta} \circ \sigma)(0,y_{n}) \]
and
\[ {\mathfrak h}_{{\mathcal P}_{+}'}(z_{n}) =
(\psi_{\sigma({\mathcal P}_{+}')}^{\eta} \circ \sigma \circ
\mathrm{exp}(z_{n} X_{{\mathcal P}_{+}'}^{\varphi}) -
\psi_{\sigma({\mathcal P}_{+}')}^{\eta} \circ \sigma)(0,y_{n}). \]
Equation (\ref{equ:auxpr2}) implies
$\lim_{n \to \infty}({\mathfrak h}_{{\mathcal P}_{-}'}(z)-{\mathfrak h}_{{\mathcal P}_{+}'}(z_{n})) =0$.
Since $\lim_{n \to \infty} z_{n}=z$ and ${\mathfrak h}_{{\mathcal P}_{+}'}$ is continuous
we obtain ${\mathfrak h}_{{\mathcal P}_{-}'}(z)={\mathfrak h}_{{\mathcal P}_{+}'}(z)$ for any $z \in {\mathbb C}$.
\end{proof}
\begin{defi}
Let $\varphi, \eta \in \dif{p1}{2}$ with $N>1$ such that
there exists a homeomorphism $\sigma$ satisfying
$\sigma \circ \varphi = \eta \circ \sigma$. We denote
by ${\mathfrak h}_{\varphi,\eta,\sigma}$ any of the functions
${\mathfrak h}_{{\mathcal P}'}$ defined in
Proposition \ref{pro:defh}.
We denote ${\mathfrak h}={\mathfrak h}_{\varphi,\eta,\sigma}$ if the data are implicit.
\end{defi}
\begin{defi}
Let $\varphi, \eta \in \diff{p1}{2}$ such that
a homeomorphism $\sigma$ conjugates $\varphi$ and $\eta$.
The mapping $\sigma$ is of the form
$\sigma(x,y) = (\sigma_{0}(x), \sigma_{1}(x,y))$.
We say that the action of $\sigma$ on the parameter
space is holomorphic (resp. anti-holomorphic, orientation-preserving)
if $\sigma_{0}$ is
holomorphic (resp. anti-holomorphic, orientation-preserving).
\end{defi}
The orientation properties of the restriction of the conjugation
to $x=0$ and of its action on the parameter space are the same.
\begin{lem}
\label{lem:orip}
Let $\varphi, \eta \in \dif{p1}{2}$ with $N>1$ such that
there exists a homeomorphism $\sigma$ satisfying
$\sigma \circ \varphi = \eta \circ \sigma$.
Then the isomorphism
${\mathfrak h}_{\varphi,\eta,\sigma}$ is orientation-preserving if and only if
the action of $\sigma$ on the parameter space is
orientation-preserving.
\end{lem}
\begin{proof}
Suppose that $\sigma$ does not preserve orientation in the parameter space.
We denote
\[ \zeta(x,y) = (\overline{x},\overline{y}), \
\tilde{\eta} = \zeta \circ \eta \circ \zeta \ \mathrm{and} \
\tilde{\sigma} = \zeta \circ \sigma \]
where $\overline{x}$ is the complex conjugation.
The diffeomorphism $\tilde{\eta}$ belongs to $\diff{p1}{2}$
and $\tilde{\sigma} \circ \varphi = \tilde{\eta} \circ \tilde{\sigma}$.
The action of $\tilde{\sigma}$ on the parameter space is
orientation-preserving. We have
\[ {\mathfrak h}_{\varphi, \tilde{\eta}, \tilde{\sigma}}(z) = \overline{{\mathfrak h}_{\varphi, \eta, \sigma}(z)} \]
for any $z \in {\mathbb C}$.
Therefore it suffices to prove that if the action of $\sigma$ in the
parameter space is orientation-preserving so is ${\mathfrak h}={\mathfrak h}_{\varphi,\eta,\sigma}$.

Fix a repelling petal ${\mathcal P}'={\mathcal P}_{-}'$.
Consider the notations in Proposition \ref{pro:defh}. We define
\[ G_{0}(x) = - 2 \pi i \sum_{Q \in {E}_{-}(x)} Res(X,Q) \ \mathrm{and} \
G(x) = - 2 \pi i \sum_{Q \in \tilde{E}_{-}(x)} Res(Y,Q) . \]
We have
\[ \psi_{\sigma({\mathcal P}_{-}')}^{\eta}(\chi_{\tilde{\mathcal O}}(z)) -
\psi_{\sigma({\mathcal P}_{+}')}^{\eta}(0,\tilde{y}_{+}) =
\lim_{(\lceil T \rceil - T)(x) \to s, \ x \in \vartheta_{\tilde{\mathcal O}}^{-1}(u), \ x \to 0} (
\lceil T(x) \rceil - G(x) ) \]
for any $z = s+iu \in [0,1] +i {\mathbb R}$
by Proposition \ref{pro:lores}.
Since ${\mathfrak h}_{|{\mathbb R}} \equiv Id$,
$\chi_{\tilde{\mathcal O}} \equiv \sigma \circ \chi_{{\mathcal O}''}$ and
$\psi_{{\mathcal P}_{-}'}^{\varphi}(\chi_{{\mathcal O}''}(z)) -
\psi_{{\mathcal P}_{-}'}^{\varphi}(\chi_{{\mathcal O}''}(0)) \equiv z$
we deduce that
\begin{equation}
\label{equ:auxor}
\lim_{x \in \vartheta_{\tilde{\mathcal O}}^{-1}(u), \ x \to 0} Im (G(x)) = c_{u} \ \mathrm{and} \
\lim_{x \in \vartheta_{\tilde{\mathcal O}}^{-1}(u), \ x \to 0} Re (G(x)) = \infty.
\end{equation}
where $c_{u}=
- Im(\psi_{\sigma({\mathcal P}_{-}')}^{\eta}( \sigma(\chi_{{\mathcal O}''}(i u))) -
\psi_{\sigma({\mathcal P}_{+}')}^{\eta}(0,\tilde{y}_{+}))$.
The function $G$ is meromorphic (Proposition 5.2 of \cite{UPD})
and has a pole of order greater than $0$.
Every curve $\vartheta_{\tilde{\mathcal O}}^{-1}(u)$
adheres to the same direction in the parameter space (Proposition \ref{pro:udlo}).
We obtain
\[ \lim_{x \in \vartheta_{\tilde{\mathcal O}}^{-1}(u), \ x \to 0} Im (G(x)) -
\lim_{x \in \vartheta_{\tilde{\mathcal O}}^{-1}(0), \ x \to 0} Im (G(x)) =
- Im ({\mathfrak h}(iu))  \]
for any $u \in {\mathbb R}$.

Consider the connected curve $\iota_{u}$ such that $\iota_{u}$ adheres to the
same direction as $\vartheta_{\tilde{\mathcal O}}^{-1}(0)$ and
\[ \iota_{u} \subset \{ x \in B(0,\delta) \setminus \{0\} :
Im (G(x)) = c_{u}. \}
\]
We have $\lim_{x \in \iota_{u}, \ x \to 0} Re (G(x)) = \infty$
for any $u \in {\mathbb R}$. The curves (see Eq. (\ref{equ:deftf}))
\[ \varsigma_{u} = \{ x \in {\mathbb C}: \psi_{-}(0,y_{-}) - \psi_{+}(0,y_{+})
+  x +iu \in {\mathbb R}^{+} \} \]
move in clock wise sense when we increase $u$.
The function $G_{0}$ is meromorphic,
it satisfies $G_{0}^{-1}(\varsigma_{u}) = \vartheta_{{\mathcal O}''}^{-1}(u)$
for $u \in {\mathbb R}$
and has a pole of order greater than $0$.
Thus the curves
$\vartheta_{{\mathcal O}''}^{-1}(u)$ move in
counter clock wise sense when $u$ increases.
Since the action of $\sigma$ on the parameter space
is orientation-preserving the same property holds true for
$\vartheta_{\tilde{\mathcal O}}^{-1}(u)$.
Equation (\ref{equ:auxor})
implies that the curves $\iota_{u}$ move in
counter clockwise sense when $u$ increases.
The function $G$ is meromorphic
and has a pole of order greater than $0$.
Thus the curves
$\{x \in {\mathbb C}: Re(x) \in {\mathbb R}^{+} \ \mathrm{and} \ Im(x)=c_{u} \}$
move in clock wise sense when we increase $u$.
Hence $c_{u}$ is a decreasing function of $u$. We obtain
\[ Im ({\mathfrak h}(i)) = c_{0} - c_{1} >0 . \]
As a consequence ${\mathfrak h}$ is orientation-preserving.
\end{proof}
\begin{cor}
\label{cor:holact}
Let $\varphi, \eta \in \dif{p1}{2}$ with $N>1$ such that
there exists a homeomorphism $\sigma$ satisfying
$\sigma \circ \varphi = \eta \circ \sigma$.
Suppose that the action of $\sigma$ on the parameter space is
holomorphic (resp. anti-holomorphic). Then
$\sigma_{|x=0}$ is holomorphic (resp.  anti-holomorphic).
\end{cor}
The corollary implies Proposition \ref{pro:holpar} in the case
$m(\varphi) = 0$.
The case $m(\varphi) > 0$ is treated in Corollary \ref{cor:holpar}.
\begin{proof}
Suppose that the action of $\sigma$ on the parameter space is holomorphic.
Consider the notations in Proposition \ref{pro:defh}.
Let ${\mathcal P}'$ be a petal of $\varphi_{|x=0}$.
The functions
\[ \sum_{Q \in E_{-}(x)} Res(X,Q) \ \ \mathrm{and} \
\sum_{Q \in \tilde{E}_{-}(\sigma(x))} Res(Y,Q) \]
are meromorphic (Proposition 5.2 of \cite{UPD}).
The function
\[ 2 \pi i \left(
\sum_{Q \in \tilde{E}_{-}(\sigma(x))} Res(Y,Q) -  \sum_{Q \in E_{-}(x)} Res(X,Q)
\right) \]
is meromorphic. Hence all the limits of the previous function in sequences
tending to $0$ are equal. We deduce that
$\lim_{n \to \infty} G_{n}^{z} = \lim_{n \to \infty} G_{n}^{0}$
for any $z \in {\mathbb C}$. We obtain
\[ \psi_{\sigma({\mathcal P}')}^{\eta}(\sigma(\mathrm{exp}(z X_{{\mathcal P}'}^{\varphi})(0,y))) -
\psi_{\sigma({\mathcal P}')}^{\eta}(\sigma(0,y)) = z   \]
for $y \in {\mathcal P}'$.
The mappings $\psi_{\sigma({\mathcal P}')}^{\eta}$ and $z \to \mathrm{exp}(z X_{{\mathcal P}'}^{\varphi})(0,y)$
are biholomorphic. As a consequence $\sigma_{|x=0}$ is holomorphic in
${\mathcal P}'$. Since the union of the petals is a pointed neighborhood of
the origin we obtain that $\sigma_{|x=0}$ is holomorphic by
Riemann's removable singularity theorem.

Suppose that the action of $\sigma$ on the parameter space is anti-holomorphic.
Denote $\zeta(x,y) = (\overline{x},\overline{y})$,
$\tilde{\eta} = \zeta \circ \eta \circ \zeta$ and
$\tilde{\sigma} = \zeta \circ \sigma$. We have
$\tilde{\sigma} \circ \varphi = \tilde{\eta} \circ \tilde{\sigma}$.
The action of $\tilde{\sigma}$ on the parameter space is holomorphic.
Therefore $\tilde{\sigma}_{|x=0}$ is holomorphic and
$\sigma_{|x=0}$ is anti-holomorphic.
\end{proof}
The next result is the General Theorem for the case $m(\varphi)=0$.
\begin{cor}
\label{cor:orp}
Let $\varphi, \eta \in \dif{p1}{2}$ with $N>1$ such that
there exists a homeomorphism $\sigma$ satisfying
$\sigma \circ \varphi = \eta \circ \sigma$.
Then $\sigma_{|x=0}$ is affine in Fatou coordinates (see Definition \ref{def:afffc}).
Moreover $\sigma_{|x=0}$ is orientation-preserving if and only if
the action of $\sigma$ on the parameter space is
orientation-preserving.
\end{cor}
Let us remark that if $\sigma_{|x=0}$ is affine in Fatou coordinates then it is
real analytic in $\{x=0\} \setminus \{(0,0)\}$.
The corollary is a consequence of Propositions \ref{pro:defh},
\ref{pro:rlinear} and  Lemmas \ref{lem:eqcon}, \ref{lem:orip}.
\begin{defi}
Let $\phi \in \diff{1}{}$.
We say that $\phi$ is analytically trivial if
there exists a local holomorphic singular vector field $Z=a(z) \partial / \partial z$
such that $\phi = \mathrm{exp}(Z)$.
This condition is equivalent to
$X_{{\mathcal P}'}^{\phi} \equiv X_{{\mathcal Q}'}^{\phi}$ for all petals
${\mathcal P}'$, ${\mathcal Q}'$ of $\phi$ such that
${\mathcal P}' \cap {\mathcal Q}' \neq \emptyset$. It is also equivalent to
$\psi_{{\mathcal P}'}^{\phi} - \psi_{{\mathcal Q}'}^{\phi}$ being constant
for all petals
${\mathcal P}'$, ${\mathcal Q}'$ of $\phi$ such that
${\mathcal P}' \cap {\mathcal Q}' \neq \emptyset$
($\psi_{{\mathcal P}'}^{\phi}$ and $\psi_{{\mathcal Q}'}^{\phi}$ are Fatou
coordinates of $\phi$ in ${\mathcal P}'$ and ${\mathcal Q}'$ respectively).
\end{defi}
\begin{rem}
The condition of being non-analytically trivial is generic among the
tangent to the identity local diffeomorphisms in one variable.
More precisely, every formal class of conjugacy (i.e. a class of equivalence
for the relation given by the formal conjugation) contains a continuous moduli
of analytic classes of conjugacy and a unique analytically trivial class.
These properties are a consequence of the analytic classification of tangent to
the identity diffeomorphisms (see \cite{Loray5}).
\end{rem}
It is possible to construct affine conjugacies in Fatou coordinates
that are not holomorphic or anti-holomorphic by restriction to
$x=0$ if both $\varphi$ and $\eta$ are embedded in analytic flows
(see Section \ref{sec:build}). They are essentially the only examples.
\begin{pro}
\label{pro:natc}
Let $\varphi, \eta \in \dif{p1}{2}$ with $N>1$ such that
there exists a homeomorphism $\sigma$ satisfying
$\sigma \circ \varphi = \eta \circ \sigma$.
Suppose that either $\varphi_{|x=0}$ or $\eta_{|x=0}$
is non-analytically trivial.
Then either ${\mathfrak h}_{\varphi,\eta,\sigma} \equiv z$ and $\sigma_{|x=0}$ is
holomorphic or ${\mathfrak h}_{\varphi,\eta,\sigma} \equiv \overline{z}$ and $\sigma_{|x=0}$ is
anti-holomorphic.
\end{pro}
The proposition implies the Main Theorem.
\begin{proof}
The isomorphism ${\mathfrak h}_{\eta,\varphi,\sigma^{-1}}$
is the inverse of ${\mathfrak h}_{\varphi,\eta,\sigma}$.
Thus we can suppose that $\varphi_{|x=0}$ is non-analytically trivial.

Let ${\mathcal P}'$ be a petal of $\varphi_{|x=0}$. Fix $(0,y_{+}) \in {\mathcal P}'$.
We have
\[ \psi_{\sigma({\mathcal P}')}^{\eta} (\sigma(0,y)) =
 \psi_{\sigma({\mathcal P}')}^{\eta} (\sigma(0,y_{+})) +
{\mathfrak h}( \psi_{{\mathcal P}'}^{\varphi}(0,y) - \psi_{{\mathcal P}'}^{\varphi}(0,y_{+})) \]
and then
\[ (\psi_{\sigma({\mathcal P}')}^{\eta} \circ \sigma)(0,y) =
((z+ c_{{\mathcal P}'}) \circ {\mathfrak h} \circ  \psi_{{\mathcal P}'}^{\varphi})(0,y)  \]
for some $c_{{\mathcal P}'} \in {\mathbb C}$ and any $y \in {\mathcal P}'$.
Consider two consecutive petals ${\mathcal P}'$ and ${\mathcal Q}'$
of $\varphi_{|x=0}$. We consider the changes of charts
\[ \psi_{\sigma({\mathcal Q}')}^{\eta} \circ (\psi_{\sigma({\mathcal P}')}^{\eta})^{-1}=
(\psi_{\sigma({\mathcal Q}')}^{\eta} \circ \sigma) \circ
(\psi_{\sigma({\mathcal P}')}^{\eta} \circ \sigma)^{-1} = \]
\[ (z+ c_{{\mathcal Q}'}) \circ {\mathfrak h} \circ  \psi_{{\mathcal Q}'}^{\varphi} \circ
(\psi_{{\mathcal P}'}^{\varphi})^{-1} \circ {\mathfrak h}^{-1} \circ (z- c_{{\mathcal P}'}). \]
The left hand side is holomorphic, thus
${\mathfrak h} \circ  \psi_{{\mathcal Q}'}^{\varphi} \circ
(\psi_{{\mathcal P}'}^{\varphi})^{-1} \circ {\mathfrak h}^{-1}$
is also holomorphic.
We denote $H = \psi_{{\mathcal Q}'}^{\varphi} \circ
(\psi_{{\mathcal P}'}^{\varphi})^{-1}$, it is holomorphic.
The isomorphisms ${\mathfrak h}$ and ${\mathfrak h}^{-1}$ are of the form
${\mathfrak h}(z) = \varsigma_{0} z + \varsigma_{1} \overline{z}$ and
${\mathfrak h}^{-1}(z) = \varrho_{0} z + \varrho_{1} \overline{z}$
where $\varrho_{0} = \overline{\varsigma_{0}}/(|\varsigma_{0}|^{2}- |\varsigma_{1}|^{2})$
and $\varrho_{1} = -\varsigma_{1}/(|\varsigma_{0}|^{2}- |\varsigma_{1}|^{2})$.
We have
\[ \frac{\partial ({\mathfrak h} \circ H \circ {\mathfrak h}^{-1})}{\partial \overline{z}} =
\frac{\partial {\mathfrak h}}{\partial z} \frac{\partial (H \circ {\mathfrak h}^{-1})}{\partial \overline{z}} +
\frac{\partial {\mathfrak h}}{\partial \overline{z}}
\frac{\partial (\overline{H \circ {\mathfrak h}^{-1}})}{\partial \overline{z}}=
\varsigma_{0} \frac{\partial H}{\partial z} \varrho_{1} +
\varsigma_{1} \overline{ \frac{\partial H}{\partial z}} \overline{\varrho_{0}} =\]
\[ \frac{\varsigma_{0} \varsigma_{1}}{|\varsigma_{0}|^{2}- |\varsigma_{1}|^{2}}
\left( \overline{ \frac{\partial H}{\partial z}} - \frac{\partial H}{\partial z} \right). \]
Suppose $\varsigma_{0}=0$. Since ${\mathfrak h}(1)=1$ we deduce ${\mathfrak h} \equiv \overline{z}$.
Hence $\sigma_{|x=0}$ is anti-holomorphic.
Suppose $\varsigma_{1}=0$. Since ${\mathfrak h}(1)=1$ we deduce ${\mathfrak h} \equiv z$.
Hence $\sigma_{|x=0}$ is holomorphic.
Since ${\mathfrak h} \circ H \circ {\mathfrak h}^{-1}$ is holomorphic we can suppose that
$\overline{\partial H/\partial z} \equiv \partial H/\partial z$.
The function $\partial H / \partial z$ is real and then constant
by the open mapping theorem. Therefore
$H$ is of the form $az + b$ for some constants $a,b \in {\mathbb C}$.
The constant $a$ is equal to $1$ since $H(z+1) \equiv H(z)+1$.
We obtain $\psi_{{\mathcal Q}'}^{\varphi} \equiv \psi_{{\mathcal P}'}^{\varphi} + b$.
We deduce that $X_{{\mathcal Q}'}^{\varphi} \equiv X_{{\mathcal P}'}^{\varphi}$.
The last property does not hold true for every pair of consecutive petals
of $\varphi_{|x=0}$ by hypothesis. Thus $\sigma_{|x=0}$ is either
holomorphic or anti-holomorphic.
\end{proof}
\begin{cor}
Let $\varphi, \eta \in \dif{p1}{2}$ with $N>1$ such that
there exists a homeomorphism $\sigma$ satisfying
$\sigma \circ \varphi = \eta \circ \sigma$.
Then $\varphi_{|x=0}$ is analytically trivial if and only if
$\eta_{|x=0}$ is analytically trivial.
\end{cor}
Proposition \ref{pro:natc} describes the invariance of
the analytic classes of the unperturbed
diffeomorphisms by
topological conjugation. Next lemma describes the action on formal invariants.
\begin{lem}
\label{lem:res}
Let $\varphi, \eta \in \dif{p1}{2}$ with $N>1$ such that
there exists a homeomorphism $\sigma$ satisfying
$\sigma \circ \varphi = \eta \circ \sigma$. Then we have
\[ {\mathfrak h}_{\varphi,\eta,\sigma}(2 \pi i Res_{\varphi}(0,0)) =  2 \pi i Res_{\eta}(0,0) \ \mathrm{or} \
{\mathfrak h}_{\varphi,\eta,\sigma}(2 \pi i Res_{\varphi}(0,0)) =  -2 \pi i Res_{\eta}(0,0) \]
(see Definition \ref{def:res12})
depending on whether or not the action of $\sigma$ on the parameter space is
orientation-preserving.
In particular
$Re (Res_{\varphi}(0,0))$ and  $Re (Res_{\eta}(0,0))$
have the same sign.
\end{lem}
\begin{proof}
The sign can be positive, negative or $0$.
We denote $R_{\varphi}=Res_{\varphi}(0,0)$ and $R_{\eta}=Res_{\eta}(0,0)$.
Suppose that either $\varphi_{|x=0}$ or $\eta_{|x=0}$ is non-analytically trivial.
If $\sigma$ is orientation-preserving on the parameter space then
${\mathfrak h} \equiv z$ and $\sigma_{|x=0}$ is holomorphic by Corollary
\ref{cor:orp} and Proposition \ref{pro:natc}.
The result is a consequence of the residues being analytic invariants.
If $\sigma$ is orientation-reversing on the parameter space then
$\sigma_{|x=0}$ is anti-holomorphic and ${\mathfrak h} \equiv \overline{z}$ by Corollary
\ref{cor:orp} and Proposition \ref{pro:natc}.
The equation $R_{\eta} = \overline{R_{\varphi}}$
implies ${\mathfrak h}(2 \pi i R_{\varphi}) =  -2 \pi i R_{\eta}$.

Suppose that both $\varphi_{|x=0}$ and $\eta_{|x=0}$ are analytically trivial.
Let $X^{\varphi}$ and $X^{\eta}$ the local vector fields defined in $x=0$
such that $\varphi_{|x=0} =\mathrm{exp}  (X^{\varphi})$ and
$\eta_{|x=0} =\mathrm{exp}  (X^{\eta})$.
Consider Fatou coordinates $\psi_{X}^{\varphi}$ and $\psi_{X}^{\eta}$
of $X^{\varphi}$ and $X^{\eta}$ respectively.
The complex number $2 \pi i R_{\varphi}$ is the
additive monodromy of $\psi_{X}^{\varphi}$ along a path
turning once around the origin in counter clock wise sense.
Since we have
\[ (\psi_{X}^{\eta} \circ \sigma)(0,y) \equiv ({\mathfrak h} \circ \psi_{X}^{\varphi})(0,y) + c \]
for some $c \in {\mathbb C}$ then
${\mathfrak h}(2 \pi i R_{\varphi}) = \pm 2 \pi i R_{\eta}$
depending on whether ${\mathfrak h}$ is orientation-preserving
or orientation-reversing.

Suppose that ${\mathfrak h}$ is orientation-preserving.
We obtain
\[ \mathrm{sign}(Re(R_{\eta})) =\mathrm{sign}(Im({\mathfrak h}(2 \pi i R_{\varphi})))=
\mathrm{sign}(Im(2 \pi i R_{\varphi})) = \mathrm{sign}(Re(R_{\varphi})) . \]
We have
\[ \mathrm{sign}(Re(R_{\eta})) =-\mathrm{sign}(Im({\mathfrak h}(2 \pi i R_{\varphi})))=
\mathrm{sign}(Im(2 \pi i R_{\varphi})) = \mathrm{sign}(Re(R_{\varphi}))  \]
when ${\mathfrak h}$ is orientation-reversing.
\end{proof}
\subsection{Analytically trivial case}
In this section we study the properties of affine conjugacies (in Fatou coordinates)
in the non-analytically trivial case. Examples of the specific kind of behavior
described in next propositions are presented in Section \ref{sec:build}.

Consider $\varphi, \eta \in \dif{p1}{2}$ with $N>1$ and a topological conjugation $\sigma$.
Suppose that $\varphi_{|x=0}$ is analytically trivial.
There are two fundamentally different cases.
On the one hand if $Res_{\varphi}(0,0) \in i {\mathbb R}$ there are plenty of (Fatou) affine mappings
conjugating $\varphi_{|x=0}$ and $\eta_{|x=0}$ but $\varphi_{|x=0}$ and $\eta_{|x=0}$ are
always analytically or anti-analytically conjugated.
On the other hand if $Res_{\varphi}(0,0) \not \in i {\mathbb R}$ there
is rigidity of conjugations. In fact there are at most two conjugations
(up to precomposition with elements of the center of $\varphi_{|x=0}$ in $\diff{}{}$)
and in general none of them is holomorphic or anti-holomorphic.
\begin{proof}[Proof of Proposition \ref{pro:atui}]
Both residues $Res_{\varphi}(0,0)$ and $Res_{\eta}(0,0)$ belong to
$i {\mathbb R}$ by Lemma \ref{lem:res}.
Suppose that  $\sigma$ is orientation-preserving on the parameter space.
The map ${\mathfrak h}$ satisfies
${\mathfrak h}_{|{\mathbb R}} \equiv Id$ (Proposition \ref{pro:rlinear}).
Then we have
\[ 2 \pi i Res_{\varphi}(0,0) = {\mathfrak h}(2 \pi i Res_{\varphi}(0,0)) = 2 \pi i Res_{\eta}(0,0) \]
by Lemma \ref{lem:res}.
Since $\nu(\varphi_{|x=0}) = \nu(\eta_{|x=0})$,
$Res_{\varphi}(0,0) = Res_{\eta}(0,0)$ and
$\varphi_{|x=0}$, $\eta_{|x=0}$ are analytically trivial
then $\varphi_{|x=0}$ and  $\eta_{|x=0}$ are analytically conjugated.

Suppose that  $\sigma$ is orientation-reversing on the parameter space.
We have
\[ 2 \pi i Res_{\varphi}(0,0) = {\mathfrak h}(2 \pi i Res_{\varphi}(0,0)) = -2 \pi i Res_{\eta}(0,0)=
2 \pi i \overline{Res_{\eta}(0,0)} \]
by Lemma \ref{lem:res}.
Since $\nu(\varphi_{|x=0}) = \nu(\eta_{|x=0})$,
$Res_{\varphi}(0,0) = \overline{Res_{\eta}(0,0)}$ and
$\varphi_{|x=0}$, $\eta_{|x=0}$ are analytically trivial
then $\varphi_{|x=0}$, $\eta_{|x=0}$ are anti-holomorphically conjugated.
\end{proof}
\begin{pro}
\label{pro:atu}
Let $\varphi, \eta \in \dif{p1}{2}$ with $N>1$ such that
there exists a homeomorphism $\sigma$ satisfying
$\sigma \circ \varphi = \eta \circ \sigma$.
Suppose that either $\varphi_{|x=0}$ or $\eta_{|x=0}$ is analytically trivial.
Suppose that either $Res_{\varphi}(0,0) \not \in i {\mathbb R}$ or
$Res_{\eta}(0,0) \not \in i {\mathbb R}$.
Then
\begin{itemize}
\item If $\sigma_{|x=0}$ is orientation-preserving
then $\sigma_{|x=0}$ is holomorphic if and only if $Res_{\varphi}(0,0) =Res_{\eta}(0,0)$.
\item If $\sigma_{|x=0}$ is orientation-reversing
then $\sigma_{|x=0}$ is anti-holomorphic if and only if $Res_{\varphi}(0,0) =\overline{Res_{\eta}(0,0)}$.
\item If $Res_{\varphi}(0,0) =Res_{\eta}(0,0) \in {\mathbb R}^{*}$
then $\sigma_{|x=0}$ is holomorphic or anti-holomorphic.
\item If $Res_{\varphi}(0,0) \not \in
\{Res_{\eta}(0,0), \overline{Res_{\eta}(0,0)}\}$ then
$\varphi_{|x=0}$ and $\eta_{|x=0}$ are neither holomorphically nor
anti-holomorphically conjugated. In particular
$\sigma_{|x=0}$ is neither holomorphic
nor anti-holomorphic.
\end{itemize}
Consider a pair of homeomorphisms $\sigma$, $\tilde{\sigma}$ conjugating
$\varphi$, $\eta$ and such that both are orientation-preserving or
orientation-reversing. Then we obtain
$\tilde{\sigma}_{|x=0} = \sigma_{|x=0} \circ \phi$ for some
holomorphic $\phi \in \diff{}{}$ commuting with
$\varphi_{|x=0}$.
\end{pro}
Proposition \ref{pro:atu} implies Proposition \ref{pro:rigi}.
\begin{proof}
The isomorphism ${\mathfrak h}$ satisfies
${\mathfrak h}(2 \pi i Res_{\varphi}(0,0)) =\pm 2 \pi i Res_{\eta}(0,0)$
(Lemma \ref{lem:res}) and
${\mathfrak h}_{|{\mathbb R}} \equiv Id$ (Proposition \ref{pro:rlinear}).
Thus ${\mathfrak h}$ depends only on wether or not $\sigma_{|x=0}$
is orientation-preserving.
We deduce
\[ {\mathfrak h}_{\varphi,\varphi, \sigma^{-1} \circ \tilde{\sigma}} =
{\mathfrak h}_{\varphi,\eta, \sigma}^{-1}\circ
{\mathfrak h}_{\varphi,\eta, \tilde{\sigma}} = Id \]
if $\sigma$, $\tilde{\sigma}$ have the same orientation.
The mapping $(\sigma^{-1} \circ \tilde{\sigma})_{|x=0}$ is holomorphic
and commutes with $\varphi_{|x=0}$.

We denote $R_{\varphi}=Res_{\varphi}(0,0)$ and $R_{\eta}=Res_{\eta}(0,0)$.
Suppose that $\sigma_{|x=0}$ is orientation-preserving.
The equation ${\mathfrak h}(2 \pi i R_{\varphi})=2 \pi i R_{\eta}$
in Lemma \ref{lem:res} implies that ${\mathfrak h} \equiv z$ is equivalent
to $R_{\varphi} = R_{\eta}$. Hence $\sigma_{|x=0}$ is holomorphic
if and only if $R_{\varphi} = R_{\eta}$.

Suppose that $\sigma_{|x=0}$ is orientation-reversing.
The equation ${\mathfrak h}(2 \pi i R_{\varphi})=-2 \pi i R_{\eta}$
in Lemma \ref{lem:res} implies that ${\mathfrak h} \equiv \overline{z}$ is equivalent
to $R_{\varphi} = \overline{R_{\eta}}$. Hence $\sigma_{|x=0}$ is anti-holomorphic
if and only if $R_{\varphi} = \overline{R_{\eta}}$.

The third item is a consequence of the previous ones.
Suppose
$R_{\varphi} \not \in \{R_{\eta}, \overline{R_{\eta}}\}$.
Then $\varphi_{|x=0}$ and $\eta_{|x=0}$ are neither holomorphically nor
anti-holomorphically conjugated.
\end{proof}
\begin{rem}
The results in this section (except the statements involving
orientation-preserving actions on the parameter space) hold true for
higher dimensional unfoldings
$\varphi(x_{1},\hdots,x_{n},y) =(x_{1},\hdots,x_{n}, F(x_{1},\hdots,x_{n},y))$
of tangent to the identity diffeomorphisms with $N>1$.
The idea is that Long Orbits ${\mathcal O}= (\varphi,y_{+},\beta,T)$
satisfy Propositions \ref{pro:tracking}, \ref{pro:lores} and \ref{pro:evol}.
Otherwise there exists a sequence $a_{n} \in \beta$
such that no subsequence satisfies the tracking properties in Proposition \ref{pro:tracking}.
The construction of the dynamical splitting amounts to desingularize the fixed points set
by a finite sequence of blow-ups. Then up to take a subsequence of
$a_{n}$ we can use the machinery in Section \ref{sec:tracking}
to obtain a contradiction.
\end{rem}
\section{Topological conjugacies for unfoldings of the identity map}
\label{sec:infinite}
Let $\varphi, \eta \in \diff{p1}{2}$ such that
there exists a homeomorphism $\sigma$ satisfying
$\sigma \circ \varphi = \eta \circ \sigma$.
Suppose that $m \stackrel{def}{=} m(\varphi) >0$.
The unfoldings of the identity map are easier
to study than the elements of $\dif{p1}{2}$.
We can construct analogues of the Long Orbits
by iterating $O(1/|x|^{m})$ times a diffeomorphism starting at
a point $(x,y_{0})$.
In this way we obtain analogous results to the ones in
Section \ref{sec:finite}.

The number $\tilde{m} \stackrel{def}{=} m(\eta)$ is positive too.
Consider convergent normal forms $X$, $Y$ of $\varphi$ and
$\eta$ respectively. The vector fields $X$, $Y$ can be written
in the form $X=x^{m} X_{0}$ and $Y=x^{\tilde{m}} Y_{0}$
where $\mathrm{Sing} (X_{0})$ and $\mathrm{Sing} (Y_{0})$ do not contain $x=0$.
We obtain
$\Delta_{\varphi} \in (x^{2m})$ and $\Delta_{\eta} \in (x^{2 \tilde{m}})$
(see Definition \ref{def:delta}).

Fix $\mu \in {\mathbb S}^{1}$.
Let $x_{n}$ be a sequence   such that $x_{n} \to 0$
and $x_{n}/|x_{n}| \to \mu$. Consider $s \in {\mathbb R}$ and
$T_{n} = [s/|x_{n}|^{m}]$. We want to study the behavior of
the sequence $\varphi^{T_{n}}(x_{n},y_{0})$. Let $\psi$ be
a Fatou coordinate of $X$. We have
\[
(\psi \circ \varphi^{T_{n}} - (\psi +T_{n}))(x_{n},y_{0}) =
\sum_{j=0}^{T_{n}-1} (\Delta_{\varphi} \circ \varphi^{j})(x_{n},y_{0}) = O(x_{n}^{m}) . \]
It implies
\begin{equation}
\label{equ:short}
\lim_{n \to \infty} \varphi^{T_{n}}(x_{n},y_{0}) = \lim_{n \to \infty}
\mathrm{exp}(T_{n} X) (x_{n},y_{0}) = \mathrm{exp}(s \mu^{m} X_{0})(0,y_{0})
\end{equation}
for any $y_{0} \in B(0,\epsilon)$.

The next two lemmas are intended to show that the conjugation $\sigma$
is well-behaved even if we consider a real blow-up in the parameter space.
\begin{lem}
There is no subsequence of $|\sigma(x_{n})|^{\tilde{m}}/|x_{n}|^{m}$ converging to
$0$ or $\infty$.
\end{lem}
The conjugating mapping $\sigma$ is of the form $\sigma(x,y) =(\sigma_{0}(x), \sigma_{1}(x,y))$.
We define $\sigma(x) = \sigma_{0}(x)$.
\begin{proof}
We denote $a_{n}=|\sigma(x_{n})|^{\tilde{m}}/|x_{n}|^{m}$. Suppose $\lim_{n \to \infty} a_{n}=\infty$.
Up to consider a subsequence we can suppose that
$\sigma(x_{n})/|\sigma(x_{n})|$ converges to $\tilde{\mu} \in {\mathbb S}^{1}$.
Since   $|\sigma(x_{n})|^{\tilde{m}} = a_{n}|x_{n}|^{m}$ then
$\lim_{n \to \infty} a_{n}|x_{n}|^{m}=0$.
We define
\[ T_{n}'= \left[ \frac{s}{a_{n}|x_{n}|^{m}} \right]= \left[ \frac{s}{|\sigma(x_{n})|^{\tilde{m}}} \right]. \]
We deduce
$\lim_{n \to \infty} \varphi^{T_{n}'}(x_{n},y_{0}) = (0,y_{0})$
for any $y_{0} \in B(0,\epsilon)$. The equation
\[ \lim_{n \to \infty} \eta^{T_{n}'}(\sigma(x_{n},y_{0})) = \lim_{n \to \infty}
\mathrm{exp}(T_{n}' Y)(\sigma(x_{n},y_{0})) = \mathrm{exp}(s \tilde{\mu}^{m} Y_{0})(\sigma(0,y_{0}))  \]
is the analogue of Eq. (\ref{equ:short}) for $\eta$.
The mappings $\varphi$ and $\eta$ are conjugated, hence we obtain
$\sigma(0,y) \equiv \mathrm{exp}(s \tilde{\mu}^{\tilde{m}} Y_{0})(\sigma(0,y))$
for any $s \in {\mathbb R}$. This is impossible since $x=0$ is not contained in
$\mathrm{Sing} (Y_{0})$.

The case $\lim_{n \to \infty} a_{n}=0$ is impossible too. The proof is analogous by defining
$T_{n}' = [s/|x_{n}|^{m}]$.
\end{proof}
\begin{lem}
The limits
\[ {\sigma}_{\sharp}(\mu) \stackrel{def}{=}
\lim_{x/|x| \to \mu, \ x \to 0} \frac{|\sigma(x)|^{\tilde{m}/m}}{|x|} \ \ \mathrm{and} \ \
\breve{\sigma}(\mu) \stackrel{def}{=} \lim_{x/|x| \to \mu, \ x \to 0} \frac{\sigma(x)}{|\sigma(x)|}  \]
exist. In particular
${\sigma}_{\sharp}:{\mathbb S}^{1} \to {\mathbb R}^{+}$ and
$\breve{\sigma}:{\mathbb S}^{1} \to {\mathbb S}^{1}$ are continuous.
\end{lem}
\begin{proof}
Consider sequences $x_{n}$, $\tilde{x}_{n}$ such that
$\lim_{n \to \infty} x_{n} = \lim_{n \to \infty} \tilde{x}_{n} =0$,
both sequences $x_{n}/|x_{n}|$, $\tilde{x}_{n}/|\tilde{x}_{n}|$ converge to $\mu$
and  the limits
\[ a_{0} = \lim_{n \to \infty} \frac{|\sigma(x_{n})|^{\tilde{m}}}{|x_{n}|^{m}}, \
\mu_{0} =  \lim_{n \to \infty} \frac{\sigma(x_{n})}{|\sigma(x_{n})|}, \
a_{1} = \lim_{n \to \infty} \frac{|\sigma(\tilde{x}_{n})|^{\tilde{m}}}{|\tilde{x}_{n}|^{m}}, \
\mu_{1} =  \lim_{n \to \infty} \frac{\sigma(\tilde{x}_{n})}{|\sigma(\tilde{x}_{n})|} \]
are well-defined. It suffices to prove that $a_{0}=a_{1}$ and $\mu_{0}=\mu_{1}$.

We define $T_{n}=[s/|x_{n}|^{m}]$ and $\tilde{T}_{n}=[s/|\tilde{x}_{n}|^{m}]$. We obtain
\[ \lim_{n \to \infty} \varphi^{T_{n}}(x_{n},y_{0}) =
\lim_{n \to \infty} \mathrm{exp}(s \mu^{m} X_{0})(0,y_{0})=
\lim_{n \to \infty} \varphi^{\tilde{T}_{n}}(\tilde{x}_{n},y_{0}) \]
for any $y_{0} \in B(0,\epsilon)$ whereas we have
\[ \lim_{n \to \infty} \eta^{T_{n}}(\sigma(x_{n},y_{0})) =
\mathrm{exp}(s a_{0} \mu_{0}^{\tilde{m}} Y_{0})(\sigma(0,y_{0})) \]
and
\[ \lim_{n \to \infty} \eta^{\tilde{T}_{n}}(\sigma(\tilde{x}_{n},y_{0})) =
\mathrm{exp}(s a_{1} \mu_{1}^{\tilde{m}} Y_{0})(\sigma(0,y_{0})) \]
for $y_{0} \in B(0,\epsilon)$ and $s \in {\mathbb R}$.
We deduce $a_{0} \mu_{0}^{\tilde{m}} =a_{1} \mu_{1}^{\tilde{m}}$ and then
$a_{0}=a_{1}$ and $\mu_{0}^{\tilde{m}} =\mu_{1}^{\tilde{m}}$.
Consider a connected set $E$ such that $E_{\pi} \cap (\{0\} \times {\mathbb S}^{1}) = \{(0,\mu)\}$
and containing the sequences $x_{n}$ and $\tilde{x}_{n}$.
Hence $\sigma(E)$ adheres to a connected set of directions contained in the finite set
$\{ \lambda \in {\mathbb S}^{1}: \lambda^{\tilde{m}} = \mu_{0}^{\tilde{m}} \}$.
As a consequence $\sigma(E)$ adheres to a unique direction and $\mu_{0}=\mu_{1}$.
\end{proof}
\begin{lem}
\label{lem:defhmp}
Let $\tilde{\psi}$ be a Fatou coordinate of $Y_{0}$.
There exists a ${\mathbb R}$-linear isomorphism
${\mathfrak h}:{\mathbb C} \to {\mathbb C}$
such that
\[ (\tilde{\psi} \circ \sigma \circ \mathrm{exp}(z X_{0}))(0,y) -
(\tilde{\psi} \circ \sigma)(0,y) \equiv {\mathfrak h}(z) . \]
In particular ${\mathfrak h}$ does not depend on $y$.
The mapping $\sigma$ is orientation-preserving on the parameter space if and
only if the mapping ${\mathfrak h}$ is orientation-preserving.
Moreover $\breve{\sigma}: {\mathbb S}^{1} \to {\mathbb S}^{1}$ is a homeomorphism
and $m=\tilde{m}$.
\end{lem}
\begin{proof}
Fix $z = \lambda |z|$. Consider $\mu \in {\mathbb S}^{1}$ with $\lambda = \mu^{m}$.
We define $T'(x) = [|z|/|x|^{m}]$. We have
\[ \lim_{x \to 0}^{x \in \mu {\mathbb R}^{+}} \varphi^{T'(x)}(x,y_{0}) =
 \lim_{x \to 0}^{x \in \mu {\mathbb R}^{+}} \mathrm{exp}
 \left( \left[\frac{|z|}{|x|^{m}} \right] x^{m} X_{0} \right)(x,y_{0}) =
\mathrm{exp}(z X_{0})(0,y_{0}) \]
for $y_{0} \in B(0,\epsilon)$.
We obtain
\[ \lim_{x \to 0}^{x \in \mu {\mathbb R}^{+}} \eta^{T'(x)}(\sigma(x,y_{0})) =
 \lim_{x \to 0}^{x \in \mu {\mathbb R}^{+}} \mathrm{exp}
\left( \left[ \frac{|z|}{|x|^{m}} \right] \sigma(x)^{\tilde{m}} Y_{0} \right)(\sigma(x,y_{0}))  \]
and then
\[ \lim_{x \to 0}^{x \in \mu {\mathbb R}^{+}} \eta^{T'(x)}(\sigma(x,y_{0})) =
\mathrm{exp}(|z| \sigma_{\sharp}(\mu)^{m} \breve{\sigma}(\mu)^{\tilde{m}} Y_{0})(\sigma(0,y_{0})) \]
for $y_{0} \in B(0,\epsilon)$.
This implies ${\mathfrak h}(z) = |z| \sigma_{\sharp}(\mu)^{m} \breve{\sigma}(\mu)^{\tilde{m}}$.
Clearly the value of ${\mathfrak h}$ does not depend on $y_{0}$.

Analogously as in Proposition \ref{pro:rlinear} we can show that
${\mathfrak h}$ is a linear isomorphism.
The mapping $\breve{\sigma}$ is a homeomorphism that satisfies
${\mathfrak h}(\mu^{m}) = \sigma_{\sharp}(\mu)^{m} \breve{\sigma}(\mu)^{\tilde{m}}$ for
any $\mu \in {\mathbb S}^{1}$. Thus ${\mathfrak h}$ is orientation-preserving if and
only if the action of $\sigma$ on the parameter space is orientation-preserving.
Notice that when $\mu$ turns once around $0$ the path
$\mu \to {\mathfrak h}(\mu^{m})$ turns $m$ times and $\mu \to \breve{\sigma}(\mu)^{\tilde{m}}$
turns $\tilde{m}$ times. We obtain $m=\tilde{m}$.
\end{proof}
\begin{cor}
\label{cor:affm}
Let $\varphi, \eta \in \diff{p1}{2}$ with $m(\varphi)>0$ such that
there exists a homeomorphism $\sigma$ satisfying
$\sigma \circ \varphi = \eta \circ \sigma$.
Then we obtain $m(\varphi)=m(\eta)$.
Moreover
$\sigma_{|x=0}$ is affine in Fatou coordinates.
The mapping $\sigma_{|x=0}$ is orientation-preserving if and only if
the action of $\sigma$ on the parameter space is
orientation-preserving.
\end{cor}
Corollary \ref{cor:affm} is the General Theorem for the case $m(\varphi)>0$.
Let us remark that
$\sigma_{|x=0}$ is real analytic in $x=0$
if $N=0$ and it is real analytic in
$\{x=0\} \setminus \{(0,0)\}$ if $N \geq 1$.
The proof of Corollary \ref{cor:affm}
is analogous to the proof of Corollary \ref{cor:orp}.

The next result completes the proof of Proposition \ref{pro:holpar}.
\begin{cor}
\label{cor:holpar}
Let $\varphi, \eta \in \diff{p1}{2}$ with $m(\varphi)>0$ such that
there exists a homeomorphism $\sigma$ satisfying
$\sigma \circ \varphi = \eta \circ \sigma$.
Suppose that the action of $\sigma$ on the parameter space is
holomorphic (resp. anti-holomorphic). Then
$\sigma_{|x=0}$ is holomorphic (resp.  anti-holomorphic).
\end{cor}
\begin{proof}
Suppose that the action of $\sigma$ on the parameter space is holomorphic.
Hence ${\sigma}_{\sharp}$ is constant and $\breve{\sigma}$ is of the form
$\lambda \mapsto \mu \lambda$ for some $\mu \in {\mathbb S}^{1}$.
We obtain that ${\mathfrak h}$ is of the form $\varsigma z$ for some
$\varsigma \in {\mathbb C}^{*}$. Thus $\sigma_{|x=0}$ is holomorphic.
If the action of $\sigma$ is anti-holomorphic we argue as in
Corollary \ref{cor:holact} to reduce the situation to the previous one.
\end{proof}
We present two examples of conjugating mappings $\sigma$ such that
$\sigma_{|x=0}$ is not necessarily holomorphic or anti-holomorphic.
They correspond to the cases $N=0$ and $N=1$ respectively.
\begin{exa}
Consider $X=x^{m} \partial /\partial y$. Fix a ${\mathbb R}$-linear
mapping ${\mathfrak h}:{\mathbb C} \to {\mathbb C}$.
We define $\sigma_{\sharp}:{\mathbb S}^{1} \to {\mathbb R}^{+}$
and $\breve{\sigma}:{\mathbb S}^{1} \to {\mathbb S}^{1}$ by using
the formula ${\mathfrak h}(\lambda^{m}) = \sigma_{\sharp}(\lambda)^{m} \breve{\sigma}(\lambda)^{m}$
for $\lambda \in {\mathbb S}^{1}$.
We define the homeomorphism
\[ \sigma (r \lambda, y) = (r \sigma_{\sharp}(\lambda)  \breve{\sigma}(\lambda), {\mathfrak h}(y)) .\]
The vector field $X_{|x=\sigma(r \lambda)}$ is
$r^{m} \sigma_{\sharp}(\lambda)^{m} \breve{\sigma}(\lambda)^{m}  \partial / \partial y=
r^{m} {\mathfrak h}(\lambda^{m}) \partial / \partial y$. The real flows of
$X_{|x=r \lambda}$ and $X_{|x=\sigma(r \lambda)}$ are
\[  \theta_{s}(y) = y + s r^{m} \lambda^{m} \ \ \mathrm{and} \ \
\tilde{\theta}_{s}(y) = y + s r^{m} {\mathfrak h}(\lambda^{m}) \]
respectively for $s \in {\mathbb R}$. Clearly the ${\mathbb R}$-linear mapping ${\mathfrak h}$
conjugates them. Thus $\sigma$  commutes with
${\rm exp}(X)$. It is ${\mathbb R}$-linear when $m=1$.
Moreover $\sigma_{|x=0}$ is holomorphic (resp. anti-holomorphic)
if and only if ${\mathfrak h}$ is holomorphic (resp. anti-holomorphic).
\end{exa}
\begin{exa}
Consider $X = x^{m} y \partial / \partial y$ for some $m \in {\mathbb N}$.
We define ${\mathfrak h}(z) = 3z/2 + \overline{z}/2$ and the homeomorphism
\[ \sigma (r \lambda, y) = (r \sigma_{\sharp}(\lambda)  \breve{\sigma}(\lambda), y |y|) \]
by using
the formula ${\mathfrak h}(\lambda^{m}) = \sigma_{\sharp}(\lambda)^{m} \breve{\sigma}(\lambda)^{m}$.
The vector fields $X_{|x=r \lambda}$ and $X_{|x=\sigma(r \lambda)}$
are equal to
$r^{m} \lambda^{m} y \partial / \partial y$ and
${\mathfrak h}(r^{m} \lambda^{m}) y \partial / \partial y$ respectively.
The real flows of
$X_{|x=r \lambda}$ and $X_{|x=\sigma(r \lambda)}$ are
\[  \theta_{s}(y) =  e^{s r^{m} \lambda^{m}} y   \ \ \mathrm{and} \ \
\tilde{\theta}_{s}(y) = e^{s r^{m} {\mathfrak h}(\lambda^{m})} y   \]
respectively for $s \in {\mathbb R}$. We have
\[ (y |y|) \circ \theta_{s}(y) = e^{s r^{m} (\lambda^{m} + Re(\lambda^{m}))} y |y| =
e^{s r^{m} {\mathfrak h}(\lambda^{m})} y |y| =\tilde{\theta}_{s}(y |y|)\]
for $y \in {\mathbb C}$ and $s \in {\mathbb R}$. Thus
$\sigma$ commutes with the real flow of $X$ and then with $\mathrm{exp}(X)$.
The mapping $\sigma_{|x=0}$ is neither holomorphic nor anti-holomorphic.
Moreover $\sigma_{|x=0}$ is real analytic in $\{x=0\} \setminus \{(0,0)\}$
and it is not a $C^{1}$ diffeomorphism in the neighborhood of $0$.
\end{exa}
%
%
%
%
\begin{rem}
Let $\varphi, \eta \in \diff{p1}{n+1}$ with
${\mathfrak l} \stackrel{def}{=} \{x_{1}=\hdots=x_{n}=0\} \subset \mathrm{Fix} (\varphi)$ such that
there exists a homeomorphism $\sigma$ satisfying
$\sigma \circ \varphi = \eta \circ \sigma$.
It is easy to prove analogous results to the ones in this
section if every irreducible component of $\mathrm{Fix} (\varphi)$
containing ${\mathfrak l}$ is a union  of fibers of
the fibration $dx_{1}=\hdots=d x_{n}=0$.
In general we  can reduce the situation to the previous one by
blowing up centers that are union of lines of the form
$\{x_{1}=c_{1}\} \cap \hdots \cap \{x_{n}=c_{n}\}$.
\end{rem}
\section{Building examples}
\label{sec:build}
Let us construct topological conjugations between real flows
of vector fields in $\Xnt$ whose restrictions to $x=0$
are neither holomorphic nor anti-holomorphic
in the case $N>1$, $m=0$.
We choose the case $N=2$ for the sake of simplicity.
This section provides examples for the exceptional
cases covered by Propositions \ref{pro:atui} and \ref{pro:atu}.
\subsection{Description of the construction}
Consider complex numbers $a,b \in {\mathbb C}$
and a ${\mathbb R}$-linear
orientation-preserving
isomorphism ${\mathfrak h}:{\mathbb C} \to {\mathbb C}$ such that ${\mathfrak h}(1)=1$ and
${\mathfrak h}(2 \pi i a) = 2 \pi i b$.
In particular $Re (a)$ and $Re(b)$ have the same sign. We define
\[ X = \frac{y^{2}-x}{1+ay} \frac{\partial}{\partial y} \ \ \mathrm{and} \ \
Y = \frac{y^{2}-x}{1+by} \frac{\partial}{\partial y} . \]
The residue functions of $X$ satisfy
\[ Res_{X}(x,\sqrt{x}) = \frac{1}{2} \left( \frac{1}{\sqrt{x}} + a \right), \ \
Res_{X}(x,-\sqrt{x}) = \frac{1}{2} \left( \frac{-1}{\sqrt{x}} + a \right). \]
Consider a germ of homeomorphism $\tau$ defined in a neighborhood of $0$
in ${\mathbb C}$ such that ${\mathfrak h}(\pi i /\sqrt{x}) = \pi i / \sqrt{\tau(x)}$
for any $x$ defined in a neighborhood of $0$. In this way we obtain
\begin{equation}
\label{equ:hpres}
{\mathfrak h}(2 \pi i Res_{X}(x,\pm \sqrt{x})) = 2 \pi i Res_{Y}(\tau(x),\pm \sqrt{\tau(x)})
\end{equation}
for any $x$ in a neighborhood of $0$. The isomorphism ${\mathfrak h}$ is of the form
${\mathfrak h}(z) = \varsigma_{0} z + \varsigma_{1} \overline{z}$ with $|\varsigma_{0}|>|\varsigma_{1}|$.
We obtain
\begin{equation}
\label{equ:dtq}
\frac{\varsigma_{0}}{\sqrt{x}} - \frac{\varsigma_{1}}{\overline{\sqrt{x}}} \equiv
\frac{1}{\sqrt{\tau(x)}} \implies
\frac{\sqrt{x}}{\sqrt{\tau(x)}} \equiv
\varsigma_{0} - \varsigma_{1} \frac{\sqrt{x}}{\overline{\sqrt{x}}}.
\end{equation}
Fix a point $y_{0} \in B(0,\epsilon) \setminus \{0\}$ close to $0$.
Let $\psi_{X}$, $\psi_{Y}$ be Fatou coordinates of $X$, $Y$ such that
$\psi_{X}(x,y_{0}) \equiv \psi_{Y}(x,y_{0}) \equiv 0$.
We want to define a homeomorphism $\sigma$ conjugating
$\Re (X)$ and $\Re (Y)$ such that
\begin{itemize}
\item $\sigma$ is of the form $\sigma(x,y) =(\tau(x), \sigma_{\natural}(x,y))$.
\item $\sigma(x,y_{0})= (\tau(x),y_{0})$ for any $x$ in a neighborhood of $0$.
\item $\psi_{Y} \circ \sigma \circ \mathrm{exp}(z X) - \psi_{Y} \circ \sigma \equiv {\mathfrak h}(z)$
for any $z \in {\mathbb C}$.
\end{itemize}
Let us remark that the monodromies of $({\mathfrak h} \circ\psi_{X})(x_{0},y)$ and
$\psi_{Y}(\tau(x_{0}), y)$ around
$(x_{0}, \pm \sqrt{x_{0}})$ and $(\tau(x_{0}), \pm \sqrt{\tau(x_{0})})$
respectively coincide by Eq. (\ref{equ:hpres}).
The natural way of defining $\sigma$ is by using the
equation ${\mathfrak h} \circ\psi_{X} \equiv \psi_{Y} \circ \sigma$.
\subsection{The method of the path}
In order to prove the existence of $\sigma$ satisfying
${\mathfrak h} \circ\psi_{X} \equiv \psi_{Y} \circ \sigma$ we apply
the method of the path (see \cite{Rou:ast}, \cite{Mar:ast}).
First we relocate the points in $\mathrm{Sing}(Y)$ by considering
the change of coordinates $y = y \sqrt{\tau(x)/x}$ for the
parameter $\tau(x)$. This corresponds to the change
of coordinates $\tilde{\sigma}(x,y) = (x, y \sqrt{x/\tau^{-1}(x)})$.
We obtain
\[ Y(\tau(x), y) =
\frac{{ \left( \frac{\sqrt{\tau(x)}}{\sqrt{x}}  y \right) }^{2} -
\tau(x)}{1 + b y \sqrt{\tau(x)} / \sqrt{x}}
\frac{\sqrt{x}}{\sqrt{\tau(x)}} \frac{\partial}{\partial y} =
\frac{\sqrt{\tau(x)}}{\sqrt{x}} \frac{y^{2}-x}{1 + b y \sqrt{\tau(x)} / \sqrt{x}}
\frac{\partial}{\partial y} \]
in the new coordinates. Let us point out that the change of coordinates $\tilde{\sigma}$
is not well-defined at $x=0$. It is not very pathological either since
$\sqrt{x/\tau^{-1}(x)}$ is bounded away from $0$ and $\infty$ for
$x$ in a neighborhood of $0$ (see Eq. (\ref{equ:dtq})).


Let us consider the function
\[ \Psi = (1-s) ({\mathfrak h} \circ \psi_{X})(x,y) + s \psi_{Y}
\left( \tau(x), \frac{\sqrt{\tau(x)}}{\sqrt{x}} y \right) = \Psi_{1} + i \Psi_{2}. \]
In general $\Psi$ is neither holomorphic nor anti-holomorphic.
Roughly speaking all the functions $\Psi(s_{0},x,y)$ have the same poles and
the monodromy around those poles does not depend on $s_{0}$.
We are trying to conjugate $\Re (X)$ and $\tilde{\sigma}^{*} (\Re (Y))$. In order
to do this we want to find a continuous vector field $Z$
defined in coordinates $(x,y,s)$ in a neighborhood of $\{(0,0)\} \times [0,1]$
in $(({\mathbb C}^{*} \times {\mathbb C}) \cup \{(0,0)\}) \times {\mathbb C}$
of the form
\[ Z = \frac{\partial}{\partial s} + c_{1}(x,y,s) \frac{\partial}{\partial y_{1}} +
c_{2}(x,y,s) \frac{\partial}{\partial y_{2}} \]
where $y=y_{1}+iy_{2}$. Moreover we ask $Z$ to satisfy
\begin{itemize}
\item $Z(\Psi) \equiv 0$.
\item $c_{j}$ is a continuous function defined in a neighborhood of $\{(0,0)\} \times [0,1]$
in $(({\mathbb C}^{*} \times {\mathbb C}) \cup \{(0,0)\}) \times {\mathbb C}$ and
vanishing at $y^{2}-x=0$ for any $j \in \{1,2\}$.
\end{itemize}
The idea is that for such $Z$ the mapping $\mathrm{exp}(Z)(x,y,0)$
conjugates $\Re(X)$ and $\tilde{\sigma}^{*}(\Re (Y))$
in a neighborhood of $(0,0)$
in $({\mathbb C}^{*} \times {\mathbb C}) \cup \{(0,0)\}$.
We define
\[ \rho =   ({\mathfrak h} \circ \psi_{X})(x,y) -  \psi_{Y}
\left( \tau(x), \frac{\sqrt{\tau(x)}}{\sqrt{x}} y \right)  . \]
The equation $Z(\Psi) \equiv 0$ is equivalent to
\[ \left\{
\begin{array}{ccccl}
c_{1} \frac{\partial \Psi_{1}}{\partial y_{1}} & + & c_{2} \frac{\partial \Psi_{1}}{\partial y_{2}} &
= & Re (\rho) \\
c_{1} \frac{\partial \Psi_{2}}{\partial y_{1}} & + & c_{2} \frac{\partial \Psi_{2}}{\partial y_{2}} &
= & Im (\rho).
\end{array}
\right.  \]
The solutions are
\begin{equation}
\label{equ:c1c2}
c_{1} =
\frac{\left|
\begin{array}{cc}
Re(\rho) & \frac{\partial \Psi_{1}}{\partial y_{2}} \\
Im(\rho) & \frac{\partial \Psi_{2}}{\partial y_{2}}
\end{array}
\right|}{\left|
\begin{array}{cc}
\frac{\partial \Psi_{1}}{\partial y_{1}} & \frac{\partial \Psi_{1}}{\partial y_{2}} \\
\frac{\partial \Psi_{2}}{\partial y_{1}} & \frac{\partial \Psi_{2}}{\partial y_{2}}
\end{array}
\right|}, \ \ \
c_{2} =
\frac{\left|
\begin{array}{cc}
\frac{\partial \Psi_{1}}{\partial y_{1}} & Re(\rho) \\
\frac{\partial \Psi_{2}}{\partial y_{1}} & Im(\rho)
\end{array}
\right|}{\left|
\begin{array}{cc}
\frac{\partial \Psi_{1}}{\partial y_{1}} & \frac{\partial \Psi_{1}}{\partial y_{2}} \\
\frac{\partial \Psi_{2}}{\partial y_{1}} & \frac{\partial \Psi_{2}}{\partial y_{2}}
\end{array}
\right|}.
\end{equation}
The denominator $D$ of the previous expressions satisfies
$D= |\partial \Psi/\partial y|^{2} - |\partial \Psi/\partial \overline{y}|^{2}$.
We have
\[ \left\{
\begin{array}{ccl}
\frac{\partial \Psi}{\partial y} & = &
(1-s) \varsigma_{0} \frac{1+ay}{y^{2}-x} + s \frac{\sqrt{x}}{\sqrt{\tau(x)}}
\frac{1 + b y \sqrt{\tau(x)} / \sqrt{x}}{y^{2}-x} \\
\frac{\partial \Psi}{\partial \overline{y}} & = &
(1-s) \varsigma_{1} \overline{ \left( \frac{1+ay}{y^{2}-x} \right) }  .
\end{array}
\right. \]
\begin{lem}
\label{lem:den}
The function $D |y^{2}-x|^{2}$ is bounded away from $0$ and
$\infty$.
\end{lem}
\begin{proof}
Consider the function $\kappa = (1-s) \varsigma_{0} + s \sqrt{x} / \sqrt{\tau(x)}$.
It suffices to prove that
$\kappa$ is bounded by
below for $x$ in a neighborhood of $0$ and $s$ in a neighborhood of $[0,1]$
and that we have $|\kappa(x,s)|/|(1-s) \varsigma_{1}| \geq C >1$
for some constant $C \in {\mathbb R}^{+}$.
Notice that $\kappa$ is bounded by above. The property
$\sqrt{x}/\sqrt{\tau(x)} \equiv
\varsigma_{0} - \varsigma_{1} \sqrt{x} / \overline{\sqrt{x}}$ implies
\[ |\kappa(x,s)| =
\left|   \varsigma_{0}  - s \varsigma_{1} \frac{\sqrt{x}}{\overline{\sqrt{x}}} \right| \geq
| \varsigma_{0}| - | \varsigma_{1}| >0 \]
for $s \in [0,1]$. We have
\[
|\kappa(x,s)|^{2}  = (1-s)^{2} | \varsigma_{0}|^{2} +
s^{2} {\left| \frac{\sqrt{x}}{\sqrt{\tau(x)}} \right|}^{2}  +
2 s (1-s) Re \left(
 \varsigma_{0} \overline{\frac{\sqrt{x}}{\sqrt{\tau(x)}}}
\right) \]
and
\[ \varsigma_{0} \overline{\frac{\sqrt{x}}{\sqrt{\tau(x)}}} =
\varsigma_{0}
\left( \overline{\varsigma_{0}} - \overline{\varsigma_{1}} \frac{\overline{\sqrt{x}}}{\sqrt{x}} \right)
\implies
Re \left(
 \varsigma_{0} \overline{\frac{\sqrt{x}}{\sqrt{\tau(x)}}}
\right) \geq |\varsigma_{0}|^{2}-|\varsigma_{0}||\varsigma_{1}| >0. \]
We deduce that $|\kappa| \geq |(1-s) \varsigma_{0}| \geq C |(1-s) \varsigma_{1}|$
for some constant $C >1$.
\end{proof}
The next step of the proof is showing that $(y^{2}-x) \rho$ can be extended as a
continuous function vanishing at $y^{2}-x=0$. We define the auxiliary functions
\[ R_{\pm} = \frac{\pm 1}{2 \sqrt{x}}
\left( \varsigma_{0} - \frac{\sqrt{x}}{\sqrt{\tau(x)}} \right) + \frac{1}{2} (\varsigma_{0} a-b)  \]
and
\[ \tilde{\rho} = R_{+}(x) \ln |y - \sqrt{x}|^{2} +  R_{-}(x) \ln |y + \sqrt{x}|^{2} . \]
Notice that ${\mathfrak h}(2 \pi i a) = 2 \pi i b$ implies $b = \varsigma_{0} a - \varsigma_{1} \overline{a}$.
We have
\[ \frac{\partial \tilde{\rho}}{\partial y} =
\frac{R_{+}(x) (y+\sqrt{x}) + R_{-}(x)(y-\sqrt{x}) }{y^{2}-x} =\frac{\partial \rho}{\partial y} \]
and
\[ \frac{\partial \tilde{\rho}}{\partial \overline{y}} =
\frac{(\varsigma_{0} a -b) \overline{y} +\varsigma_{1} }{\overline{y^{2}-x}} =
\varsigma_{1} \overline{\left( \frac{1+ay}{y^{2}-x} \right)}=
\frac{\partial \rho}{\partial y} . \]
\begin{lem}
\label{lem:tilrho}
The function $\rho - \tilde{\rho}$ is a bounded function of $x$.
\end{lem}
\begin{proof}
It is obvious that $\rho - \tilde{\rho}$ is constant in each line $x=x_{0}$.
Since $\rho(x,y_{0})$ is bounded by construction it suffices to show that
$\tilde{\rho}(x,y_{0})$ is bounded in the neighborhood of $0$.
We have
\[ \tilde{\rho}(x,y_{0}) = \frac{1}{2 \sqrt{x}}
\left( \varsigma_{0} - \frac{\sqrt{x}}{\sqrt{\tau(x)}} \right)
\ln {\left| \frac{y_{0} - \sqrt{x}}{y_{0} + \sqrt{x}} \right|}^{2}
+(\varsigma_{0} a -b) \ln |y_{0}^{2}-x| .\]
It suffices to prove that
\[ \hat{\rho}(x,y_{0}) \stackrel{\mathrm{def}}{=} \frac{1}{2 \sqrt{x}}
\left( \varsigma_{0} - \frac{\sqrt{x}}{\sqrt{\tau(x)}} \right)
\ln {\left| \frac{y_{0} - \sqrt{x}}{y_{0} + \sqrt{x}} \right|}^{2} \]
is bounded. We have
\[ \hat{\rho}(x,y_{0})= \frac{1}{2 \sqrt{x}}
\left( \varsigma_{0} - \frac{\sqrt{x}}{\sqrt{\tau(x)}} \right)
\ln \left(
1 - 2 \frac{\sqrt{x} \overline{y_{0}} + \overline{\sqrt{x}} y_{0} }{
|y_{0}|^{2} +\sqrt{x} \overline{y_{0}} + \overline{\sqrt{x}} y_{0} + |x|}
\right) \]
and then
\[ \hat{\rho}(x,y_{0}) \sim -
\left( \varsigma_{0} - \frac{\sqrt{x}}{\sqrt{\tau(x)}} \right)
\left(  \frac{1}{y_{0}} +\frac{1}{\overline{y_{0}}} \frac{\overline{\sqrt{x}}}{\sqrt{x}} \right)  \]
is bounded.
\end{proof}
\begin{lem}
\label{lem:limrho}
We have $\lim_{(x,y) \to (x_{0},y_{0})} \rho(x,y) (y^{2}-x)=0$
for any $(x_{0},y_{0})$ in the curve $y^{2}-x=0$.
\end{lem}
\begin{proof}
It suffices to show the result for $\hat{\rho}$, see the proof of Lemma
\ref{lem:tilrho}.
Suppose $(x_{0},y_{0}) = (0,0)$, otherwise the proof is straightforward.

Suppose that $|y/\sqrt{x}| \leq 8$. We have
\[ |\hat{\rho}(x,y) (y^{2}-x)| \leq
\frac{9}{2} \left| \varsigma_{0} - \frac{\sqrt{x}}{\sqrt{\tau(x)}} \right|
\left(
|y - \sqrt{x}| |\ln |y - \sqrt{x}|^{2}| +
|y + \sqrt{x}| |\ln |y + \sqrt{x}|^{2}| \right) .\]
We deduce $\lim_{|y/\sqrt{x}| \leq 8, \ (x,y) \to (0,0)} \rho(x,y) (y^{2}-x)=0$.

Suppose that $|y/\sqrt{x}| \geq 8$. We have
\[ |\hat{\rho}(x,y) (y^{2}-x)|  \leq \frac{1}{2 \sqrt{|x|}}
\left| \varsigma_{0} - \frac{\sqrt{x}}{\sqrt{\tau(x)}} \right|
|y^{2}-x|
\left| \ln {\left| \frac{y - \sqrt{x}}{y + \sqrt{x}} \right|}^{2} \right| \]
and then
\[ |\hat{\rho}(x,y) (y^{2}-x)| \leq \frac{1}{2 \sqrt{|x|}}
\left| \varsigma_{0} - \frac{\sqrt{x}}{\sqrt{\tau(x)}} \right|
|y^{2}-x| C \frac{\sqrt{|x|}}{|y|} \]
for some $C \in {\mathbb R}^{+}$. We obtain
\[ |\hat{\rho}(x,y) (y^{2}-x)| \leq \frac{C}{2}
\left| \varsigma_{0} - \frac{\sqrt{x}}{\sqrt{\tau(x)}} \right|
|y| \left| 1 - \frac{x}{y^{2}} \right|   \leq
\frac{65 C}{128} |y| \left| \varsigma_{0} - \frac{\sqrt{x}}{\sqrt{\tau(x)}} \right|.  \]
We deduce $\lim_{|y/\sqrt{x}| \geq 8, \ (x,y) \to (0,0)} \rho(x,y) (y^{2}-x)=0$.
\end{proof}
\begin{lem}
The functions $c_{1}$ and $c_{2}$
are continuous. They are defined in a neighborhood of
$\{ (0,0) \} \times [0,1]$
in $(({\mathbb C}^{*} \times {\mathbb C}) \cup \{(0,0)\}) \times {\mathbb C}$ and
vanish at $y^{2}-x=0$.
\end{lem}
\begin{proof}
Let us prove the result for $c_{1}$ without lack of generality.
Equation (\ref{equ:c1c2}) and Lemma \ref{lem:den}
imply that it suffices to show that
\[ |y^{2}-x|^{2}
\left|
\begin{array}{cc}
Re(\rho) & \frac{\partial \Psi_{1}}{\partial y_{2}} \\
Im(\rho) & \frac{\partial \Psi_{2}}{\partial y_{2}}
\end{array}
\right| \]
tends to $0$ when we approach a point of
$y^{2}-x=0$. Since $\partial \Psi_{j} / \partial y_{k} = O(1/(y^{2}-x))$
for all $j, k \in \{1,2\}$ the property is a consequence of Lemma \ref{lem:limrho}.
\end{proof}
\subsection{Consequences of the construction}
We have all the ingredients required to provide examples of exceptional
conjugating mappings that are not holomorphic or anti-holomorphic.
Hence the condition requiring the unperturbed diffeomorphisms to be
non-analytically trivial in the Main Theorem can not be removed.
\begin{pro}
\label{pro:sigma}
The mapping $\sigma_{\flat}(x,y) = {\rm exp}(Z)(x,y,0)$ conjugates
$\Re (X)$ and $\tilde{\sigma}^{*} (\Re (Y))$
in a neighborhood of $(0,0)$
in $({\mathbb C}^{*} \times {\mathbb C}) \cup \{(0,0)\}$.
Moreover the mapping $\sigma = \tilde{\sigma} \circ \sigma_{\flat}$ is
a homeomorphism conjugating $\Re (X)$ and $\Re (Y)$ defined in
a neighborhood of $(0,0)$. Moreover $\sigma$ and $\sigma^{-1}$
are real analytic outside $y^{2}-x=0$.
\end{pro}
\begin{proof}
It is clear that $\sigma_{\flat}$ and its inverse
${\rm exp}(-Z)(x,y,1)$ satisfy the properties by construction.
The mapping $\sigma$ conjugates $\Re (X)$ and $\Re (Y)$.
The remaining issue is the study of the  properties of $\sigma$ in the
neighborhood of the line $x=0$.

Suppose $a \in i {\mathbb R}$. Then we have $a=b$ and we define
${\mathfrak h}_{u}(z)=(1-u) z + u {\mathfrak h}(z)$ and $a_{u}=a$ for $u \in [0,1]$.
If $a \not \in i {\mathbb R}$ the real parts of $a$ and $b$ have
the same sign since ${\mathfrak h}$ is orientation preserving. We
define $a_{u}=(1-u)a +ub$ and ${\mathfrak h}_{u}:{\mathbb C} \to {\mathbb C}$ as the ${\mathbb R}$-linear
mapping such that ${\mathfrak h}_{u}(1)=1$ and
${\mathfrak h}_{u}(2 \pi i a)=2 \pi i a_{u}$ for $u \in [0,1]$.
We define $X_{u}=[(y^{2}-x)/(1+a_{u} y)] \partial / \partial y$
and a Fatou coordinate $\psi_{u}$ of $X_{u}$ such that
$\psi_{u}(x,y_{0}) \equiv 0$ for $u \in [0,1]$.
We denote by $\sigma_{u}$ the diffeomorphism obtained by applying
the previous method to
$X$, $X_{u}$ and ${\mathfrak h}_{u}$.
We obtain $\sigma_{0}=Id$ and $\sigma_{1} = \sigma$.
Since $\sigma_{u}$ depends continuously on $u$ and $\sigma_{u}$
conjugates the functions ${\mathfrak h}_{u} \circ \psi_{X}$ and $\psi_{u}$
we obtain
$\sigma_{u}(x,y_{0}) = (\tau_{u}(x), y_{0})$ for all
$u \in [0,1]$ and $x$ in a neighborhood of $0$.
In particular we deduce $\sigma(x,y_{0}) = (\tau(x), y_{0})$
for any $x$ in a neighborhood of $0$.
The equation ${\mathfrak h} \circ\psi_{X} \equiv \psi_{Y} \circ \sigma$
implies that $\sigma$ and $\sigma^{-1}$
are real analytic outside $y^{2}-x=0$.
\end{proof}
\begin{exa}
Consider $b= a \in i {\mathbb R}$ and a ${\mathbb R}$-linear
orientation-preserving isomorphism
such that ${\mathfrak h}(1)=1$. Then the mapping $\sigma$ provided by
Proposition \ref{pro:sigma} satisfies $\sigma^{*}(\Re (X)) = \Re(X)$
and $\sigma \circ \mathrm{exp}(X) = \mathrm{exp}(X) \circ \sigma$.
Moreover $\sigma$ is holomorphic if and only if
${\mathfrak h} \equiv z$.

If ${\mathfrak h}$ is orientation-reversing consider the conjugation
$\sigma$ associated to $(X,X,\overline{z} \circ {\mathfrak h})$.
Now $\zeta \circ \sigma$ conjugates  $\Re (X)$
and $\Re (Y)$ where $Y= [(y^{2}-x)/(1+\overline{a} y)] \partial / \partial y$
and $\zeta(x,y) = (\overline{x},\overline{y})$.
The mapping ${\mathfrak h}_{{\rm exp}(X), {\rm exp}(Y), \zeta \circ \sigma}$
is equal to ${\mathfrak h}$.
We described both situations in Proposition \ref{pro:atui} where the diffeomorphisms
by restriction to $x=0$ are holomorphically (resp. anti-holomorphically)
conjugated but the conjugation
is not holomorphic (resp. anti-holomorphic) in general.
\end{exa}
\begin{exa}
Consider $a,b \not \in i {\mathbb R}$ such that $Re (a) Re (b) >0$.
Let ${\mathfrak h}$ be the ${\mathbb R}$-linear mapping
such that ${\mathfrak h}(1)=1$ and ${\mathfrak h}(2 \pi i a) = 2 \pi i b$.
Then the mapping $\sigma$ provided by
Proposition \ref{pro:sigma} satisfies $\sigma^{*}(\Re (Y)) = \Re(X)$
and $\sigma \circ \mathrm{exp}(X) = \mathrm{exp}(Y) \circ \sigma$.
Moreover $\sigma$ is holomorphic if and only if
$a=b$.

Let ${\mathfrak h}$ be the ${\mathbb R}$-linear mapping
such that ${\mathfrak h}(1)=1$ and ${\mathfrak h}(2 \pi i a) = -2 \pi i b$.
We obtain $(\overline{z} \circ {\mathfrak h})(2 \pi i a) = 2 \pi i \overline{b}$.
Then the mapping $\sigma$ provided by
Proposition \ref{pro:sigma} and associated to
$X$, $\tilde{Y}=[(y^{2}-x)/(1+\overline{b} y)] \partial / \partial y$
and $\overline{z} \circ {\mathfrak h}$
satisfies $\sigma^{*}(\Re (\tilde{Y})) = \Re(X)$
and $\sigma \circ \mathrm{exp}(X) = \mathrm{exp}(\tilde{Y}) \circ \sigma$.
The homeomorphism $\zeta \circ \sigma$ conjugates $\Re (X)$ and $\Re (Y)$.
Moreover $\zeta \circ \sigma$ is anti-holomorphic if and only if
$a=\overline{b}$.
We described these situations in Proposition \ref{pro:atu}.
\end{exa}

\bibliography{rendu}
\end{document}